\documentclass[english,ruled]{article}
\usepackage[T1]{fontenc}
\usepackage[latin9]{inputenc}
\usepackage{geometry}
\usepackage{algorithm, algpseudocode}
\algnewcommand{\LeftComment}[1]{\Statex \(\triangleright\) \texttt{#1}}
\usepackage{amsmath}
\usepackage{amsthm}
\usepackage{amssymb}
\usepackage{enumitem}

\usepackage[draft]{changes}
\renewcommand{\added}{}

\usepackage{subcaption}
\usepackage{mathtools}
\usepackage{mathabx}
\usepackage{float}
\usepackage{caption}
\usepackage{pifont}
\usepackage{babel}
\usepackage{hyperref}
\hypersetup{
	colorlinks=true,
	linkcolor=red,
	citecolor = blue
}
\usepackage{array}
\usepackage{multirow}

\newcommand\myrel[2]{\mathrel{\overset{\makebox[0pt]{\mbox{\tiny\sffamily #1}}}{#2}}}
\newcommand\mathfontsize[2]{
	\mbox{\fontsize{#1}{1}\selectfont\(#2\)} 
}

\usepackage{titlesec}
\titleformat{\section}%
{\normalfont\bfseries}{\thesection}{1em}{}
\titleformat{\subsection}[runin]
{\normalfont\bfseries}{\thesubsection}{1em}{}

\providecommand{\tabularnewline}{\\}

\makeatletter

\theoremstyle{plain}
\newtheorem*{assumption*}{\protect\assumptionname}

\theoremstyle{remark}
\newtheorem*{remark*}{\protect\remarkname}

\theoremstyle{plain}
\newtheorem{remark}{\protect\remarkname}

\theoremstyle{plain}
\newtheorem{theorem}{\protect\theoremname}%

\theoremstyle{definition}
\newtheorem{definition}{\protect\definitionname}

\theoremstyle{plain}
\newtheorem{assumption}{\protect\assumptionname}

\theoremstyle{plain}
\newtheorem{proposition}{\protect\propositionname}

\theoremstyle{plain}
\newtheorem{lemma}{\protect\lemmaname}

\theoremstyle{plain}
\newtheorem{corollary}{\protect\corollaryname}

\theoremstyle{plain}
\newtheorem{eg}{Example}

\usepackage[short]{optidef}
\usepackage{hyperref}

\makeatother

\usepackage{babel}
\providecommand{\assumptionname}{Assumption}
\providecommand{\corollaryname}{Corollary}
\providecommand{\definitionname}{Definition}
\providecommand{\lemmaname}{Lemma}
\providecommand{\propositionname}{Proposition}
\providecommand{\remarkname}{Remark}
\providecommand{\theoremname}{Theorem}

\global\long\def\tsum{{\textstyle {\sum}}}%

\numberwithin{equation}{section}

\newcommand{\SFO}{\textit{SFO}}

\newcommand{\GD}{\texttt{LCPG}}

\newcommand{\SGD}{\texttt{LCSPG}}

\newcommand{\SQP}{\texttt{SQP}}
\newcommand{\GSQP}{\texttt{GSQP}}
\newcommand{\MBA}{\texttt{MBA}}

\newcommand{\conex}{\texttt{ConEx}}

\newcommand{\SVRG}{\texttt{LCSVRG}}
\global\long\def\trans{{\mathrm T}}

\renewcommand\frac[2]{\tfrac{#1}{#2}}

\begin{document}

\global\long\def\vertiii#1{\left\vert \kern-0.25ex  \left\vert \kern-0.25ex  \left\vert #1\right\vert \kern-0.25ex  \right\vert \kern-0.25ex  \right\vert }%
\global\long\def\matr#1{\bm{#1}}%
\global\long\def\til#1{\tilde{#1}}%
\global\long\def\wtil#1{\widetilde{#1}}%
\global\long\def\wh#1{\widehat{#1}}%
\global\long\def\mcal#1{\mathcal{#1}}%
\global\long\def\mbb#1{\mathbb{#1}}%
\global\long\def\mtt#1{\mathtt{#1}}%
\global\long\def\ttt#1{\texttt{#1}}%
\global\long\def\inner#1#2{\langle#1,#2\rangle}%
\global\long\def\binner#1#2{\big\langle#1,#2\big\rangle}%
\global\long\def\Binner#1#2{\Big\langle#1,#2\Big\rangle}%
\global\long\def\br#1{\left(#1\right)}%
\global\long\def\bignorm#1{\bigl\Vert#1\bigr\Vert}%
\global\long\def\Bignorm#1{\Bigl\Vert#1\Bigr\Vert}%
\global\long\def\setnorm#1{\Vert#1\Vert_{-}}%
\global\long\def\rmn#1#2{\mathbb{R}^{#1\times#2}}%
\global\long\def\deri#1#2{\frac{d#1}{d#2}}%
\global\long\def\pderi#1#2{\frac{\partial#1}{\partial#2}}%
\global\long\def\onebf{\mathbf{1}}%
\global\long\def\zerobf{\mathbf{0}}%

\global\long\def\norm#1{\lVert#1\rVert}%
\global\long\def\bnorm#1{\big\Vert#1\big\Vert}%
\global\long\def\Bnorm#1{\Big\Vert#1\Big\Vert}%

\global\long\def\brbra#1{\big(#1\big)}%
\global\long\def\Brbra#1{\Big(#1\Big)}%
\global\long\def\rbra#1{(#1)}%
\global\long\def\sbra#1{[#1]}%
\global\long\def\bsbra#1{\big[#1\big]}%
\global\long\def\Bsbra#1{\Big[#1\Big]}%
\global\long\def\cbra#1{\{#1\}}%
\global\long\def\bcbra#1{\big\{#1\big\}}%
\global\long\def\Bcbra#1{\Big\{#1\Big\}}%

\global\long\def\grad{\nabla}%
\global\long\def\Expe{\mathbb{E}}%
\global\long\def\rank{\text{rank}}%
\global\long\def\range{\text{range}}%
\global\long\def\diam{\text{diam}}%
\global\long\def\epi{\text{epi }}%
\global\long\def\inte{\operatornamewithlimits{int}}%
\global\long\def\cov{\text{Cov}}%
\global\long\def\argmin{\operatornamewithlimits{argmin}}%
\global\long\def\argmax{\operatornamewithlimits{argmax}}%
\global\long\def\tr{\operatornamewithlimits{tr}}%
\global\long\def\dis{\operatornamewithlimits{dist}}%
\global\long\def\sign{\operatornamewithlimits{sign}}%
\global\long\def\prob{\mathbb{P}}%
\global\long\def\st{\operatornamewithlimits{s.t.}}%
\global\long\def\dom{\text{dom}}%
\global\long\def\diag{\text{diag}}%
\global\long\def\and{\text{and}}%
\global\long\def\st{\text{s.t.}}%
\global\long\def\Var{\operatornamewithlimits{Var}}%
\global\long\def\raw{\rightarrow}%
\global\long\def\law{\leftarrow}%
\global\long\def\Raw{\Rightarrow}%
\global\long\def\Law{\Leftarrow}%
\global\long\def\vep{\varepsilon}%
\global\long\def\dom{\operatornamewithlimits{dom}}%

\global\long\def\Lbf{\mathbf{L}}%

\global\long\def\Cbb{\mathbb{C}}%
\global\long\def\Ebb{\mathbb{E}}%
\global\long\def\Fbb{\mathbb{F}}%
\global\long\def\Nbb{\mathbb{N}}%
\global\long\def\Rbb{\mathbb{R}}%
\global\long\def\extR{\widebar{\mathbb{R}}}%
\global\long\def\Pbb{\mathbb{P}}%
\global\long\def\Acal{\mathcal{A}}%
\global\long\def\Bcal{\mathcal{B}}%
\global\long\def\Ccal{\mathcal{C}}%
\global\long\def\Dcal{\mathcal{D}}%
\global\long\def\Fcal{\mathcal{F}}%
\global\long\def\Gcal{\mathcal{G}}%
\global\long\def\Hcal{\mathcal{H}}%
\global\long\def\Ical{\mathcal{I}}%
\global\long\def\Kcal{\mathcal{K}}%
\global\long\def\Lcal{\mathcal{L}}%
\global\long\def\Mcal{\mathcal{M}}%
\global\long\def\Ncal{\mathcal{N}}%
\global\long\def\Ocal{\mathcal{O}}%
\global\long\def\Pcal{\mathcal{P}}%
\global\long\def\Ucal{\mathcal{U}}%
\global\long\def\Scal{\mathcal{S}}%
\global\long\def\Tcal{\mathcal{T}}%
\global\long\def\Xcal{\mathcal{X}}%
\global\long\def\Ycal{\mathcal{Y}}%
\global\long\def\Ubf{\mathbf{U}}%
\global\long\def\Pbf{\mathbf{P}}%
\global\long\def\Ibf{\mathbf{I}}%
\global\long\def\Ebf{\mathbf{E}}%
\global\long\def\Abs{\boldsymbol{A}}%
\global\long\def\Qbs{\boldsymbol{Q}}%
\global\long\def\Lbs{\boldsymbol{L}}%
\global\long\def\Pbs{\boldsymbol{P}}%
\global\long\def\i{i}%
\global\long\def\Ibb{\mathbb{I}}

\DeclarePairedDelimiterX{\inprod}[2]{\langle}{\rangle}{#1, #2}
\DeclarePairedDelimiter\abs{\lvert}{\rvert}
\DeclarePairedDelimiter{\bracket}{ [ }{ ] }
\DeclarePairedDelimiter{\paran}{(}{)}
\DeclarePairedDelimiter{\braces}{\lbrace}{\rbrace}
\DeclarePairedDelimiterX{\gnorm}[3]{\lVert}{\rVert_{#2}^{#3}}{#1}
\DeclarePairedDelimiter{\floor}{\lfloor}{\rfloor}
\DeclarePairedDelimiter{\ceil}{\lceil}{\rceil}
\setlength{\abovedisplayskip}{2pt}
\setlength{\belowdisplayskip}{2pt}

\allowdisplaybreaks

\title{Level Constrained First Order Methods for %
	Function Constrained Optimization\thanks{DB was partially supported by NSF grant CCF 2245705. QD was partially supported by China NSFC grant 11831002, 72394364. GL was partially supported by AHA grant 23CSA1052735 and USDA grant 2020-67021-31526. Part of the work was done while QD was at Shanghai University of Finance and Economics.}
}
\author{Digvijay Boob \thanks{dboob@smu.edu, Operations Research and Engineering Management, Southern Methodist University.} \qquad 
	Qi Deng \thanks{contact.qdeng@gmail.com, Antai College of Economics \& Management, Shanghai Jiao Tong University.} \qquad 
	Guanghui Lan \thanks{george.lan@isye.gatech.edu, Industrial and Systems Engineering, Georgia Institute of Technology.\\
	Corresponding authors: Digvijay Boob and Qi Deng.}
}
\maketitle
\begin{abstract}
We present a new feasible proximal gradient method for constrained optimization where both the objective and constraint functions are given by summation of a smooth, possibly nonconvex function and a convex simple function. The algorithm converts the original problem into a sequence of convex subproblems. Either exact or approximate solutions of convex subproblems can be computed efficiently in many cases. For the inexact case, computing the solution of the subproblem requires evaluation of at most one gradient/function-value of the original objective and constraint functions. An important feature of the algorithm is the constraint level parameter. By carefully increasing this level for each subproblem, we provide a simple solution to overcome the challenge of bounding the Lagrangian multipliers, and show that the algorithm follows a strictly feasible solution path till convergence to the stationary point. Finally, we develop a simple, proximal gradient descent type analysis, showing that the complexity bound of this new algorithm is comparable to gradient descent for the unconstrained setting which is new in the literature. 
Exploiting this new design and analysis technique, we extend our algorithms to some more challenging constrained optimization problems where 1) the objective is a stochastic or finite-sum function, and 2) structured nonsmooth functions replace smooth components of both objective and constraint functions. Complexity results for these problems also seem to be new in the literature. We also show that our method can be applied for convex function constrained problems where we show complexities similar to the proximal gradient method.
\end{abstract}

\section{Introduction}

In this paper, we study the following constrained optimization problem
\begin{equation} \label{prob:main}
	\begin{aligned}
		\min_{x\in\Rbb^d} &\quad  \psi_0(x)\coloneq f_0(x)+\chi_0(x) \\
		\st & \quad  \psi_i(x) \coloneq f_i(x)+\chi_i(x) \le \eta_i,\quad i=1,\dots,m.
	\end{aligned}
\end{equation}
where $\psi_{i}(x)$ is a composite function that sums up  functions $f_i(x)$ and $\chi_{i}(x)$. Here, $f_i,  i ={0}, 1, \dots, m$, are smooth functions,  $\chi_0(x)$ is  a proper, convex, lower-semicontinuous (lsc) function and $\chi_{i}(x), i=1,\dots,m$, are convex continuous functions over the domain of $\chi_{0}$ (i.e. $\dom{\chi_{0}}$). We  consider that $\chi_{i}, i =0, \dots, m$ are `simple' functions, namely,  a feasible optimization problem of the form below
\begin{equation}\label{eq:subproblem_conceptual}
	\min_{x \in \Rbb^d}\ \bcbra{\gnorm{x-a^0}{}{2} + \chi_{0}(x): \gnorm{x-a^i}{}{2} + \chi_{i}(x) \le b^i, i = 1, \dots, m.}
\end{equation} can be solved to efficiently obtain either an exact solution or an inexact solution of desired accuracy. Note that if $\chi_{i} =0 , i= 1, \dots, m$, then \eqref{eq:subproblem_conceptual} becomes a proximal operator for function $\chi_{0}$ on the intersection of balls. 
If we further assume  $\chi_{0} = 0$, then \eqref{eq:subproblem_conceptual} is a special type of quadratically constrained quadratic programming (QCQP) that can be solved efficiently because all the Hessians are identity matrices. 
In addition, we consider the case where constraints $f_i,  i =1, \dots, m$, are 
	structured nonsmooth functions that can be approximated by smooth functions (also called smoothable functions). 
Note that problem \eqref{prob:main} covers a variety of convex and nonconvex function constrained optimization depending on the assumptions of $f_i,  i = 0, \dots, m$. 

Nonlinear optimization with function constraints is a classical topic in continuous optimization.  
While the earlier study focused on asymptotic performance,  recent work has put more emphasis on   the complexity analysis of algorithms, mainly driven by the growing interest in large-scale optimization and machine learning. For most of our discussion on the complexity analysis, we generally require convergence to an $\epsilon$-approximate KKT point (c.f. Definition~\ref{def:type1-KKT}).
Penalty methods \cite{cartis2014on,WangMaYuan17-1,lin2019inexact}, including augmented Lagrangian methods \cite{LanMon13-1,LanMon16-1,Xu2019-1,li2021augmented},  is one popular approach for constrained optimization. 
In \cite{CartisGouldToint11-1}, Cartis et al. presented an exact penalty method by minimizing a sequence of convex composition functions.
When the penalty weight is bounded, this method solves $\Ocal(1/\epsilon)$ trust region subproblems. If the penalty weight is unbounded, the complexity is of $\Ocal(1/\epsilon^{2.5})$ to reach an $\epsilon$-KKT point. In a subsequent work \cite{cartis2014on}, the authors provided a target following method that achieves the complexity of $\Ocal(1/\epsilon)${, regardless of the growth of the penalty parameter}. In \cite{WangMaYuan17-1}, Wang et al. extended the penalty method for solving constrained problems where the objective takes the expectation form. 
Sequential quadratic programming (SQP) is another important approach for constrained optimization. Typically, SQP  involves linearization of the constraints, quadratic approximation of the objective, and possibly some trust region constraint for {the} convergence guarantee \cite{burke1989sequential,burke1989robust}. 
The recent work \cite{Facchinei19} established a unified convergence analysis  of SQP (GSQP) in more general settings  where the feasibility and constraint qualification may or may not hold. 
Different from the standard SQP, the Moving Balls Approximation (\MBA) method \cite{Auslender2010moving} follows a feasible solution path and transforms the initial problem into a diagonal QCQP. 
A subsequent work \cite{Bolte2016majorization} presented a unified analysis of {\MBA} and other variants of SQP methods. Under the assumption of Kurdyka-{\L}ojasiewicz (KL) property, they establish  the global convergence rates which depend on the {\L}ojasiewicz exponent.

Despite much progress in prior works, there are some significant remaining issues. Specifically, most of the analysis is carried out for only smooth optimization and requires that the exact optimal solution of the convex subproblem is readily available.
Unfortunately, both assumptions can be unrealistic in many large-scale applications.
To overcome these issues, \cite{boob2019proximal,ma2020quadratically,lin2019inexact} presented some new proximal point algorithms that iteratively solve strongly convex proximal subproblems inexactly using first-order methods. A significant computational advantage is that first-order methods only need to compute a relatively easy proximal gradient mapping in each iteration. 
 In particular, \cite{boob2019proximal} proposed to solve the proximal point subproblem by a new first-order primal-dual  method called {\conex}.  Under some strict feasibility assumption, they derived the total complexities of the overall algorithm for which the objective and constraints can be either stochastic or deterministic, and either nonsmooth or smooth. 
 Notably, for nonconvex and smooth constrained problems, inexact proximal point~\cite{boob2019proximal}  requires $\Ocal(1/\epsilon^{1.5})$ function/gradient evaluations. A similar complexity bound is obtained by  the  proximal point penalty method \cite{lin2019inexact} when a feasible point is available. 
 Nevertheless, at this point, it may  be difficult to directly compare the efficiency of the proximal point with the earlier approach, given that  very different  oracles are employed in each method.
  The inexact proximal point method appears to be less efficient in terms of gradient and function value computations since  first-order penalty method~\cite{cartis2014on} and a variant of SQP \cite{Facchinei19} (where the surrogate is formed by first-order approximation) has an  $\Ocal(1/\epsilon)$ complexity bound.
Nevertheless, it might be more efficient if the corresponding proximal mapping is much easier to solve than the subproblems in penalty or SQP methods.

In this paper, we attempt to alleviate some of the aforementioned issues in solving nonconvex constrained optimization. 
Our main contribution is {the development of}  a novel {\bf L}evel {\bf C}onstrained {\bf P}roximal {\bf G}radient (\GD) method for constrained optimization,
based on the following key ideas. 

First, we convert the original  problem  \eqref{prob:main} into a sequence of simple convex subproblems of the form \eqref{eq:subproblem_conceptual} for which  an exact or an approximate solution can be computed efficiently. 
In particular, solving the subproblem requires at most one gradient and function value computation for $f_i,  i =0, \dots, m$. This phenomenon is quite similar to simple single-loop methods even though \GD~method can be multi-loop since we allow for an inexact solution of \eqref{eq:subproblem_conceptual} %
using some kind of iterative scheme.  %

Second,  starting from a strictly feasible initial point and carefully controlling the feasibility levels of the subproblem constraints,  we ensure that {\GD} follows a strictly feasible solution path. This also allows us to deal with nonsmooth constraints where $\chi_i$ is not necessarily $0$ and further extends {\GD} to the inexact case where the subproblem admits an approximate solution. Even though subtle, the level-control design is crucial in bounding the Lagrange multipliers under the well-known Mangasarian-Fromovitz constraint qualification (MFCQ)~\cite{manga67,boob2019proximal}. Subsequently, we also show asymptotic convergence of the \GD~method.

Third, we offer a new insight into the complexity analysis of {\GD} as a gradient descent type method, which could be of independent interest. 
When the objective and constraints are nonconvex composite,  we aim to find a first-order $\epsilon$-KKT point (c.f. Definition \ref{def:type1-KKT}) under the aforementioned MFCQ assumption. 
We can show that \GD~method converges in $O(1/\epsilon)$ iterations. %
Furthermore, each subproblem requires at most one function-value and gradient computation. %
The net outcome is that {\em gradient complexity} of our method is of $O(1/\epsilon)$. Notice that the number of iterations required by the proximal point method under MFCQ  is also $O(1/\epsilon)$ (see \cite[Theorem 5]{boob2019proximal}). However, each iteration of this method requires $O(1/\epsilon^{0.5})$ gradient computation, and hence its total gradient complexity can be bounded by $O(1/\epsilon^{1.5})$. This is much worse than \GD~method. We compare with some significant lines of work in Table~\ref{table:complexity}.

Exploiting the intrinsic connection between \GD~and proximal gradient (without function constraints), we extend \GD~to a variety of cases.
1) We can show a similar $O(1/\epsilon)$ gradient complexity for an inexact \GD~method for which the subproblem is solved to a pre-specified accuracy. 
If we  assume  $\chi_{i} = 0$, then the corresponding subproblem  \eqref{eq:subproblem_conceptual} (i.e. diagonal QCQP) can be efficiently solved by a customized interior point method in logarithmic time. 
In the more general setting when $\chi_{i} \neq 0$, we propose to solve \eqref{eq:subproblem_conceptual} by the first-order method {\conex}, which has very cheap iterations.
2) We also extend \GD~method to stochastic (\SGD) and variance-reduced (\SVRG) variants when $f_0$ is either stochastic or finite-sum function, respectively. \SGD~and \SVRG~ require $O(1/\epsilon^2)$ (similar to SGD \cite{saeed-lan-nonconvex-2013}) and $O(\sqrt{n}/\epsilon)$ (similar to SVRG \cite{Reddi2016proximal}) stochastic gradients, respectively, where $n$ is the number of components in the finite-sum objective. The complexity of variants of \GD~method for stochastic cases can also be seen in Table~\ref{tab:count-sfo}. 3) We consider the case when function $f_i, i =0,1, \dots, m$ are nondifferentiable but contain a smooth saddle structure (referred to as {\em structured nonsmooth}). We extend \GD~method for such nonsmooth nonconvex function constrained problem using Nesterov's smoothing scheme \cite{Nesterov2005smooth}. In this case, \GD~method requires $O(1/\epsilon^2)$ gradients.

We show that the GD-type analysis of the \GD~method can be extended to the convex case. In particular, when the objective and constraint functions are convex, we %
show that \GD~method requires $O(1/\epsilon)$ gradient computations for smooth and composite constrained problems, and this complexity improves to  $O(\log{({1}/{\epsilon})})$  when the objective is smooth and strongly-convex. %
Furthermore, we develop the complexity of inexact variants of \GD~method by leveraging the analysis of gradient descent with inexact projection oracles \cite{Schmidt2011convergence}. Inexact \GD~method maintains the gradient complexity of $O(1/\epsilon)$ and $O(\log(1/\epsilon))$ for convex and strongly convex problems, respectively.

Throughout our analysis, 
we require that the Lagrange multipliers for the convex subproblems of type \eqref{eq:subproblem_conceptual} are bounded. This problem is addressed in different ways in arguably all works in the literature. In this paper, we show that under the assumption of MFCQ,  Lagrange multipliers associated with the sequence of subproblems remain bounded by a quantity specified as $B$. Even then, the value of $B$ cannot be estimated a priori. Fortunately, this bound is not needed in the implementation of our methods. However, it plays a role in complexity analysis. Hence, our comparison with the existing complexity literature (e.g., proximal point method of \cite{boob2019proximal}) is valid when bound $B$ on the sequence of Lagrange multipliers largely depends on the problem itself and not on the sequence of subproblems. One can easily see that such uniform bounds on Lagrange multipliers hold under the strong feasibility constraint qualification~\cite{boob2019proximal}, a similar uniform Slater's condition~\cite{ma2020quadratically} or for nonsmooth nonconvex relaxation in the application of sparsity constrained optimization \cite{Boob2020feasible}. The problem of comparing bounds $B$ on Lagrange multipliers requires getting into specific applications, which is not the purpose of this paper. Hence, throughout our comparison with existing literature, we assume that bound $B$ for different methods is of a similar order.

\providecommand{\tabularnewline}{\\}
\begin{table}[t]
\centering
\renewcommand{\arraystretch}{1}
\begin{tabular}{c|c|c|c|c|c|c}
\hline 
\multirow{3}{*}{Algorithms} & \multirow{3}{1.7cm}{\centering Function\\ type} & \multirow{3}{*}{Composite?} & \multirow{3}{*}{Inexact?} & \multirow{3}{*}{CQ} & \multirow{3}{*}{Criteria} & \multirow{3}{1.7cm}{Gradient\\ complexity$^\dagger$}\tabularnewline
 &  &  &  &  & &\tabularnewline
 &  &  &  &  & &\tabularnewline
\hline 
\MBA~\cite{Auslender2010moving} & \textbf{str cvx} & no &no & Feasibility+MFCQ & Opt. gap & $\mathcal{O}(\log(1/\epsilon))$ \tabularnewline
\hline 
{\GSQP}~\cite{Facchinei19} & \textbf{noncvx} & no &no & Extended MFCQ & KKT & $\mathcal{O}(1/\epsilon)$\\
\hline 
\texttt{IPP} & \multirow{2}{*}{\textbf{noncvx}} & \multirow{2}{*}{yes} & \multirow{2}{*}{yes} & \multirow{2}{*}{Strong feasibility} & \multirow{2}{*}{KKT} & \multirow{2}{*}{$\mathcal{O}(1/\epsilon^{1.5})$}\tabularnewline
+{\conex}~\cite{boob2019proximal}& &  &  &  &  & \tabularnewline
\hline 
\multirow{4}{*}{\texttt{IPPP}~\cite{lin2019inexact}} & \multirow{4}{*}{\textbf{noncvx}} & \multirow{4}{*}{no} & \multirow{4}{*}{yes} &\multirow{2}{2.8cm}{\centering Feasibility \\ + Nonsingularity}   & \multirow{4}{*}{KKT} & \multirow{2}{*}{$\mathcal{O}(1/\epsilon^{1.5})$}\tabularnewline
 &  &  & & &  &  \tabularnewline
\cline{5-5} \cline{7-7} 
&  &  & & \multirow{2}{*}{Nonsingularity} &  & \multirow{2}{*}{$\mathcal{O}(1/\epsilon^{2})$}\tabularnewline
& & & & & &\tabularnewline
\hline 
\multirow{2}{*}{\texttt{IQRC}\cite{ma2020quadratically}} &\multirow{2}{*}{\bf noncvx} & \multirow{2}{*}{yes*} & \multirow{2}{*}{yes} & \multirow{2}{*}{Uniform Slater} &\multirow{2}{*}{KKT} & \multirow{2}{*}{$\mathcal{O}(1/\epsilon^{1.5})$} \tabularnewline
& & & & & &\tabularnewline
\hline%
\multirow{3}{1.8cm}{\centering {\GD}\\ (this work)} & \textbf{cvx} & \multirow{3}{*}{yes} & \multirow{3}{*}{yes} & \multirow{3}{*}{Slater+MFCQ} & Opt. gap & $\mathcal{O}(1/\epsilon)$\tabularnewline
\cline{2-2} \cline{6-7} %
& \textbf{str cvx} &  &  & & Opt. gap & $\mathcal{O}(\log(1/\epsilon))$\tabularnewline
\cline{2-2} \cline{6-7} %
& \textbf{noncvx} &  &  & & KKT & $\mathcal{O}(1/\epsilon)$\tabularnewline
\hline 
\end{tabular}
\caption{Comparison of algorithm function/gradient evaluation complexities. {\bf cvx:} convex, {\bf str cvx:} strongly convex, and {\bf noncvx:} nonconvex. \\
	For convex problems, we consider the complexity to reach a feasible solution with  $O(\epsilon)$-optimality gap. 
	For nonconvex problems, we consider the complexity to reach an approximate KKT solution that satisfies $\Vert\partial\mathcal{L}\Vert_{-}^{2}\le\epsilon$.
	Note that different works have quite different error measurements of the complementary slackness. For example, in our translation, \cite{Facchinei19,lin2019inexact} requires an $\mathcal{O}(\sqrt{\epsilon})$ error on the complementary slackness and feasibility. Our measure requires $0$-feasibility error and $\mathcal{O}(\epsilon)$-complementary slackness error.\\
	* \texttt{IQRC} does not explicitly discuss the composite case. Their subproblem oracle can be upgraded to handle proximal cases relatively easily.\\
    $\dagger$ Different methods have different costs for solving the subproblem. Some methods require explicit gradient computations for solving the subproblems and hence, are expected to be quite computationally costly. Some methods (including ours) have simple subproblems. See Remark \ref{rem:grad_compx_vs_comp_compx} for a detailed discussion. Hence, comparing total computational complexity is not possible. We instead compare gradient complexities to provide a realistic estimate of the computational effort of each of these methods.
    }
\label{table:complexity}
\end{table}

\begin{table}
	\centering{}%
	\begin{tabular}{c|c|c|c}
		\hline 
		& {\GD} & {\SGD} & {\SVRG}\tabularnewline
		\hline 
		Complexity & $\Ocal(n\vep^{-1})$ & $\Ocal(\vep^{-2})$ & $\Ocal\big(n^{1/2}\vep^{-1}\big)$\tabularnewline
		\hline 
	\end{tabular}
	\caption{Total number of stochastic gradient evaluations to obtain $\protect\Ocal\big(\vep,\vep\big)$ randomized KKT points in the finite sum problem.}
	\label{tab:count-sfo}
\end{table}

\paragraph{Comparison with MBA method} 
Auslender et al. \cite{Auslender2010moving} provided a Moving Balls Approximation (\MBA) method for smooth constrained problems, i.e. $\chi_i(x), i=0, \dots, m$, are not present. They use Lipschitz continuity of constraint gradients along with MFCQ to ensure that the subproblems satisfy Slater's conditions (see \cite[Proposition 2.1(iii)]{Auslender2010moving}). A similar result is also used in \cite{Zhaosong2021} where they provide a line-search version of {\MBA} for functions satisfying certain KL properties. The {\MBA} method was studied for semi-algebraic functions in \cite{Bolte2016majorization} where they used the KL-property of semi-algebraic functions. The work \cite{Auslender2010moving} also provides the complexity guarantee for constrained programs with a smooth and strongly convex objective. Our results differ from the past studies in the following several aspects. First, we do not assume any KL property on the class of functions, hence making the method applicable to a wider class of problems. Second, we show complexity analysis for a variety of cases, e.g., stochastic, finite-sum, or structured nonsmooth cases. Note that complexity results are not known for the \MBA~type method even for the purely smooth problem. Third, we show  complexity results for both convex and strongly convex cases which strictly subsumes the results in \cite{Auslender2010moving}. Fourth, it should be noted that \cite{Auslender2010moving} also considered the efficiency of solving subproblems. They proposed an accelerated gradient method that obtains $\Ocal(1/\sqrt{\epsilon})$ complexity for solving the dual of the QCQP subproblem. However, it is unclear what accuracy is enough for ensuring asymptotic convergence of the whole algorithm. %
\paragraph{Comparison with generalized SQP}
The work \cite{Facchinei19} developed the first complexity analysis of the generalized SQP (GSQP) method by using a novel ghost penalty approach. 
Different from our  feasible method, they consider a general setting where the feasibility and constraint qualification may or may not hold. They show that SQP-type methods have an $\Ocal(1/\epsilon^2)$ complexity for reaching some $\epsilon$-approximate generalized stationary point. 
Under an extended MFCQ condition, they established an improved complexity $\Ocal(1/\epsilon)$ for reaching  the scaled-KKT point, which matches our complexity result under a similar MFCQ assumption. 
  Notably, both their analysis and ours rely on MFCQ to show that the global upper bound (constant $B$) on the  multipliers of the subproblems exists. 
However, to obtain the best $\Ocal(1/\epsilon)$ complexity, GSQP  explicitly relies on the value of such unknown upper bound to determine the stepsize, which appears to be rather challenging in practical use.
In contrast, our algorithm does not involve  the constant $B$ in the algorithm implementation; we only require the Lipschitz constants of the gradients, which is standard for gradient descent methods.

\paragraph{Outline} This paper is organized as follows: 
Section~\ref{sec:notation} describes notations and assumptions. It also provides various definitions used throughout the paper. 
Section~\ref{sec:Gradient-Descent} presents the \GD~method which uses exact solutions of subproblems. It also establishes asymptotic convergence and convergence rate results. 
Section~\ref{subsec:Stochastic-Gradient-Descent} and Section~\ref{subsec:Variance-Reduced-Gradient} provides the \SGD~and \SVRG~method for stochastic and finite-sum problems, respectively. Section \ref{sec:nonsmooth} shows the extension of {\GD} for nonsmooth nonconvex function constraints. 
Section~\ref{sec:inexact-GD} introduces the inexact \GD~method and provides its complexity analysis when the subproblems are inexactly solved by an interior point method or first-order method. 
Finally, Section \ref{sec:convex-programming} extends \GD~method for convex optimization problems and establishes its complexity for both strongly convex and convex problems.

\section{Notations and assumptions}\label{sec:notation}

\paragraph{Notations.}
$\Rbb^n_+$ stands for the non-negative orthant in $\Rbb^n$.
We use $\gnorm{\cdot}{}{}$ to express the Euclidean norm. For a set $\Xcal$,
we define $\gnorm{\Xcal}{-}{}=\text{dist}(0,\Xcal)=\inf\big\{\gnorm{x}{}{},x\in\Xcal\big\}$. If $\Xcal$ is a convex set then we denote its normal cone at $x$ as $N_\Xcal(x)$. Furthermore, we denote the dual cone of the normal cone at $x$ as $N^*_\Xcal(x)$.
Let $e$ be a vector full of elements one. For simplicity, we denote $[m]=\{1,2,\ldots, m\}$, 
$f(x)=[f_{1}(x),\ldots,f_{m}(x)]^{\trans}$, $\chi(x)=[\chi_{1}(x),\ldots,\chi_{m}(x)]^{\trans}$,
and $\psi(x)=[\psi_{1}(x),\psi_{2}(x),\ldots,\psi_{m}(x)]^{\trans}$.
For vectors $x, y\in \Rbb^m$, $x\le y$ is understood as $x_i\le y_i$ for $i\in[m]$.

\begin{assumption}[General]\label{ass:obj-cst}
We assume that the optimal value of problem~\eqref{prob:main} is finite: $\psi_{0}^{*}>-\infty$. Furthermore, the objective and constraint functions have the following properties.
\begin{enumerate}[label=\textbf{\arabic*:},ref=1.\arabic*]
\item \label{assum:chi} $\chi_{0}$ is a proper, convex, and lower semi-continuous (lsc) function. %
Moreover, we assume that for all $i = 1, \dots, m$, the function $\chi_{i}(x)$ is convex
continuous over $\dom\chi_0$.
\item \label{assum:f0-lip}$f_{i}(x)$ is $L_{i}$-Lipschitz smooth on $\dom \chi_0$: $\gnorm{\nabla f_{i}(x)-\nabla f_{i}(y)}{}{}\le L_{i}\gnorm{x-y}{}{}$ for any $x,y\in\dom\chi_0$. For brevity, we  denote $L=[L_{1},\ldots,L_{m}]^{\trans}$. 

\item \label{assum:bounded_set} The feasible set for \eqref{prob:main}, i.e., $\bigcap_{i\in[m]}\{x:\psi_i(x) \le \eta_i\} \cap \dom \chi_{0}$ is nonempty and  compact\footnote{This assumption is used to ensure the limit point of a certain sequence in the analysis. There are other conditions that can ensure the existence of a limit point as we will discuss in Remark \ref{rem:limit-point_x^k}.}.
\end{enumerate}
\end{assumption}

The Lagrangian function of problem~(\ref{prob:main}) is denoted by 
\begin{equation}
\Lcal(x,\lambda)=\psi_{0}(x)+\tsum_{i=1}^{m}\lambda_{i}\sbra{\psi_{i}(x)-\eta_{i}}.\label{eq:lagrange-0}
\end{equation}
For functions $\psi_i$, we denote its subdifferential %
as 
\[\partial \psi_{i}(x) = \{\grad f_{i}(x) \}+ \partial \chi_{i}(x),i = 0, \dots, m,\]
where the sum is in the Minkowski sense. {Note that this definition of the subdifferential for nonconvex functions was first proposed in \cite{boob2019proximal}. Moreover, $\partial \psi_i = \{\grad f_i \}$ when $\psi_i$ is a ``purely'' smooth nonconvex function and $\partial \psi_i = \partial \chi_i$ when $\psi_i$ is a nonsmooth convex function. Hence, it is a valid definition for the subdifferential of a nonconvex function. 
Below, we define the KKT condition using the above subdifferential.}
\begin{definition}
[KKT condition] \label{def:kkt-cond}We say  $x\in\dom {\chi_0}$ is a  KKT point of problem~\eqref{prob:main} if $x$ is feasible and there exists a vector $\lambda\in\Rbb_{+}^{m}$
such that
\begin{equation}\label{eq:kkt-1}
0  \in\partial_{x}\Lcal(x,\lambda), \quad 0  =\tsum_{i=1}^m\lambda_{i}\big[\psi_{i}(x)-\eta_{i}\big].
\end{equation}
The values $\{\lambda_{i}\}$ are called Lagrange multipliers. 
\end{definition}

It is known that the KKT condition is necessary for optimality under the assumption of certain constraint qualifications (c.f. \cite{bertsekasnonlinear}). Our result will be based on a variant of the Mangasarian-Fromovitz constraint qualification, which is formally given below.
\begin{definition}
[MFCQ]We say that a point $x$ satisfies the Mangasarian-Fromovitz constraint qualification for \eqref{prob:main} if there exists a vector $z \in -N^*_{\dom{\chi_{0}}}(x)$ such that
\begin{equation}
\max_{v\in\partial\psi_{i}(x)}\langle v,z\rangle<0,\quad i\in\Acal(x), \label{eq:mfcq}
\end{equation}
where $\Acal(x)=\{i:1\le i\le m,\psi_{i}(x)=\eta_{i}\}$.
\end{definition}

\begin{proposition}
[Necessary condition]Let $x$ be a local optimal solution of Problem~\eqref{prob:main}. If
it satisfies MFCQ~\eqref{eq:mfcq}, then there is a vector $\lambda\in\Rbb_{+}^{m}$
such that the KKT condition~\eqref{def:kkt-cond} holds.
\end{proposition}

Next, we introduce some optimality measures before formally presenting any algorithms. It is   natural to characterize algorithm performance by measuring the error of satisfying the KKT condition. Towards this goal, we have the following definition.
\begin{definition}\label{def:type1-KKT}
We say that  $x$ is an $\epsilon$ type-I (approximate) KKT
point if it is feasible (i.e. $\psi(x)\le\eta$), and there exists a vector $\lambda\in\Rbb_{+}^{m}$ satisfying the following
conditions:
\begin{align*}
\gnorm{ \partial_{x}\Lcal(x,\lambda)}{-}{2} & \le {\epsilon}\\
-\tsum_{i=1}^m \lambda_{i}\sbra{\psi_{i}(x)-\eta_{i}} & \le {\epsilon}.%
\end{align*}
Moreover,  $x$ is a randomized ${\epsilon}$ type-I KKT point
if both $x$ and $\lambda$ are {feasible} randomized primal-dual solutions that satisfy 
\begin{align*}
	\Ebb\bracket{\gnorm{ \partial_{x}\Lcal(x,\lambda)}{-}{2}} & \le {\epsilon}\\
	\Ebb\bracket{-\tsum_{i=1}^m \lambda_{i}\sbra{\psi_{i}(x)-\eta_{i}} }& \le {\epsilon},
\end{align*}
where the expectation is taken over the randomness of  $x$ and $\lambda$.
\end{definition}
Besides the above definition, we invoke a second optimality measure which will help analyze the performance of a proximal algorithm (see, for example, \cite{boob2019proximal}).  
Therein, it is arguably  more convenient to approach the proximity of some nearly stationary points. 
\begin{definition}\label{def:type2-KKT}
	We say that $x$ is a ${(\epsilon, \nu)}$  type-II  KKT point if there exists an ${\epsilon}$ type-I KKT point $\hat{x}$ and $\gnorm{x-\hat{x}}{}{2}\le \nu$. Similarly, $x$ is a randomized ${(\epsilon, \nu)}$  type-II KKT point if $\hat{x}$ is  a random vector and $\mathbb{E}[\gnorm{x-\hat{x}}{}2 ]\le {\nu}$.
\end{definition}

\section{A proximal gradient method\label{sec:Gradient-Descent}}

We present the level constrained proximal gradient  ({\GD}) method in Algorithm~\ref{alg:lcgd}. 
The main idea of this algorithm is to turn the original nonconvex problem into a sequence of relatively easier subproblems that involve some convex surrogate functions $\psi^k_i(x)$ ($0\le i\le m$) and variable constraint levels $\eta^k$: 
\begin{equation}\label{subproblem}
	\begin{aligned}
		\min_{x\in{\Rbb^d}} &\quad \psi_0^k(x)  \\
		\st & \quad \psi_i^k(x) \le \eta_i^k,\quad i\in [m].
	\end{aligned}
\end{equation}
Above, we take the surrogate function $\psi_{i}^{k}(x)$ ($0\le i\le m$) by partially linearizing $\psi_i(x)$ at $x^k$ and adding the proximal term $\tfrac{L_i}{2}\gnorm{x-x^k}{}{2}$:
\begin{align}
\psi_{i}^{k}(x)& \coloneq f_{i}(x^{k})+\big\langle\nabla f_{i}(x^{k}),x-x^{k}\big\rangle+\tfrac{L_{i}}{2}\gnorm{x-x^{k}}{}{2}+\chi_{i}(x).\label{eq:psik-1}
\end{align}
It should be noted that our algorithm may not be well-defined if it were to be initialized by an infeasible solution $x^0$. Furthermore, we require the initial point to be strictly feasible with respect to the nonlinear constraints $\psi(x)\le \eta$.
Therefore, we explicitly state this assumption below and assume it holds throughout the paper.
\begin{assumption}[Strict feasibility]\label{ass:strict-feasible} There exist a point $x^0\in \dom\chi_0$ and a vector $\eta^0\in\Rbb^m$ satisfying
\[
\psi_i(x^0)<\eta^0 < \eta.
\]
\end{assumption}
With a strictly feasible solution,  we assume that  the constraint levels $\{\eta^k\}$ are incrementally updated and  converge to  the constraint levels  for the original problem:
\[
\lim_{k\raw\infty} \eta_i^k=\eta_i,\quad i\in[m].
\]

\begin{algorithm}
	\begin{algorithmic}[1]
		\State {\bf Input: } $x^{0}$, $\eta^{0}$;
		\For{k=0,1,2,...,K}
			\State For $i=0,1,\ldots,m$, set $\psi_{i}^{k}(x)$ by (\ref{eq:psik-1});
			\State Update $x^{k+1}$ by solving (\ref{subproblem});
			\State Choose $\delta^{k}>0$ and update $\eta^{k+1}$ by 
			\begin{align}
				\eta_{i}^{k+1}& =\eta_{i}^{k}+\delta_{i}^{k} < \eta_i, & i=1,2,\ldots,m.\label{eq:update-eta}	
			\end{align}
		\EndFor
		\State 
		{\bf Output: }$x^{\hat{k}+1}$ where $\hat{k}$ is a random index sampled from $\{0,1,2,...,K\}$.
	\end{algorithmic}
	\caption{Level constrained proximal gradient method ({\GD}) \label{alg:lcgd}}
\end{algorithm}

The following Lemma will be used many times throughout the rest of the paper.
\begin{lemma}
[Three-point inequality]\label{lem:3pt}Let $g:\Rbb^d\raw (-\infty, \infty]$ be a  proper lsc convex function and
\[
x^{+}=\argmin_{z\in {\Rbb^d}}\left\{ g(z)+\tfrac{\gamma}{2}\gnorm{z-y}{}{2}\right\} ,
\]
then for any $x\in \Rbb^d$, we have 
\begin{equation}
g(x^{+})-g(x)\le\tfrac{\gamma}{2}\brbra{\gnorm{x-y}{}{2}-\gnorm{x^{+}-y}{}{2}-\gnorm{x-x^{+}}{}{2}}.\label{eq:3point}
\end{equation}
\end{lemma}

Next, we  present some important properties of the generated solutions in the following theorem.
\begin{proposition}\label{prop:lcgd-suff-descent}
Suppose that Assumption~\refeq{ass:strict-feasible} holds, then Algorithm~\refeq{alg:lcgd} has the following properties.
\begin{enumerate}
\item The sequence $\big\{ x^{k}\big\}$ {is well-defined and is feasible for problem \eqref{prob:main}}. $\{x^k\}$ satisfies the sufficient descent
property:
\begin{equation}
\tfrac{L_{0}}{2}\gnorm{x^{k+1}-x^{k}}{}{2}\le\psi_{0}(x^{k})-\psi_{0}(x^{k+1}).\label{eq:sufficient-decrease}
\end{equation}
Moreover, the sequence of objective values $\big\{\psi_{0}(x^{k})\big\}$
is monotonically decreasing and $\lim_{k\raw\infty}\psi_{0}(x^{k})$
exists.
\item  
There exists a vector
$\lambda^{k+1}\in\Rbb_{+}^{m}$ such that the KKT condition holds:
\begin{equation}
\begin{aligned}\partial\psi_{0}^{k}(x^{k+1})+\tsum_{i=1}^{m}\lambda_{i}^{k+1}\partial\psi_{i}^{k}(x^{k+1}) & \ni0\\
\lambda_{i}^{k+1}\bsbra{\psi_{i}^{k}(x^{k+1})-\eta_{i}^{k}} & =0,\quad i\in[m].
\end{aligned}
\label{eq:kkt-subprob}
\end{equation}
\end{enumerate}
\end{proposition}

\begin{proof}
Part 1). First, we show that $\{x^k\}$ is a well-defined sequence, namely, $\Xcal_k\cap\dom{\chi_{0}}$ is a nonempty set where $\Xcal_k=\big\{ x:\psi_{i}^{k}(x)\le\eta_{i}^{k}\big\}$. This result clearly holds for $k = 0$ by Assumption \ref{ass:strict-feasible}. We show the general case $(k>0)$ by induction. Suppose that $x^k$ is well-defined, i.e., $\Xcal_{k-1}\cap \dom{\chi_{0}}$ is nonempty. Then, by the definition of $\psi_{i}^{k}, \psi_{i}^{k-1}$ and the definition $x^k$, we have $x^k \in \dom{\chi_{0}}$ and
	\begin{equation}\label{eq:x^k_feasibility}
		\psi_{i}^k(x^k) = \psi_{i}(x^k) \le \psi_{i}^{k-1}(x^{k}) \le \eta^{k-1}_i < \eta_{i}^k \hspace{1em} \text{for all } i \in [m].
\end{equation} 
Here, the first inequality follows due to the  smoothness of $f_i(x)$ which ensures for all $x$, 
\[f_i(x)  \le f_i(x^{k-1})  + \inprod{\grad f_i(x^{k-1})}{x-x^{k-1}} + \tfrac{L_i}{2}\gnorm{x-x^{k-1}}{}{2},\quad \forall i \in [m].\]
This is equivalent to $x^k \in\Xcal_k\cap\dom{\chi_{0}}$, implying that $\Xcal_k \cap \dom{\chi_{0}}$ is nonempty. We conclude that $x^{k+1}$ is well-defined. Hence, by induction, we conclude that $\{x^k\}$ is a well-defined sequence. Furthermore, in view of $x^k\in \dom{\chi_{0}}$, relation~\eqref{eq:x^k_feasibility} and the fact that $\eta^k_i < \eta_i$, we have $x^k \in \dom{\chi_{0}} \cap\{x: \psi_{i}(x)\le \eta_{i},\ i = 1, \dots,m\}$. Hence, the whole sequence $\{x^k\}$ remains feasible for the original problem.

Now, let us apply Lemma~\ref{lem:3pt} with $g(x)=\inprod{\nabla f_{0}(x^{k})}{x}+\chi_{0}(x)+\onebf_{\Xcal_k}(x)$, {$y=x^k$} and $\gamma=L_{0}$.
Then, for any $x \in \dom{\chi_0} \cap \Xcal_k$, we have 
\begin{equation*}
\inprod{\nabla f_{0}(x^{k})}{x^{k+1}-x} + \chi_{0}(x^{k+1})-\chi_{0}(x)\le
{\frac{L_0}{2}\brbra{\gnorm{x-x^k}{}{2}-\gnorm{x-x^{k+1}}{}{2}-\gnorm{x^k-x^{k+1}}{}{2}}}.
\label{eq:3point2}
\end{equation*}
Placing $x=x^{k}$ in the above relation, we have 
\begin{equation} \label{eq:int_rel20}
\inprod{\nabla f_{0}(x^{k})}{x^{k+1}-x^{k}}+\chi_{0}(x^{k+1})-\chi_{0}(x^{k})\le-L_{0}\gnorm{x^{k+1}-x^{k}}{}{2}.
\end{equation}
Moreover, since $f_{0}(\cdot)$ is Lipschitz smooth, we have that
\[
f_{0}(x^{k+1})\le f_{0}(x^{k})+\inprod{f_{0}(x^{k})}{x^{k+1}-x^{k}}+\tfrac{L_{0}}{2}\gnorm{x^{k+1}-x^{k}}{}{2}.
\]
Summing up the above two inequalities and using the definition $\psi_{0}=f_{0}+\chi_{0}$,
we conclude~\eqref{eq:sufficient-decrease}. Hence, the sequence $\big\{\psi_{0}(x^{k})\big\}$ is monotonically decreasing.
The convergence of $\big\{\psi_{0}(x^{k})\big\}$ follows from the lower boundedness assumption.

Part 2). 
Note that  \eqref{eq:x^k_feasibility} ensures the strict feasibility of $x^k$ w.r.t. the constraint set $\Xcal_k$ for the $k$-th subproblem. Therefore, Slater's condition for \eqref{subproblem} and the optimality of $x^{k+1}$ imply that there must exist a vector $\lambda^{k+1} \in \Rbb^m_+$ satisfying KKT condition (\ref{eq:kkt-subprob}).
Hence, we complete the proof.
\end{proof}
{In order to show convergence to the KKT solutions, we need the following constraint qualifications.} 
\begin{assumption}[{Uniform} MFCQ]
\label{assu:y-bounded}
 All the feasible points of  problem~\eqref{prob:main} satisfy MFCQ.
\end{assumption}
We state the main asymptotic convergence property of {\GD} in the following theorem.
\begin{theorem}\label{prop:lcgd:bound-dual}%
Suppose that Assumption~\refeq{assu:y-bounded} holds, then we have the following conclusions.
\begin{enumerate}
	\item The  dual  solutions $\{\lambda^{k+1}\}$ are bounded from above. Namely, there exists a constant $B>0$ such that 
\begin{equation}\label{eq:lambda-bound}
\sup_{0\le k\le \infty}\gnorm{\lambda^{k+1}}{}{}<B.	
\end{equation}
	\item Every limit point of Algorithm~\refeq{alg:lcgd} is a KKT point.
\end{enumerate}
\end{theorem}

\begin{proof}
Part 1). First, we can immediately see that $\{x^k\}$ is a bounded sequence and hence the limit point exists. This result is from Assumption \ref{assum:bounded_set} and the feasibility of $x^k$ for problem \eqref{prob:main} (c.f. Proposition~\ref{prop:lcgd-suff-descent}, Part 1). 
Without loss of generality, we can assume $\lim_{k\raw\infty}x^{k}=\bar{x}$.
For the sake of contradiction, suppose that $\lambda^{k+1}$ is unbounded,
then passing to a subsequence if necessary, we can assume $\gnorm{\lambda^{k+1}}{}{}\raw\infty$
for simplicity. 
Note that we also have $\lim_{k\raw0}\gnorm{x^{k+1}-x^{k}}{}{2}=0$ due
to the sufficient descent property~(\ref{eq:sufficient-decrease}).
From the KKT condition~(\ref{eq:kkt-subprob}), we have 
\begin{equation}\label{eq:lag-opt-x_k}
\psi_{0}^k(x)+\inprod{\lambda^{k+1}}{\psi^{k}(x)-\eta^{k}} \ge \psi_{0}^k(x^{k+1})+\inprod{\lambda^{k+1}}{\psi^{k}(x^{k+1})-\eta^{k}}, \hspace{1em} \forall x {\ \in \dom{\chi_{0}} }.
\end{equation}
Let $X := \bigcap_{i\in[m]}\{x:\psi_i(x) \le \eta_i\} \cap \dom \chi_{0}$ be the feasible domain for problem \eqref{prob:main}. Due to the fact that $x^k \in X$ (Proposition \ref{prop:lcgd-suff-descent}), boundedness of $X$ (Assumption \ref{assum:bounded_set}) and strong convexity of $\psi^k_{0}$, there exists $l_0 \in \Rbb$ such that $X \subset \{x: \psi_{0}^k(x) < l_0\}$ for all $k$. Then, using \eqref{eq:lag-opt-x_k} for all $x \in \dom{\chi_{0}} \cap \{\psi_{0}^k(x) \le l_0\}$ and dividing both sides by $\gnorm{\lambda^{k+1}}{}{}$, we have 
\begin{equation}
\psi_{0}^k(x)/\gnorm{\lambda^{k+1}}{}{}+\big\langle\lambda^{k+1}/\gnorm{\lambda^{k+1}}{}{},\psi^{k}(x)\big\rangle\ge\psi_{0}^k(x^{k+1})/\gnorm{\lambda^{k+1}}{}{}+\big\langle\lambda^{k+1}/\gnorm{\lambda^{k+1}}{}{},\psi^{k}(x^{k+1})\big\rangle.\label{eq:middle-03}
\end{equation}
Let us take $k\raw\infty$ on both sides of (\ref{eq:middle-03}). Note that \added{for all $x \in \dom{\chi_0} \cap \{\psi_0^k(x) \le l_0\}$}, we have
\begin{align}
	&\lim_{k\raw\infty}\psi_{0}^k(x)/\gnorm{\lambda^{{k+1}}}{}{} =0, \quad \lim_{k\raw\infty}\psi_{0}^k(x^{{k+1}})/\gnorm{\lambda^{{k+1}}}{}{} =0,  \label{eq:lim-mid-1} \\
	&\lim_{k\raw\infty}\psi_{i}^{k}(x) =f_{i}(\bar{x})+\langle\nabla f_{i}(\bar{x}),x-\bar{x}\rangle+\tfrac{L_{i}}{2}\gnorm{x-\bar{x}}{}{2}+\chi_{i}(x),\ i\in[m],\label{eq:lim-mid-2}\\
	&\lim_{k\raw\infty}\psi_{i}^{k}(x^{k+1}) =f_{i}(\bar{x})+\chi_{i}(\bar{x})=\psi_{i}(\bar{x}),\ i\in[m], \label{eq:middle-02}
\end{align}
where  (\ref{eq:lim-mid-1}) is due to  boundedness of $\psi_{0}^k(x)$ \added{on $\dom \chi_{0} \cap \{\psi_{0}^k(x) \le l_0\}$} and  boundedness of $\psi_{0}^k(x^{k+1})$ since $x^{k+1} \in X$ which is a bounded set. Moreover,  (\ref{eq:lim-mid-2}) and (\ref{eq:middle-02}) use the continuity of $f_i(x)$ and $\chi_i(x)$, $i\in[m]$.
Next, we consider the sequence $\{u^{k}=\lambda^{k+1}/\gnorm{\lambda^{k+1}}{}{}\}$.
Since $\gnorm{u^{k}}{}{}$ is a bounded sequence, it has a convergent subsequence.
Let $\bar{u}$ be a limit point of $\{u^{k}\}$ and the subsequence
$\{j_{k}\}\subseteq \{1,2,...,\}$ such that $\lim_{k\raw\infty}u^{j_{k}}=\bar{u}$.
Since the subsequence of a convergent sequence is also convergent, we pass to the subsequence $j_{k}$ in~(\ref{eq:middle-03}) and
apply (\ref{eq:lim-mid-1}), (\ref{eq:lim-mid-2}) and (\ref{eq:middle-02}), yielding
\begin{equation}\label{eq:bound-mid-1}
\tsum_{i=1}^{m}\bar{u}_{i}\big[\langle\nabla f_{i}(\bar{x}),x-\bar{x}\rangle+\tfrac{L_{i}}{2}\gnorm{x-\bar{x}}{}{2}+\chi_{i}(x)\big]\ge\tsum_{i=1}^{m}\bar{u}_{i}\chi_{i}(\bar{x}),
\end{equation}
\added{for all $x \in \dom{\chi_{0}} \cap \{\psi_0^k(x) \le l_0\}$.} Hence, $\bar{x}$ minimizes $\tsum_{i=1}^{m}\bar{u}_{i}\big[\langle\nabla f_{i}(\bar{x}),x-\bar{x}\rangle+\tfrac{L_{i}}{2}\gnorm{x-\bar{x}}{}{2}+\chi_{i}(x)\big]$ \added{on $\dom{\chi_{0}} \cap \{\psi_0^k(x) \le l_0\}$}.
Now noting $\bar{x} \in X \subset \{\psi_0^k(x) < l_0\}$ %
and using the stationarity condition for optimality of $\bar{x}$, we have 
\begin{equation}
0\in\tsum_{i=1}^{m}\bar{u}_{i}\big[\nabla f_{i}(\bar{x})+\partial\chi_{i}(\bar{x})\big] + \added{N_{\dom{\chi_{0}}}(\bar{x}) }=\tsum_{i=1}^{m}\bar{u}_{i}\partial\psi_{i}(\bar{x}) + \added{N_{\dom{\chi_{0}}}(\bar{x}) },\label{eq:middle-04}
\end{equation}
where we dropped the constraint $\psi_0^k(x) \le l_0$ due to complementary slackness and the fact that $\psi_{0}^k(\bar{x}) < l_0$.

Let $\Bcal=\{i:\bar{u}_{i}>0\}$, then we must have $\lim_{k\raw\infty}\lambda_{i}^{j_{k}}=\lim_{k\raw\infty}\bar{u}_{i}\gnorm{\lambda^{j_{k}}}{}{}=\infty$
for $i\in\Bcal$. Based on complementary slackness, we have $\psi_{i}^{j_{k}}(x^{j_{k}+1})=\eta_{i}^{j_{k}}$
for any $i\in\Bcal$ for large enough $k$. Due to~(\ref{eq:middle-02}), we have the limit:
$\psi_{i}(\bar{x})=\eta_{i}$. Therefore, the $i$-th constraint is
active at $\bar{x}$, i.e. $i\in\Acal(\bar{x})$. In view of~(\ref{eq:middle-04}),
there exists subgradients $v_{i}\in\partial\psi_{i}(\bar{x}), i \in [m]$ \added{and $v_0 \in N_{\dom{\chi_{0}}}(\bar{x})$} such that 
\begin{equation}
0=v_0 + \tsum_{i\in\Bcal}\bar{u}_{i}v_{i} .\label{eq:middle-05}
\end{equation}
However, equation~(\ref{eq:middle-05}) contradicts the MFCQ assumption.
This is because MFCQ guarantees the existence of $z \added{\in -N^*_{\dom{\chi_{0}}}(\bar{x})}$ such that $\langle z,v_{i}\rangle<0$
for all $i\in\Acal(\bar{x})$, which implies
\begin{equation*}
0 = \langle z,v_0+\tsum_{i\in\Bcal}\bar{u}_{i}v_{i}\rangle\le \tsum_{i\in\Bcal}\bar{u}_{i}\langle z,v_{i}\rangle \le \tsum_{i\in\Bcal} \bar{u}_i \max_{v \in \partial \psi_{i}(\bar{x})}\inprod{z}{v}<0,
\end{equation*}
where the first inequality follows since $z \in -N_{\dom{\chi_{0}}}(\bar{x})$ and $v_0 \in N_{\dom{\chi_{0}}}(\bar{x})$ implying that $\inprod{z}{v_0} \le 0$, the second inequality follow since $\bar{u}_i \ge 0$ and $v_i \in \partial \psi_i(\bar{x})$, and the last strict inequality follows due to uniform MFCQ (c.f. Assumption \ref{assu:y-bounded} and relation \eqref{eq:mfcq}) and $\bar{u}_i > 0$ for at least one $i \in \Bcal$.
This clearly leads to a contradiction. Hence, we conclude that $\{\lambda^{k+1}\}$ is a bounded sequence and conclude the proof.

\added{Part 2).} 
Without loss of generality, we assume that $\bar{x}$ is the only limit point.
Since $\{\lambda^{k+1}\}$ is a bounded sequence, there exists a limit
point $\bar{\lambda}$. Passing to a subsequence if necessary, we
have $\lambda^{k+1}\raw\bar{\lambda}$. 

From the optimality condition $0\in\partial_{x}\Lcal_{k}(x^{k+1},\lambda^{k+1})$,
for any $x$, we have
\begin{align}
 & \big\langle\nabla f_{0}(x^{k})+\tsum_{i=1}^{m}\lambda_i^{k+1}\nabla f_{i}(x^{k}),x^{k+1}-x\big\rangle+\chi_{0}(x^{k+1})-\chi_{0}(x)+\langle\lambda^{k+1},\chi(x^{k+1})-\chi(x)\big\rangle\nonumber \\
 & \quad\le\tfrac{L_{0}+\langle\lambda^{k+1},L\rangle}{2}\big[\gnorm{x-x^{k}}{}{2}-\gnorm{x^{k+1}-x^{k}}{}{2}-\gnorm{x-x^{k+1}}{}{2}\big].\label{eq:middle-06}
\end{align}
Let us take $k\raw\infty$ on both sides of~(\ref{eq:middle-06}).
We note that  $\lim_{k\raw\infty}\gnorm{x^{k}-x^{k+1}}{}{}=0$ due to Proposition~\ref{prop:lcgd-suff-descent}, $\lim_{k\raw\infty}\added{\chi_i(x^{k+1})=\chi_i(\bar{x})}$ due to the continuity of $\chi_{i}$ ($i\in[m]$), and $\chi_{0}(\bar{x})\le\liminf_{k\raw\infty}\chi_{0}(x^{k})$ due to the lower semi-continuity of $\chi_{0}(\cdot)$. 
It then follows that
\begin{equation}\label{eq:opt_interemediate_KKT}
	\big\langle\nabla f_{0}(\bar{x})+\tsum_{i=1}^{m}\bar{\lambda}_i\nabla f_{i}(\bar{x}),\bar{x}-x\big\rangle+\chi_{0}(\bar{x})-\chi_{0}(x)+\langle\bar{\lambda},\chi(\bar{x})-\chi(x)\big\rangle\le0.
\end{equation}
In other words, $\bar{x}$ is the minimizer of  convex optimization problem $\min_x\big\langle\nabla f_0(\bar{x})+\tsum_{i=1}^{m}\bar{\lambda}_i\nabla f_{i}(\bar{x}),x\big\rangle+\chi_{0}(x)+\bar{\lambda}\big[\chi(x)-\eta\big]$ over \added{$\dom \chi_0$}. 
Hence we have 
\begin{equation}
0\in\nabla f_0(\bar{x})+\partial\chi_{0}(\bar{x})+\langle\bar{\lambda},\partial\psi(\bar{x})\rangle.\label{eq:kkt-mid-1}
\end{equation}
In addition, using the complementary slackness, we have
\begin{align}
0 & =\lim_{k\raw\infty}\tsum_{i=1}^{m}\lambda_{i}^{k+1}\big[\psi_{i}^{k}(x^{k+1})-\eta_{i}^{k}\big]\nonumber \\
 & =\lim_{k\raw\infty}\tsum_{i=1}^{m}\lambda_{i}^{k+1}\big[f_{i}(x^{k})+\langle\nabla f_{i}(x^{k}),x^{k+1}-x^{k}\rangle+\tfrac{L_{i}}{2}\gnorm{x^{k+1}-x^{k}}{}{2}+\chi_{i}(x^{k+1})-\eta_{i}^{k}\big]\nonumber \\
 & =\tsum_{i=1}^{m}\bar{\lambda}_{i}\big[f_{i}(\bar{x})+\chi_{i}(\bar{x})-\eta_{i}\big]\nonumber \\
 & =\langle\bar{\lambda},\psi(\bar{x})-\eta\rangle,\label{eq:kkt-mid2}
\end{align}
due to the convergence $\lim_{k\raw\infty}\lambda^{k+1}=\bar{\lambda}$,
$\lim_{k\raw\infty}\eta^{k}=\eta$, $\lim_{k\raw\infty}\chi(x^{k+1})=\chi(\bar{x})$
and $\lim_{k\raw\infty}\gnorm{x^{k+1}-x^{k}}{}{}=0$. Putting (\ref{eq:kkt-mid-1})
and (\ref{eq:kkt-mid2}) together, we conclude that $(\bar{x},\bar{\lambda})$
satisfies the KKT condition.
\end{proof}

\begin{remark}\label{rem:limit-point_x^k}
	To show the existence of a limit point $\bar{x}$, we  use Assumption~\ref{assum:bounded_set} to ensure that the sequence $x^k$ remains in a bounded domain. 
	For the sake of conciseness, we henceforth assume the existence of a limit point $\bar{x}$ and do not delve into the technical assumption used to ensure this condition.
	Moreover, it should be noted that the boundedness property can be obtained under other assumptions, e.g., %
	assuming the compactness of sublevel set $\{x: \psi_{0}(x) \le \psi_{0}(x^0)\}$  and using the sufficient descent condition \eqref{eq:sufficient-decrease},  we can immediately show the existence of $\bar{x}$. However,  it appears to be more challenging to show convergence via this approach when sufficient descent condition fails (e.g., in the forthcoming stochastic optimization). 
\end{remark}

\subsection{Dependence of $B$ on the constraint qualification.} \label{sec:B_bounded} In Theorem \ref{prop:lcgd:bound-dual}, we proved existence of a bound $B$ on the dual multiplier. However, the size of that bound still remains unknown. Through Example \ref{ex:SCAD-2D} below, we observe that the limiting behaviour of the sequence $\lambda^k$ (which largely governs the size of $B$) is closely tied to the magnitude of the number $c(\bar x)$, where 
$$c(x) := -\min_{\gnorm{z}{}{} \le 1}\max_{v \in \partial \psi_i(x)} \inprod{v}{z}.$$ 
Here, the inner max follows from the relation \eqref{eq:mfcq} and outer min tries to find the best possible $z$ that ensures MFCQ. It is observed that if $c(\bar{x})$ is a large positive number, then MFCQ is strongly satisfied and $B$ is a reasonable bound. In contrast, if $c(\bar{x})$ is close to $0$, then $B$ can get quite large. 
\begin{eg} \label{ex:SCAD-2D}
Consider a two dimensional optimization problem with SCAD constraint: $\min_{x} \psi_0(x)$ subject to $\psi_1(x) \le \eta_1$ where $\psi_0(x) = 7-x_1$ and $\psi_1(x) = \beta\gnorm{x}{1}{} - \tsum_{i =1}^2 h_{\beta, \theta} (x_i)$. Note that %
\begin{equation}\label{eq:scad}
h_{\beta,\theta}(u) = \begin{cases} 0 &\text{if $|u| \le \beta$}; \\
\tfrac{(|u|-\beta)^2}{2(\theta-1)} &\text{if } \beta \le |u| \le \beta \theta;\\
\beta |u| - \tfrac{(\theta+1)\beta^2}{2} &\text{if } |u| \ge \beta\theta \end{cases}.
\end{equation} 
This function fits our framework with the smoothness parameter $\tfrac{1}{\theta-1}$. 
Lets consider $\beta = 1, \theta = 5$, the level $\eta_1  = 3$ and limit point $\bar{x} = (5,0)$. Clearly, the constraint is active at $\bar{x}$ and $h$ is $\tfrac{1}{4}$-smooth function. 
Then, $\partial \psi(\bar{x}) = \{\begin{bmatrix}0 \\ t \end{bmatrix}: t \in [-1,1]\}$ as per the definition. 
Then, one can see that uniform MFCQ is violated at the limit point $\bar{x}$. Indeed, 
\[ \max_{v \in \partial \psi(x)} \inprod v z = \max_{t \in  [-1,1]} tz_2 = \abs{z_2} \nless 0, \]
implying $c(\bar x) = 0$. Furthermore, no $\lambda$ can be found satisfying the KKT condition
\[  \begin{bmatrix}
    -1 \\ 0
\end{bmatrix}+ \lambda \begin{bmatrix}
    0 \\ t
\end{bmatrix} = \begin{bmatrix}
    0 \\ 0
\end{bmatrix},\]
for all $t \in [-1,1]$.
Hence, as we get close to this limit point, bound on $\gnorm{\lambda^k}{}{}$ will get arbitrarily large. Easy way to see this fact is to construct a subproblem at the limit point itself. After observing the feasible region for the subproblem at $(5,0)$, it is clear that it has only one feasible solution $(5,0)$ which gives rise to degeneracy. See Figure \ref{fig:MFCQ_unsatisfied} for more details. Figure \ref{fig:at4_0} shows the well-behaved subproblem at an interior point while Figure \ref{fig:limit_MFCQ_unsatisfied} show the degeneracy at the limit.

However, as we change level $\eta_1$ to any value either above or below $3$, we do not get any violation of MFCQ. It also gives nondegenerate feasible sets at limit point and $\lambda^k$ remains bounded for all $k$. 
See Figure \ref{fig:MFCQ_satisfied} below for more details. In particular, if $\eta_1 = 2.5 < 3$, then $\bar{x} = (3,0)$ is the limit point. At this point, we have $\partial \psi(\hat{x}) = \{\begin{bmatrix}0.5 \\ t \end{bmatrix}: t \in [-1,1]\}$. Moreover, we have
\[ \max_{v \in \partial \psi(\bar{x})} \inprod{v}{z} = 0.5z_1 + \max_{t \in [-1,1]}tz_2 = 0.5z_1 + |z_2|. \]
Choosing the unit vector $z = (z_1, z_2) = (-1,0)$, we obtain that point $\bar{x}$ satisfies MFCQ with $c(\bar{x}) = 0.5$. Hence, even when the search point reaches the limit point $\hat{x}$ (i.e., $\epsilon \to 0$), the $\lambda^k$ still exists. (See, in particular, Figure \ref{fig:limit_MFCQ_satisfied} whose subproblem at $\bar{x}$ has a nonempty interior).
\end{eg}
\begin{minipage}{0.48\textwidth}
    \vspace{-2em}
    \begin{figure}[H]
    \centering
    \caption{The nonconvex constraint $\psi_1(x) \le \eta_1$ where $\eta_1 =3$. The dotted blue curves are the subproblem constraint for two different points. Since the MFCQ is violated at $(5,0)$, the subproblem reduces to a single feasible point at the limit point $(5,0)$.}
    \label{fig:MFCQ_unsatisfied}
    \begin{subfigure}[b]{0.48\textwidth}
        \centering
        \captionsetup{justification=raggedright}
        \includegraphics[width=0.99\textwidth]{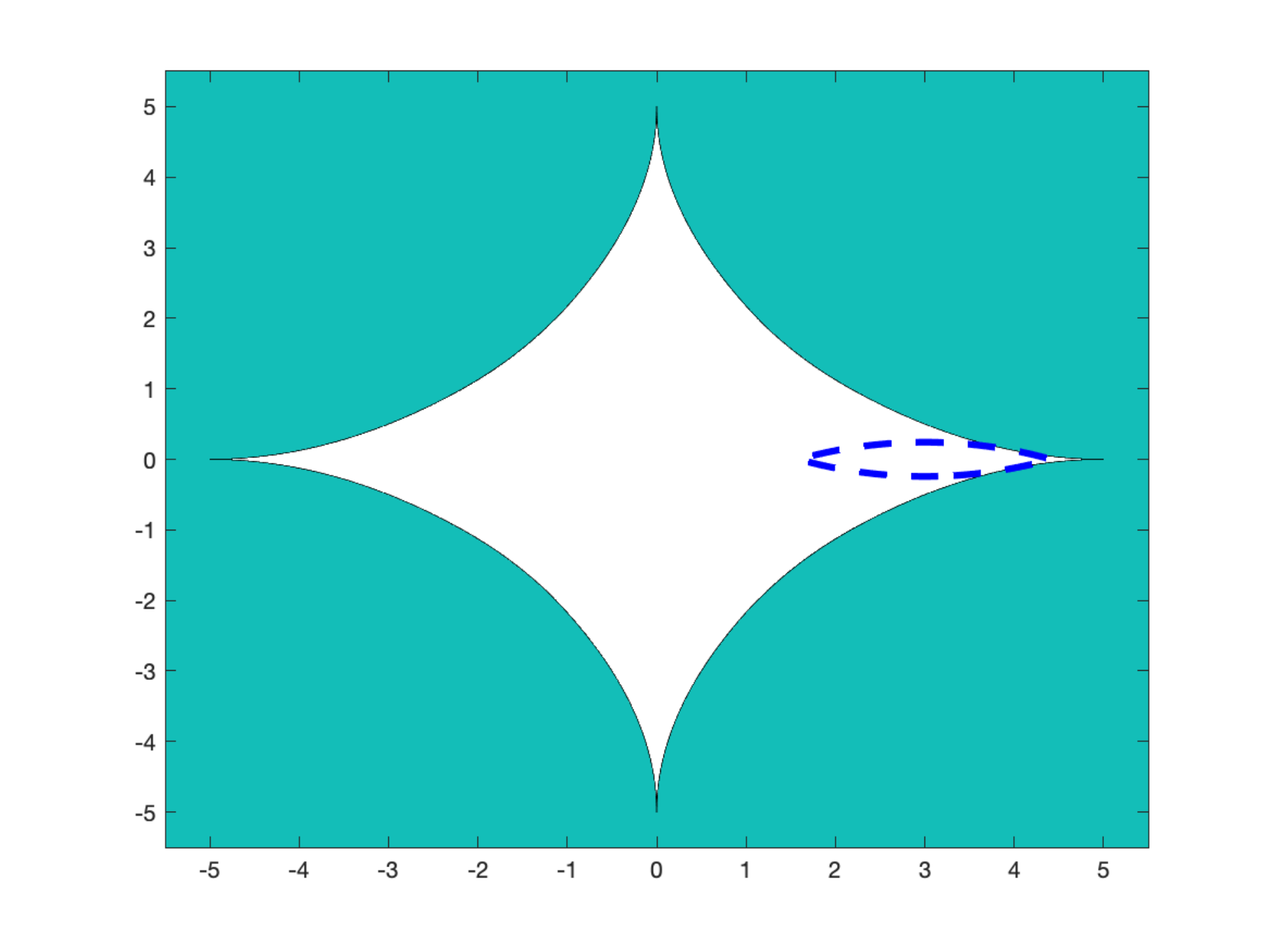}
        \caption{Dotted constraint for the subproblem at (4,0)}\label{fig:at4_0}
    \end{subfigure}
    \hfill
    \begin{subfigure}[b]{0.48\textwidth}
        \centering
        \includegraphics[width=0.99\textwidth]{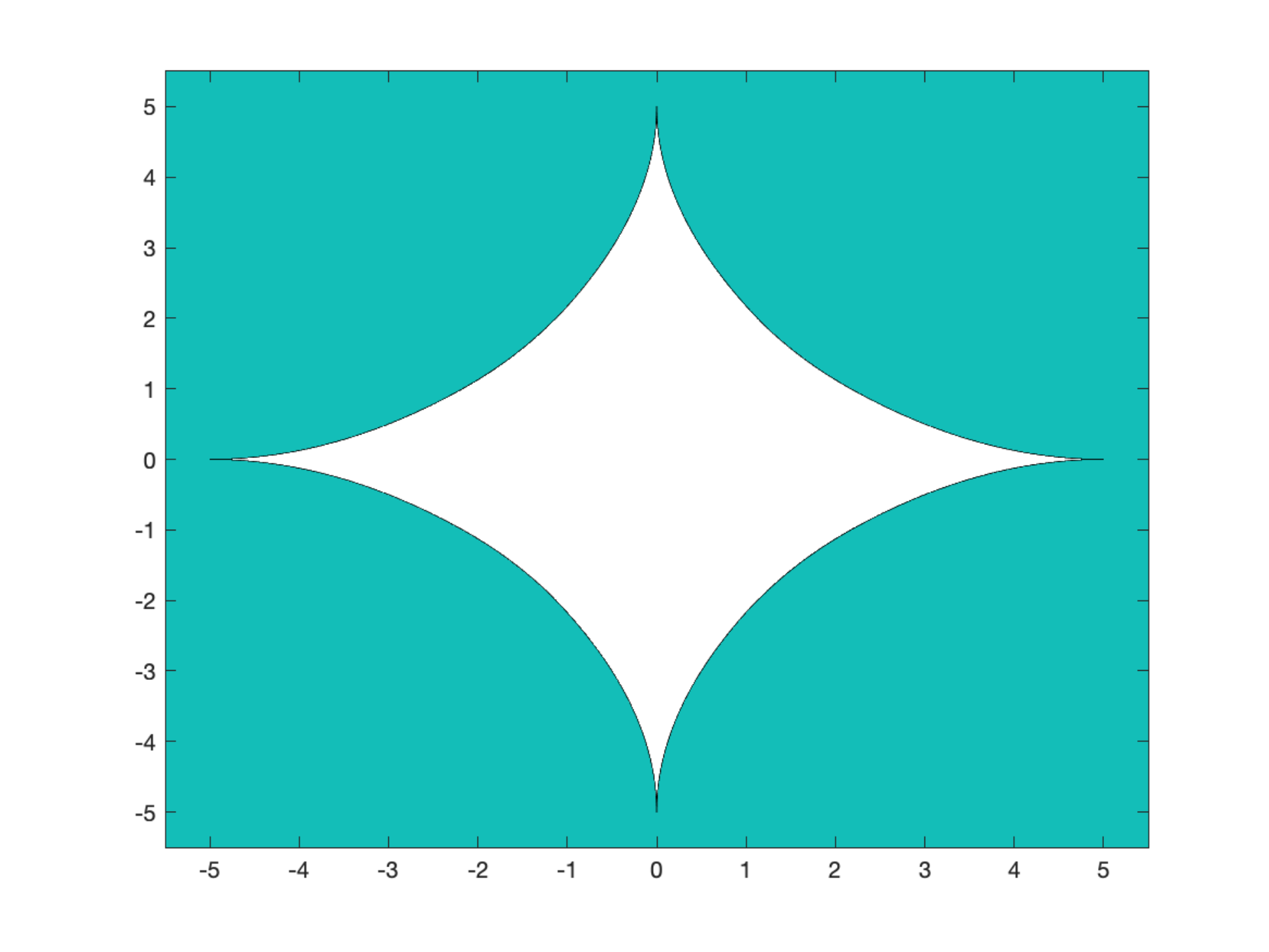}    
        \caption{Degenerate limiting subproblem at (5,0) }\label{fig:limit_MFCQ_unsatisfied}
    \end{subfigure}
\end{figure}    
\end{minipage}
\hfill
\begin{minipage}{0.48\textwidth}
    \vspace{-2em}
    \begin{figure}[H]
        \centering
        \caption{The nonconvex constraint $\psi_1(x)\le \eta_1$ where $\eta_1 = 2.5$. The dotted blue curves are subproblem constraint for two different points. Since the MFCQ is satisfied, the limiting subproblem constraint at $(3,0)$ is still a full dimensional set with nonempty interior. %
        }\label{fig:MFCQ_satisfied}
        \begin{subfigure}[b]{0.48\textwidth}
            \centering
            \includegraphics[width=0.99\linewidth]{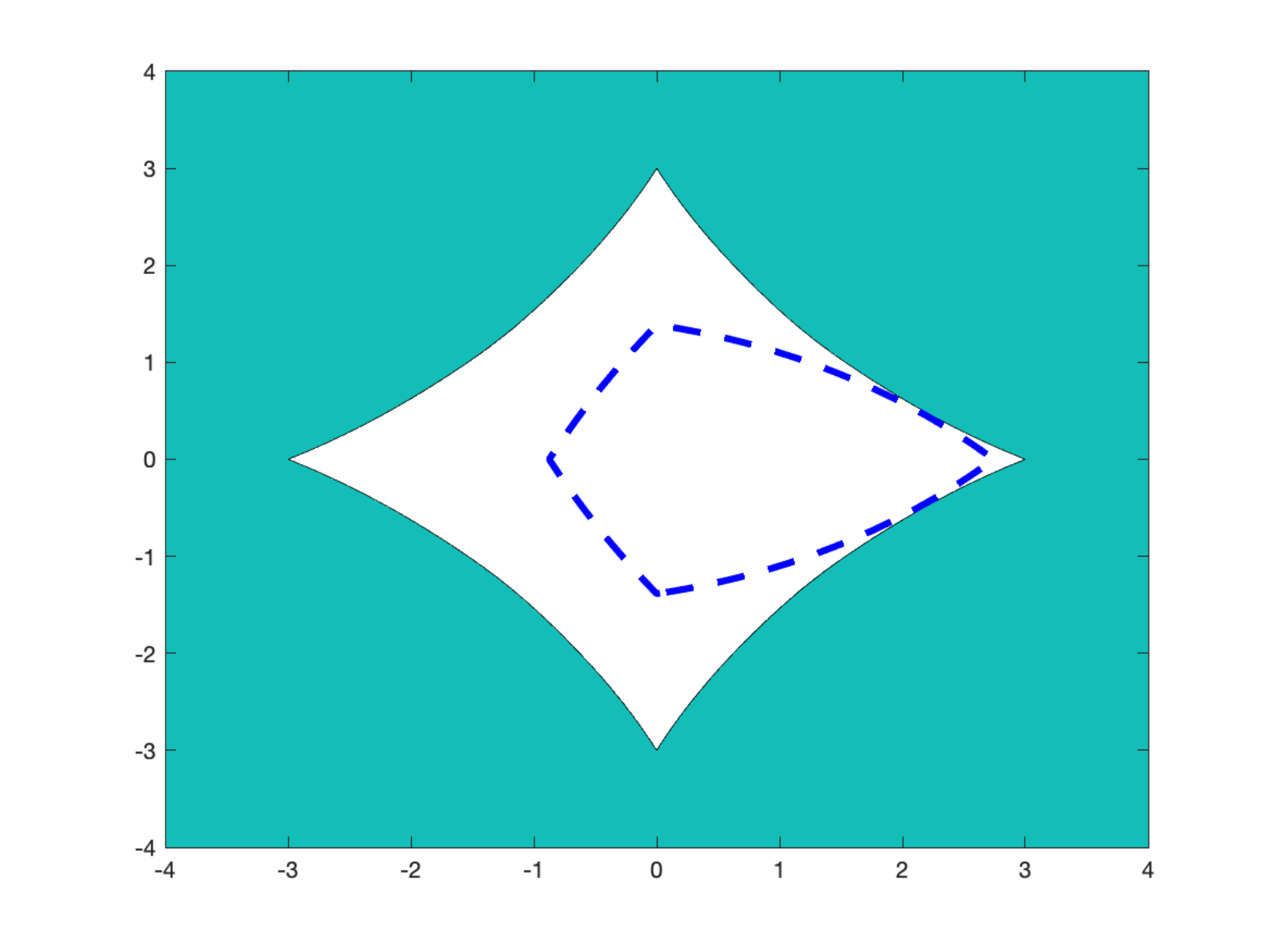}
            \caption{Dotted constraint for the subproblem at (2,0)}%
        \end{subfigure}
        \hfill
        \begin{subfigure}[b]{0.48\textwidth}
            \centering
            \includegraphics[width=0.99\linewidth]{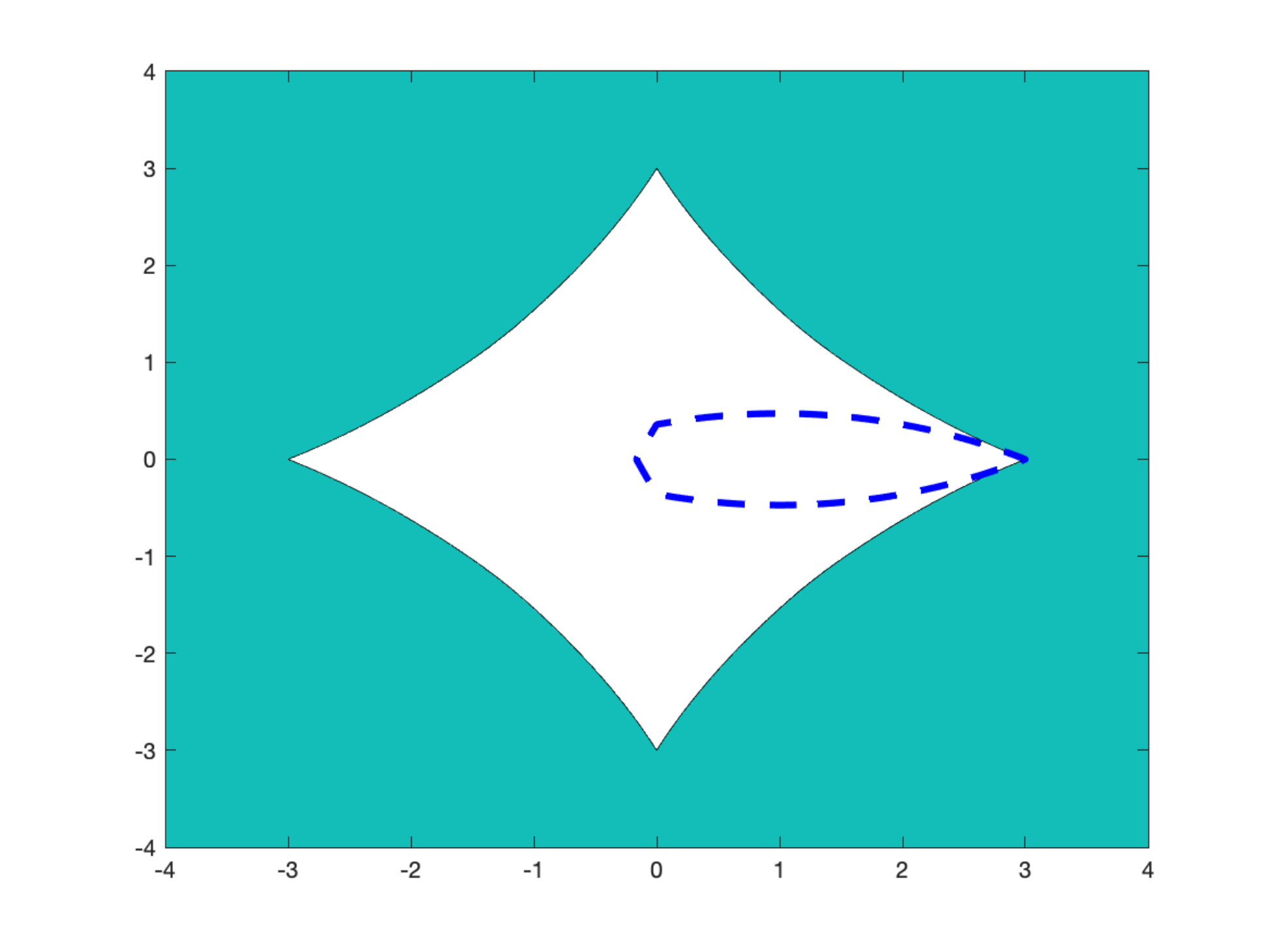}    
            \caption{Constraint for limiting subproblem at (3,0)} \label{fig:limit_MFCQ_satisfied}
        \end{subfigure}
    \end{figure}
\end{minipage}

In view of the above example, we see that the limiting behavior of $\gnorm{\lambda^k}{}{}$ (and by implication the order of $B$) is closely related to the ``strength'' of the constraint qualification MFCQ at the limit point. In order to get an apriori bound $B$, we use a somewhat stronger yet verifiable constraint qualification called as {\em strong feasibility} which is a slight modification of the CQ proposed in \cite[Assumption 3]{boob2019proximal}.

\begin{assumption}[Strong feasibility CQ] \label{ass:strong_feas}
    There exists $\hat{x} \in \Xcal := \bigcap_{i\in[m]}\{x:\psi_i(x) \le \eta_i\} \cap \dom \chi_{0}$ such that 
    \begin{equation}
        \psi_i(\hat{x}) \le \eta^0_i - 2L_iD_\Xcal^2
    \end{equation}
    where $D_\Xcal := \max_{x_1, x_2 \in \Xcal} \gnorm{x_1-x_2}{}{}$ is the diameter of the set $\Xcal$.
\end{assumption}
In view of Assumption \ref{assum:bounded_set}, we note that $\Xcal$ is a bounded set. Hence, $D_\Xcal$ and (consequently) Assumption \ref{ass:strong_feas} are well-defined. See \cite{boob2019proximal} for a connection between Assumption \ref{ass:strong_feas} and Assumption \ref{assu:y-bounded}. Below, we show that strong feasibility CQ leads to a fixed apriori bound on $\lambda^k$.

\begin{lemma}\label{lem:strong_feas_prior_bound}
    Suppose Assumption \ref{ass:strong_feas} is satisfied. Then, $\gnorm{\lambda^k}{1}{} \le B := \tfrac{\psi_0(\hat{x}) - \psi_0^* + L_0D_\Xcal^2}{L_{\min}D_\Xcal^2}$
\end{lemma}
\begin{proof}
    Note that 
    \begin{align}
        \psi_i^k(\hat{x}) &= f_i(x^k) + \inprod{\grad f_i(x^k)}{\hat{x} - x^k} + \tfrac{L_i}{2} \gnorm{\hat{x} - x^k}{}{2} + \chi_i(\hat{x}) \nonumber\\
        & \le f_i(\hat{x}) + \chi_i(\hat{x}) + L_i\gnorm{\hat{x} -x^k}{}{2} \nonumber\\
        &= \psi_i(\hat{x}) + L_i\gnorm{\hat{x} -x^k}{}{2} \le \eta^0_i-L_iD_\Xcal^2 < \eta^k_i - L_iD_\Xcal^2, \label{eq:subprob_feas_under_strong_feas}
    \end{align}
    where first inequality uses $f_i(\hat{x}) \ge f_i(x^k) + \inprod{\grad f_i(x^k)}{\hat{x} - x^k} - \tfrac{L_i}{2} \gnorm{\hat{x} -x^k}{}{2}$ (follows due to $L_i$-smoothness of $f_i$), and second inequality follows by Assumption \ref{ass:strong_feas} along with the fact that $x^k \in \Xcal$ (see Proposition \ref{prop:lcgd-suff-descent}). 
    
    In view of \eqref{eq:subprob_feas_under_strong_feas}, we have strict feasibility of subproblem \eqref{subproblem} for all $k$ implying that $\lambda^{k+1}$ exists. Hence, we have $x^{k+1} = \argmin_{x} \psi^k_0(x) + \inprod{\lambda^{k+1}}{\psi^k(x)}$. Then, for all $x \in \dom \chi_0$, we have
    \begin{align*}
         \psi^k_0(x^{k+1}) &= \psi^k_0(x^{k+1}) + \inprod{\lambda^{k+1}}{\psi^k(x^{k+1}) -\eta^k} \\
         &\le \psi^k_0(x) + \inprod{\lambda^{k+1}}{\psi^k(x) - \eta^k} 
    \end{align*}
    where  equality follows from the complementary slackness of the KKT condition, and  inequality is due to optimality of $x^{k+1}$. Using $x = \hat{x}$ in the above relation and combining with  \eqref{eq:subprob_feas_under_strong_feas}, we obtain
    \begin{align}
        \psi^k_0(\hat{x}) - \psi^k_0(x^{k+1}) \ge \inprod{\lambda^{k+1}}{ \eta^k-\psi^k(\hat{x})} \ge \inprod{\lambda^{k+1}}{L}D_\Xcal^2 \ge \gnorm{\lambda^{k+1}}{1}{}L_{\min}D_\Xcal^2 \label{eq:strong_feas_bound_int_rel1}
    \end{align}
    Finally, note that 
    \begin{align*}
        \psi^k_0(\hat{x}) - \psi^k_0(x^{k+1}) \le \psi_0(\hat{x}) + L_0\gnorm{\hat{x} -x^k}{}{2} -\psi_0(x^{k+1}) \le \psi_0(\hat{x}) - \psi_0^* + L_0D_\Xcal^2
    \end{align*}
    where first inequality follows by \eqref{eq:subprob_feas_under_strong_feas} for $i = 0$ and $\psi^k_0(x) \ge \psi_0(x)$, and second inequality follows by the definition of $\psi_0^*$ and $D_\Xcal$. Combining the above relation with \eqref{eq:strong_feas_bound_int_rel1}, we get the result. Hence, we conclude the proof.
\end{proof}

The  discussion above implies that the value of $B$ is intricately related to the constraint qualification. While uniform MFCQ is unverifiable and does not allow for a priori bounds on $B$, it is widely used in nonlinear programming  to ensure the existence of such a bound \cite{bertsekasnonlinear}. As observed in Figure~\ref{fig:limit_MFCQ_unsatisfied} and Figure~\ref{fig:limit_MFCQ_satisfied}, the actual value of $B$ depends on the closeness of the MFCQ violation at the limit point. This situation is rare, but the current assumptions do not eliminate that possibility.  Problems of this nature are ill-conditioned, and to our knowledge, no algorithm can ensure bounds on the dual in such a situation. The existing literature deals with this issue in two ways: One track assumes existence of $B$ (similar to Theorem \ref{prop:lcgd:bound-dual}) and performs the complexity or convergence analysis; A second track assumes a stronger constraint qualification that removes the ill-conditioned problems and shows more explicit bound on the dual (similar to Lemma \ref{lem:strong_feas_prior_bound}). %
We perform our analysis for both cases. To conclude, we henceforth assume that the bound $B$ is a constant and do not delve into the discussion on related constraint qualification. To substantiate that the bound $B$ is small, we perform detailed numerical experiments in Section \ref{sec:experiments}.

\subsection{Convergence rate analysis of \GD~method.}
Our main goal now is to develop some non-asymptotic convergence rates for Algorithm~\ref{alg:lcgd}. \begin{lemma}
\label{lem:lk-bound}In Algorithm~\refeq{alg:lcgd}, for $k=1,2,...,$
we have
\begin{align}
\gnorm{ \partial_{x}\Lcal(x^{k+1},\lambda^{k+1})}{-}{}  & \le2\big(L_{0}+\inner{\lambda^{k+1}}{L}\big)\gnorm{x^{k+1}-x^{k}}{}{}.\label{eq:bound-L-1}
\end{align}
\end{lemma}
\begin{proof}
Let $\Lcal_{k}$ be the Lagrangian function of subproblem~\eqref{subproblem}:
\begin{align}
\Lcal_{k}(x,\lambda) & =\psi_{0}^{k}(x)+\tsum_{i=1}^{m}\lambda_{i}\sbra{\psi_{i}^{k}(x)-\eta_{i}^{k}}. \label{eq:lagrange-1}
\end{align}
Using \eqref{eq:lagrange-0} and \eqref{eq:lagrange-1}, we have 
\begin{align}
 & \partial_x\Lcal_{k}(x^{k+1},\lambda^{k+1})\nonumber \\
 &\quad = \nabla f_{0}(x^{k})+L_{0}(x^{k+1}-x^{k})+\partial\chi_{0}(x^{k+1})+\tsum_{i=1}^{m}\lambda_{i}^{k+1}\big[\nabla f_{i}(x^{k})+L_{i}(x^{k+1}-x^{k})+\partial_x\chi_{i}(x^{k+1})\big]\nonumber \\
 &\quad = \partial_x\Lcal(x^{k+1},\lambda^{k+1})+\nabla f_{0}(x^{k})-\nabla f_{0}(x^{k+1})+\tsum_{i=1}^{m}\lambda_{i}^{k+1}\big[\nabla f_{i}(x^{k})-\nabla f_{i}(x^{k+1})\big]\nonumber \\
 &\qquad +\big(L_{0}+\langle\lambda^{k+1},L\rangle\big)\big(x^{k+1}-x^{k}\big).\label{eq:relation-L}
\end{align}
Using the smoothness of $f_i(x)$, the optimality condition $0\in\partial_x\Lcal_{k}(x^{k+1},\lambda^{k+1})$ and the triangle inequality, we obtain
\begin{align}
\gnorm{\partial_x\Lcal(x^{k+1},\lambda^{k+1})}{-}{} & \le 2L_{0}\gnorm{x^{k+1}-x^{k}}{}{}+2\inner{\lambda^{k+1}}{L}\gnorm{x^{k+1}-x^{k}}{}{}.\label{eq:bound-L-1-1}
\end{align}
Hence we conclude the proof.
\end{proof}
In view of Lemma~\ref{lem:lk-bound}, we derive the complexity of
{\GD} to attain approximate KKT solutions in the following theorem.

\begin{theorem}\label{thm:lcgd:sum-bound}
Let $\alpha_k>0$  $(k=0,1,..,K)$ be a non-decreasing sequence  and suppose that Assumption~\refeq{assu:y-bounded} holds, then
there exists a constant $B>0$ such that 
\begin{align}
\tsum_{k=0}^{K}\alpha_k\gnorm{ \partial_{x}\Lcal(x^{k+1},\lambda^{k+1})}{-}{2} & \le{8(L_{0}+B\gnorm{L}{}{})^{2}} D^2\alpha_K,\label{eq:sum-lag-square-1}\\
\tsum_{k=0}^{K} \alpha_k \langle \lambda^{k+1}, \vert\psi(x^{k+1})-\eta\vert \rangle  
& \le 2B\gnorm{L}{}{}D^2\alpha_K  + B \tsum_{k=0}^K\alpha_k\gnorm{\eta-\eta^k}{}{}, \label{eq:sum-slack-1}
\end{align}
where $D=\sqrt{\tfrac{\psi_{0}(x^{0})-\psi_{0}^{*}}{L_{0}}}$.  Moreover, if we choose the index $\hat{k}\in\{0,1,...,K\}$ with probability $\prob(\hat{k}=k)=\alpha_k/(\tsum_{k=0}^K\alpha_k)$, then $x^{\hat{k}+1}$ is a randomized $\epsilon_K$ type-I KKT point with 
\begin{equation}\label{eq:epsi-1-2}
	\epsilon_K=\tfrac{1}{\tsum_{k=0}^K\alpha_k} \max \bcbra{ {8(L_{0}+B\gnorm{L}{}{})^{2}}D^2\alpha_K,\ 2B\gnorm{L}{}{} D^2 \alpha_K + B \tsum_{k=0}^K\alpha_k\gnorm{\eta-\eta^k}{}{}}
\end{equation}
\end{theorem}

\begin{proof}
From the sufficient descent property~(\ref{eq:sufficient-decrease}), we have
\begin{align}\label{eq:weighted-square-sum}
\tsum_{k=0}^K {\alpha_k}\gnorm{x^k-x^{k+1}}{}{2}  &\le \tfrac{2}{L_0}  \tsum_{k=0}^K \alpha_k \bsbra{\psi_0(x^k)-\psi_0(x^{k+1})} \nonumber \\
 & =  \tfrac{2}{L_0}\bsbra{\alpha_0\psi_0(x^0) + \tsum_{k=1}^{K}(\alpha_k-\alpha_{k-1})\psi_0(x^k) - \alpha_K\psi_0(x^{K+1}) } \nonumber \\
&  \le \tfrac{2\alpha_K}{L_0} \bsbra{\psi_0(x^0) - \psi_0(x^{K+1})} \nonumber\\
& \le \tfrac{2\alpha_K}{L_0} \bsbra{\psi_0(x^0) - \psi_0(x^{*})}  = 2\alpha_K D^2, 
\end{align}
where the  second inequality uses the monotonicity of sequence $\psi_0(x^k)$. 
In view of Theorem~\ref{prop:lcgd:bound-dual} and Cauchy-Schwarz inequality, we have $\langle\lambda^{k+1},L\rangle\le \gnorm{\lambda^{k+1}}{}{}\gnorm{L}{}{} \le B\gnorm{L}{}{}$.
This relation and (\ref{eq:bound-L-1}) implies
\[
\setnorm{\partial_{x}\Lcal(x^{k+1},\lambda^{k+1})}^{2}\le 4{(L_{0}+B\gnorm{L}{}{})^{2}} \gnorm{x^k-x^{k+1}}{}{2}.
\]
Combining the above inequality with (\ref{eq:weighted-square-sum}) immediately yields \eqref{eq:sum-lag-square-1}.

Next, we bound the error of complementary slackness. We have
\begin{align}
\tsum_{i=1}^m\lambda_{i}^{k+1}\vert\psi_{i}(x^{k+1})-\eta_{i}\vert 
& =\tsum_{i=1}^m \lambda_{i}^{k+1}\big|\psi_i^k(x^{k+1})-\eta_i^k -( \eta_i-\eta_i^k) + \psi_{i}(x^{k+1})-\psi_{i}^k(x^{k+1})\big|\nonumber \\
 & \le \tsum_{i=1}^m\big[\lambda_{i}^{k+1}\big|\psi_{i}^{k}(x^{k+1})-\eta_{i}^{k}\big|+\lambda_{i}^{k+1}(\eta_{i}-\eta_{i}^{k})\nonumber \\
 & \quad+\lambda_{i}^{k+1}\big\vert f_{i}(x^{k+1})-f_{i}(x^{k})-\langle\nabla f_{i}(x^{k}),x^{k+1}-x^{k}\rangle-\tfrac{L_{i}}{2}\gnorm{x^{k+1}-x^{k}}{}{2}\big\vert \big]\nonumber \\
 & \le \tsum_{i=1}^m\lambda_{i}^{k+1}(\eta_{i}-\eta_{i}^{k})+\lambda_{i}^{k+1}{L}_{i}\gnorm{x^{k+1}-x^{k}}{}{2}\nonumber \\
 & \le B \gnorm{\eta-\eta^k}{}{} + B\gnorm{L}{}{} \gnorm{x^{k+1}-x^k}{}2 \label{eq:bound-cs-k}
\end{align}
 where the first inequality uses the triangle
inequality, the second inequality uses complementary slackness and the Lipschitz smoothness
of $f_{i}(\cdot)$, and the last inequality follows from Cauchy-Schwartz inequality and boundedness  of $\lambda^{k+1}$. Summing up (\ref{eq:bound-cs-k})  weighted by $\alpha_k$ for $k=0,...,K$,
we have 
\begin{align*}
\tsum_{k=0}^{K}\alpha_k \langle \lambda^{k+1}, \vert\psi(x^{k+1})-\eta\vert \rangle 
& \le \tsum_{k=0}^{K} \alpha_k \big[{B\gnorm{L}{}{}}\gnorm{x^{k+1}-x^k}{}2 + B \gnorm{\eta-\eta^k}{}{} \big].
\end{align*}
Combining the above result with \eqref{eq:weighted-square-sum} gives \eqref{eq:sum-slack-1}. Finally, the fact that $x^{\hat{k}+1}$ is a randomized $\epsilon_K$ type-I KKT point for $\epsilon_K$ defined in \eqref{eq:epsi-1-2} is immediately follows from \eqref{eq:sum-lag-square-1}, \eqref{eq:sum-slack-1} and Definition~\ref{def:type1-KKT}.
\end{proof}
The following corollary shows that the output of Algorithm~\ref{alg:lcgd}
is a randomized $\Ocal(1/K)$ KKT
point under more specific parameter selection.
\begin{corollary}\label{cor:rate-lcgd}
In Algorithm~\refeq{alg:lcgd}, %
suppose that all the assumptions of Theorem~\refeq{thm:lcgd:sum-bound} hold. Set $\delta^{k}=\tfrac{\eta-\eta^{0}}{(k+1)(k+2)}$
and $\alpha_k=k+1$. Then  $x^{\hat{k}+1}$ is a randomized $\epsilon$ Type-I KKT point with
\begin{equation}
	\epsilon = \tfrac{2}{K+2}\max\bcbra{8(L_{0}+B\gnorm{L}{}{})^{2}D^2, 2B\gnorm{L}{}{} D^2 + B \gnorm{\eta-\eta^0}{}{}}
\end{equation}
\end{corollary}

\begin{proof}
Notice that $\alpha_K=K+1$, $\tsum_{k=0}^K \alpha_k = \tfrac{(K+1)(K+2)}{2}$.
Moreover, for $i\in[m]$ and $k\ge0$, we have
\[
\eta_{i}^{k}=\eta_{i}^{0}+\tsum_{i=0}^{k-1}\delta^{i}=\eta_{i}^{0}+(\eta_{i}-\eta_{i}^{0})\tsum_{i=0}^{k-1}\tfrac{1}{(i+1)(i+2)}=\tfrac{k}{k+1}\eta_{i}+\tfrac{1}{k+1}\eta_{i}^{0},
\]
which implies that $\tsum_{k=0}^{K}\alpha_k\gnorm{\eta-\eta^{k}}{}{}=(K+1) \gnorm{\eta-\eta^0}{}{}$. Plugging these values into \eqref{eq:epsi-1-2} gives us the desired conclusion.
\end{proof}
\begin{remark}
Corollary~\ref{cor:rate-lcgd} shows that the gradient complexity of {\GD} for smooth composite constrained problems is on a par with that of gradient descent for unconstrained optimization problems.
To the best of our knowledge, this is the first complexity result for a constrained problem where the constraint functions can be nonsmooth and nonconvex.
Note that the convergence rate involves the unknown bound $B$ on the Lagrangian multipliers. The presence of such a constant is not new in nonlinear programming literature~\cite{Auslender2010moving,CartisGouldToint11-1,facchinei2017feasible,cartis2014on,Facchinei19}. 
Fortunately,  we can safely implement \GD~method since the step-size scheme does not rely on $B$. 
On the other hand,  the bound $B$ is often a problem-dependent quantity. E.g., in \cite{boob2019proximal} authors show a class of problems for which an a priori bound $B$ can be established, or \cite{Boob2020feasible} shows the exact value of $B$ for a class of nonconvex relaxations of sparse optimization problems. In such cases, our comparisons are arguably fair.  
Hence, throughout the paper, we make comparative statements under the assumption that $B$ largely depends on the problem.
\end{remark}

\section{Stochastic optimization \label{sec:Stochastic-optimization}}

The goal of this section is to extend our proposed 
framework to stochastic constrained optimization where the objective  $f_0$ is an expectation function: 
\begin{equation}\label{eq:f0-expectation}
f_0(x)\coloneq\Ebb_{\xi\in\Xi} \sbra{F(x,\xi)}.
\end{equation}
Here, $F(x,\xi)$ is differentiable and $\xi$ denotes a random variable following a certain distribution $\Xi$. Directly evaluating either the objective  $f_0$ or its gradient can be computationally challenging due to the stochastic nature of the problem. To address this, we introduce the following additional assumptions.
\begin{assumption}
The information of $f_{0}$ is available via a stochastic first-order
oracle \textup{({\SFO})}. Given any input $x$ and a random sample $\xi$,
\textup{\SFO} outputs a stochastic gradient $\nabla F(x,\xi)$ such that
\begin{equation*}
\Ebb\big[\nabla F(x,\xi)\big] =\nabla f_0(x),\ \textup{and } \
\Ebb\bsbra{\norm{\nabla F(x,\xi)-\nabla f_0(x)}^{2}}\le\sigma^{2},
\end{equation*}
for some $\sigma\in(0,\infty)$. 
\end{assumption}

\subsection{Level constrained stochastic proximal gradient \label{subsec:Stochastic-Gradient-Descent}}
In Algorithm~\ref{alg:lc-sgd}, we present a stochastic variant of {\GD} for solving problem\ \ref{prob:main} with $f_0$ defined by~\eqref{eq:f0-expectation}.
As observed in \eqref{eq:Gk} and \eqref{eq:psi0-sgd}, the \SGD~method uses a mini-batch of random samples to estimate the true gradient in each iteration. It should be noted that the value $f_0(x^k)$ is presented in \eqref{eq:psi0-sgd} only for the ease of description, it is not required when solving \eqref{subproblem}.
\begin{algorithm}
	\begin{algorithmic}[1]
		\State {\bf Input: }$x^{0}$, $\eta^{0}<\eta$, $b_k$, $\delta^{k}$;
		\For{$k=0,1,\dots,K$}
			\State Sample a mini-batch $B_{k}$ of size $b_{k}$ and compute 
			\begin{equation}
				G^{k}=\frac{1}{b_{k}}\tsum_{i=1}^{m_{k}}\nabla F(x^{k},\xi_{i,k});\label{eq:Gk}
			\end{equation}
			\State Set $\psi_{0}^{k}(x)$ by 
			\begin{equation}
				\psi_{0}^{k}(x)\coloneq \big\langle G^{k},x\big\rangle+\frac{\gamma_{k}}{2}\norm{x-x^{k}}^{2}+\chi_{0}(x);\label{eq:psi0-sgd}
			\end{equation}
			\State For $i=1,\ldots,m$, set $\psi_{i}^{k}(x)$ by (\ref{eq:psik-1});
			\State Update $x^{k+1}$ by (\ref{subproblem}) and update $\eta^{k+1}$ by (\ref{eq:update-eta});
		\EndFor
		{\bf Output: }$x^{\hat{k}+1}$ where $\hat{k}$ is a random index sampled from $\{0,1,2,...,K\}$.
	\end{algorithmic}
	\caption{Level constrained stochastic proximal gradient ({\SGD}) \label{alg:lc-sgd}}
\end{algorithm}
Note that the proximal point method of \cite{boob2019proximal} does not need to account for a stochastic nonconvex problem separately since they solve corresponding stochastic convex subproblems using a ConEx method developed in their work. On the contrary, {\SGD} directly applies to stochastic nonconvex function constrained problems and convex subproblems are deterministic in nature. Hence, we need to develop asymptotic convergence and convergence rates for the \SGD~method separately.

Let $\zeta^{k}=G^{k}-\nabla f(x^{k})$ denote the error of gradient
estimation. We have 
\[
\Ebb\big[\norm{\zeta^{k}}^{2}\big]=\frac{1}{b_{k}^{2}}\tsum_{i=1}^{b_{k}}\Ebb_{\xi}\big[\norm{\nabla F(x^{k},\xi_{i,k})-\nabla f(x^{k})}^{2}\big]\le\frac{\sigma^{2}}{b_{k}}.
\]

The following proposition summarizes some important properties of the generated solutions of {\SGD}.
\begin{proposition}
\label{lem:sgd-suff-decrease}In Algorithm~\ref{alg:lc-sgd}, for
any $\beta_{k}\in(0,2\gamma_{k}-L_{0})$, we have 
\begin{equation}
\psi_{0}(x^{k+1})\le\psi_{0}(x^{k})-\frac{2\gamma_{k}-\beta_{k}-L_{0}}{2}\norm{x^{k+1}-x^{k}}^{2}+\frac{\norm{\zeta^{k}}^{2}}{2\beta_{k}}.\label{eq:sufficient-decrease-stochastic}
\end{equation}
Moreover, there exists a vector $\lambda^{k+1}\in\Rbb^m_+$ such that the KKT condition~\eqref{eq:kkt-subprob} \textup{(with $\psi_0^k$ defined in~\eqref{eq:psi0-sgd})} holds.
\end{proposition}

\begin{proof}
By the KKT condition, $x^{k+1}$ is the minimizer of $\Lcal_{k}(\cdot,\lambda^{k+1})$.
Therefore, we have 
\begin{equation}
\Lcal_{k}(x^{k+1},\lambda^{k+1})+\frac{\gamma_{k}+\langle\lambda^{k+1},L\rangle}{2}\norm{ x^{k+1}-x}^{2}\le\Lcal_{k}(x,\lambda^{k+1}),\quad\forall x\in\Xcal.\label{eq:middle-01}
\end{equation}
Placing $x=x^{k}$ in (\ref{eq:middle-01}) and using (\ref{eq:lagrange-1}),
we have
\begin{align}
& \inprod{ G^{k}}{x^{k+1}-x^{k}} + \chi_{0}(x^{k+1})-\chi_{0}(x^{k}) \nonumber\\
\le{} & \tsum_{i=1}^{m}\lambda_{i}^{k+1}\big[\psi_{i}^{k}(x^{k})-\eta_{i}^{k}\big]-\lambda_{i}^{k+1}\big[\psi_{i}^{k}(x^{k+1})-\eta_{i}^{k}\big] -\gamma_{k}\norm{x^{k+1}-x^{k}}^{2}\nonumber\\
\le{} & -\gamma_{k}\norm{x^{k+1}-x^{k}}^{2}, \label{eq:sgd-mid-01}
\end{align}
where the second inequality is due to the complementary slackness
$\lambda_{i}^{k+1}\big[\psi_{i}^{k}(x^{k+1})-\eta_{i}^{k}\big]=0$ and
strict feasibility $\lambda_{i}^{k+1}\big[\psi_{i}^{k}(x^{k})-\eta_{i}^{k}\big]=\lambda_{i}^{k+1}\big[\psi_{i}(x^{k})-\eta_{i}^{k}\big]<0.$
Using~(\ref{eq:sgd-mid-01}) and Lipschitz smoothness of $f_0$, we have
\begin{align}
\psi_{0}(x^{k+1}) %
 & \le f_{0}(x^{k})+\langle\nabla f_{0}(x^{k}),x^{k+1}-x^{k}\rangle+\frac{L_{0}}{2}\big\Vert x^{k+1}-x^{k}\big\Vert^{2}+\chi_{0}(x^{k+1})\nonumber \\
 & = f_{0}(x^{k})+\chi_{0}(x^{k+1})+\langle G^{k},x^{k+1}-x^{k}\rangle+\frac{L_{0}}{2}\big\Vert x^{k+1}-x^{k}\big\Vert^{2}-\langle\zeta^{k},x^{k+1}-x^{k}\rangle\nonumber \\
 & \le\psi_{0}(x^{k})-\frac{2\gamma_{k}-L_{0}}{2}\norm{x^{k+1}-x^{k}}^{2}-\langle\zeta^{k},x^{k+1}-x^{k}\rangle\nonumber \\
 & =\psi_{0}(x^{k})-\frac{2\gamma_{k}-\beta_k-L_{0}}{2}\norm{x^{k+1}-x^{k}}^{2}+\norm{\zeta^{k}}\cdot\norm{x^{k+1}-x^{k}}-\frac{\beta_k}{2}\norm{x^{k+1}-x^{k}}^{2}\nonumber \\
 & \le\psi_{0}(x^{k})-\frac{2\gamma_{k}-\beta_{k}-L_{0}}{2}\norm{x^{k+1}-x^{k}}^{2}+\frac{\norm{\zeta^{k}}^{2}}{2\beta_{k}}.\label{eq:bound-phi-01}
\end{align}
Above, the last inequality uses the fact $-\frac{a}{2} x^2+bx\le \frac{b^2}{2a}$ for any $x,b\in \Rbb, a>0$.
Showing the existence of the KKT condition follows a similar argument of proving part 2, Proposition~\ref{prop:lcgd-suff-descent}.
\end{proof}
We prove a technical result in the following lemma which plays a crucial role in proving dual boundedness.
\begin{lemma}\label{lem:a.s.convergence}
	Let $\{X_k\}_{k \ge 1}$ be a sequence of random vectors such that $\Ebb[X_k] = {0}$ for all $k \ge 1$ and $\tsum_{k=1}^\infty \sigma_k^2 \le M < \infty$ where $\sigma_k := \sqrt{\Ebb[\gnorm{X_k}{}{2}]}$. Then, $\lim_{k\raw \infty} X_k = 0$ almost surely \textup{(a.s.)}.
\end{lemma}
\begin{proof}
	We prove this result by contradiction.  If the result does not hold then there exists $\epsilon>0$ and $c > 0$ such that 
	\begin{equation}\label{eq:a.s.requirement}
		\prob\Brbra{\limsup_{k} \gnorm{X_k}{}{} \ge \epsilon} \ge c.
	\end{equation}
	However, due to Chebyshev's inequality, we have $\prob\rbra{\gnorm{X_k}{}{} \ge \epsilon} \le \tfrac{\sigma_k^2}{\epsilon^2}$. Since $\sigma_k^2$ is summable, there exists $k_0$ such that $\tsum_{k =k_0}^\infty \prob(\gnorm{X_k}{}{} \ge \epsilon) \le \tsum_{k =k_0}^\infty\tfrac{\sigma_k^2}{\epsilon^2} < c$. Therefore, we have
	\begin{align*}
		\prob\Brbra{\limsup_{k}\gnorm{X_k}{}{} \ge \epsilon} &= \prob\Brbra{\limsup_{k\ge k_0} \gnorm{X_k}{}{} \ge \epsilon}
		\le \tsum_{k =k_0}^\infty\prob\rbra{\gnorm{X_k}{}{} \ge \epsilon}<c.
	\end{align*}
	The above relation contradicts \eqref{eq:a.s.requirement}. Hence, we have $\lim_{k\raw \infty} X_k = 0$ a.s.
\end{proof}

In the following theorem, we present the main  asymptotic property of {\SGD}.
\begin{theorem}\label{thm:sgd-dual-bound}
 Suppose that $\tsum_{k=0}^\infty b_k^{-1}<\infty$,	then we have that $\lim_{k\raw \infty}\gnorm{x^{k+1}-x^k}{}{}=0$, a.s.
	Moreover, suppose that Assumption~\refeq{assu:y-bounded} holds, $\gamma_k<\infty$, $\beta_k$ is lower bounded and $2\gamma_k-\beta_k-L_0>0$,
then we have that 1) $\sup_{k} \gnorm{\lambda^k}{}{}<\infty$ a.s., and 2) all the limit points of Algorithm~\refeq{alg:lc-sgd} satisfy the KKT condition, a.s.\\
\end{theorem}
\begin{proof}
First, we fix notations. Let $(\Omega, \Fcal, \Pbb)$   be the probability space defined over the sampling minibatches $B_0,B_1,\ldots,$. Let  $\Ebb_k[\cdot]$  be the expectation conditioned on the sub $\sigma$-algebra generating $B_0,B_1,\ldots, B_{k-1}$.
Applying it to~\eqref{eq:sufficient-decrease-stochastic} gives 
\[
\Ebb_k\sbra{\psi_0(x^{k+1})} \le \psi_0\rbra{x^k} + \frac{\sigma^2}{2b_k\beta_k}.
\]
In view of the super-martingale convergence theorem~\cite{robbins1985a}, we have that
\begin{equation}\label{eq:sgd-asymp-01}
	\lim_{k\raw\infty}\psi_0(x^k) \ \text{exists and is finite a.s., when } \tsum_{k=0}^\infty (b_k\beta_k)^{-1}<\infty.
\end{equation}
Let $C_{k+1}=\tsum_{s=0}^k\frac{2\gamma_s-\beta_s-L_0}{2}\gnorm{x^{s+1}-x^s}{}{2}$ for $k\ge 0$ and $C_0=0$. We have
\[
\Ebb_k\sbra{\psi_0(x^{k+1})+C_{k+1}} \le \psi_0\rbra{x^k} +C_k + \frac{\sigma^2}{2b_k\beta_k}.
\]
Applying the super-martingale convergence theorem~\cite{robbins1985a}  again we can show that the limit of $\psi_0(x^k)+C_k$ exists a.s. Together with \eqref{eq:sgd-asymp-01} and lower-boundedness of $\beta_k$ and $2\gamma_k-\beta_k-L_0$, we have that $\lim_{k\raw\infty}\gnorm{x^{k+1}-x^k}{}{2}=0$, a.s.

Next, we prove the boundedness of $\gnorm{\lambda^k}{}{}$. Let us consider the events
\begin{align*}
\Ucal &=\Bcbra{\omega\in\Omega: \sup_k \gnorm{\lambda^k(\omega)}{}{}=\infty}, \ \Acal=\Bcbra{\omega\in\Omega: \sup_k\gnorm{G^{k}(\omega)}{}{}<\infty},\\
\Bcal &=\Bcbra{\omega\in\Omega: \lim_k \gnorm{x^{k+1}(\omega)-x^k(\omega)}{}{}=0}.
\end{align*}
We just argued $\prob(\Bcal)=1$. It is easy to see that if both conditions (i) $\prob(\Acal)=1$ and (ii) $\Ucal \subseteq \Acal^c\cup \Bcal^c$ hold, then 
 we have $\prob(\Ucal)\le \prob(\Acal^c)+\prob(\Bcal^c)=0$. Hence $\{\lambda^k\}$ is a bounded sequence a.s.

Since $\{b_k^{-1}\}$ is summable, we have $\tsum_{k=1}^\infty \Ebb[\gnorm{\zeta^k}{}{2}] \le \tfrac{\sigma^2}{b_k} < \infty$. Hence, using Lemma \ref{lem:a.s.convergence}, we have that $\lim_{k\raw \infty} \zeta^k = 0$ a.s.
Due to the boundedness of $\nabla f(x^k)$, we have that $G^k=\zeta^k+\nabla f(x^k)$ is bounded, a.s.

We prove condition (ii) by contradiction. Suppose that our claim fails.  We take an element $\omega\in \Ucal \cap (\Acal\cap \Bcal)$ and then pass to a  subsequence $\{j_k\}$ such that $\lim_{k\raw\infty} \gnorm{\lambda^{j_k}(\omega)}{}{}=\infty$. In the rest of the proof, we skip $\omega$ for brevity. Passing to another subsequence if necessary, let $\bar{x}$ be a limit point of $\{x^{j_k}\}$. By our presumption, $\bar{x}$ satisfies MFCQ.
Moreover, the KKT condition implies that
\begin{eqnarray}
& \inner{G^{j_k}}{x^{j_k+1}}+\chi_0(x^{j_k+1})+\frac{\gamma_{j_k}}{2}\gnorm{x^{j_k+1}-x^{j_k}}{}{2} +\inner{\lambda^{j_k+1}}{\psi^{j_k}(x^{j_k+1})} \nonumber \\
&\quad \le\inner{G^{j_k}}{x} +\chi_0(x)+\frac{\gamma_{j_k}}{2}\gnorm{x-x^{j_k}}{}{2} +\inner{\lambda^{j_k+1}}{\psi^{j_k}(x)},\quad\forall x\in\dom \chi_0.\label{eq:sgd-asymp-02}
\end{eqnarray}
Dividing both sides by $\gnorm{\lambda^{j_k+1}}{}{}$ gives 
\begin{eqnarray}
& \bsbra{\inner{G^{j_k}}{x^{j_k+1}}+\chi_0(x^{j_k+1})+\frac{\gamma_{j_k}}{2}\gnorm{x^{j_k+1}-x^{j_k}}{}{2}}/\gnorm{\lambda^{j_k+1}}{}{} +\inner{u^{k}}{\psi^{j_k}(x^{j_k+1})}\nonumber \\
&\quad \le \bsbra{\inner{G^{j_k}}{x} +\chi_0(x)+\frac{\gamma_{j_k}}{2}\gnorm{x-x^{j_k}}{}{2}}/\gnorm{\lambda^{j_k+1}}{}{} +\inner{u^{k}}{\psi^{j_k}(x)},\quad\forall x\in\dom \chi_0.\label{eq:sgd-asymp-03}
\end{eqnarray}
where we denote ${u}^{k}=\frac{\lambda^{j_k+1}}{\norm{\lambda^{j_k+1}}}$. 
Since $\{u^k\}$ is bounded, passing to a subsequence if needed, we have $\lim_{k\raw\infty}u^{k}=\bar{u}$.
 Since $\omega\in \Acal \cap\Bcal$, $\{G^{j_k}\}$ is bounded and $\{\frac{\gamma_{j_k}}{2}\gnorm{x^{j_k+1}-x^{j_k}}{}{2}\}$ converges to 0. Therefore,  taking $k\raw\infty$ on both sides of \eqref{eq:sgd-asymp-03}, we have 
 \begin{equation}
 \inner{\bar{u}}{\chi(\bar{x})}\le \binner{\bar{u}}{\psi(\bar{x})+\inner{\nabla\psi(\bar{x})}{x-\bar{x}}+\frac{L}{2}\norm{x-\bar{x}}^2+\chi(x)},\ \forall x\in \dom \chi_0. \label{eq:sgd-asymp-04}
 \end{equation}
Analogous to the proof of Theorem~\ref{prop:lcgd:bound-dual}, it is easy to show that $\bar{x}$ violates MFCQ, which however, contradicts Assumption~\ref{assu:y-bounded}. As a consequence of this argument, we have $\Ucal \subseteq \Acal^c \cup \Bcal^c$. Hence, we claim that the event $\sup_k \gnorm{\lambda^k}{}{}<\infty$ will happen a.s. and complete our proof of the boundedness condition.

Next, we prove asymptotic convergence to KKT solutions. For any random element $\omega$, let $\bar{x}(\omega)$ be any limit point of $\{x^k\}$. Passing to some subsequence if necessary, we assume that $\lim_{k\raw\infty}x^k=\bar{x}$ and $\lim_{k\raw\infty}\lambda^{k+1}=\bar{\lambda}$.
 \begin{align*}
 & \inprod{G^k+ \nabla f(x^{k})\lambda^{k+1}}{x^{k+1}-x}+\chi_{0}(x^{k+1})-\chi_{0}(x)+\langle\lambda^{k+1},\chi(x^{k+1})-\chi(x)\big\rangle\nonumber \\
 & \quad\le\frac{\gamma_k+\langle\lambda^{k+1},L\rangle}{2}\brbra{\norm{x-x^{k}}^{2}-\norm{x^{k+1}-x^{k}}^{2}-\norm{x-x^{k+1}}^{2}}.%
\end{align*}
Moreover, we have
\begin{align*}
	\inner{G^k}{x^{k+1}-x} 
	& = \inner{\nabla f_0(x^k)}{x^{k+1}-x}+\inner{\zeta^k}{x^{k+1}-x} \nonumber\\
	& = \inner{\nabla f_0(x^k)}{x^{k+1}-x} + \inner{\zeta^k}{x^{k+1}-x^k} + \inner{\zeta^k}{x^{k}-x}\nonumber \\
	& \ge \inner{\nabla f_0(x^k)}{x^{k+1}-x} - \norm{\zeta^k}\norm{x^{k+1}-x^k} + \inner{\zeta^k}{x^{k}-x}. %
\end{align*}
Combining the above two results, we have
\begin{align*}
 & \inprod{\nabla f_0(x^k)+ \nabla f(x^{k})\lambda^{k+1}}{x^{k+1}-x}+\chi_{0}(x^{k+1})-\chi_{0}(x)+\inprod{\lambda^{k+1}}{\chi(x^{k+1})-\chi(x)}\nonumber \\
 & \quad\le\frac{\gamma_k+\langle\lambda^{k+1},L\rangle}{2}\brbra{\norm{x-x^{k}}^{2}-\norm{x^{k+1}-x^{k}}^{2}-\norm{x-x^{k+1}}^{2}} + \norm{\zeta^k}\norm{x^{k+1}-x^k} + \inner{\zeta^k}{x-x^{k}}\nonumber\\
 & \quad \le \frac{\gamma_k+\langle\lambda^{k+1},L\rangle}{2}\brbra{\norm{x-x^{k}}^{2}-\norm{x-x^{k+1}}^{2}} + \frac{\norm{\zeta^k}^2}{2(\gamma_k+\langle\lambda^{k+1},L\rangle)} + \inner{\zeta^k}{x-x^{k}}.
\end{align*}
Taking $k\raw\infty$ in the above relation and noting that almost surely we have $\lim_{k\raw\infty}\zeta^k = 0$ and $\lim_{k\raw\infty} \norm{x^k-x^{k+1}}=0$, then
\[
 \inprod{\nabla f_0(\bar{x})+ \nabla f(\bar{x})\lambda^{k+1}}{\bar{x}-x}+\chi_{0}(\bar{x})-\chi_{0}(x)+\inprod{\lambda^{k+1}}{\chi(\bar{x})-\chi(x)}\le  0, \quad a.s. %
\]
Using an argument similar to the one in Theorem~\ref{prop:lcgd:bound-dual}, we can show that $\bar{x}$ is almost surely a KKT point. 
 \end{proof}

Our next goal is to develop the iteration complexity of Algorithm~\ref{alg:lc-sgd}. 
To achieve this goal, we need to assume that the dual is uniformly bounded, namely, condition~\eqref{eq:lambda-bound} holds for all the random events. While this condition is stronger than the almost sure boundedness of $\lambda^{k+1}$ shown by Theorem~\ref{thm:sgd-dual-bound}, it is indeed satisfied in many scenarios, e.g., when strong feasibility (Assumption \ref{ass:strong_feas}) holds or other scenarios described in \cite{boob2019proximal,Boob2020feasible}.
We present the main complexity result in the following theorem.
\begin{theorem}
\label{thm:main-sgd} Suppose that condition~\eqref{eq:lambda-bound} holds. Then, the
sequence $\{(x^{k+1},\lambda^{k+1})\}$ satisfies that
\begin{align}
 \tsum_{k=0}^{K}\frac{\alpha_k(2\gamma_{k}-\beta_{k}-L_{0})}{4(\gamma_{k}+L_{0}+2B\norm{L})^{2}}\norm{\partial_x\Lcal(x^{k+1},\lambda^{k+1})}_{-}^{2} 
 &\le  L_0D^2\alpha_K  + \tsum_{k=0}^{K}\Big(\frac{\alpha_k(2\gamma_{k}-\beta_{k}-L_{0})}{2(\gamma_{k}+L_{0}+2B\norm{L})^{2}}+\frac{\alpha_K}{2\beta_{k}}\Big)\norm{\zeta^{k}}^{2} \label{eq:lkbound-sgd-1} \\
 \tsum_{k=0}^{K}\alpha_k\big(2\gamma_{k}-\beta_{k}-L_{0}\big)\langle\lambda^{k+1},\vert\psi(x^{k+1})-\eta\vert\rangle  
 &\le  2 BL_0\norm{L}D^2\alpha_K+B\norm{L}  \tsum_{k=0}^{K}\frac{\alpha_K\norm{\zeta^{k}}^{2}}{\beta_{k}}  \nonumber\\
 & \quad +B\tsum_{k=0}^{K} \alpha_k\big(2\gamma_{k}-\beta_{k}-L_{0}\big)\norm{\eta-\eta^{k}}.\label{eq:csbound-sgd-1}
\end{align}
\end{theorem}

\begin{proof}
First, appealing to (\ref{eq:psi0-sgd}), (\ref{eq:lagrange-0}) and (\ref{eq:lagrange-1}),
we have
\begin{align}
\partial_x\Lcal_{k}(x^{k+1},\lambda^{k+1}) & =\partial_x\Lcal(x^{k+1},\lambda^{k+1})+\nabla f_{0}(x^{k})-\nabla f_{0}(x^{k+1})\nonumber \\
 & \quad +\tsum_{i=1}^{m}\lambda_{i}^{k+1}\big[\nabla f_{i}(x^{k})-\nabla f_{i}(x^{k+1})\big] +\big(\gamma_{k}+\langle\lambda^{k+1},L\rangle\big)\big(x^{k+1}-x^{k}\big)+\zeta^{k}.\nonumber %
\end{align}
It follows that
\begin{align}
& \setnorm{\partial_x\Lcal(x^{k+1},\lambda^{k+1})}\nonumber \\
\le{} & \norm{\nabla f_{0}(x^{k})-\nabla f_{0}(x^{k+1})}+\norm{\zeta^{k}}+\tsum_{i=1}^{m}\lambda_{i}^{k+1}\norm{\nabla f_{i}(x^{k})-\nabla f_{i}(x^{k+1})}+(\gamma_{k}+\langle\lambda^{k+1},L\rangle)\norm{x^{k+1}-x^{k}}\nonumber\\
\le{} & \big(\gamma_{k}+L_{0}+2\langle\lambda^{k+1}, L\rangle\big)\norm{x^{k+1}-x^{k}}+\norm{\zeta^{k}}.\nonumber %
\end{align}
In view of the above result and basic inequality $(a+b)^{2}\le2a^{2}+2b^{2}$,
we have 
\begin{equation}
\setnorm{\partial_x\Lcal(x^{k+1},\lambda^{k+1})}^{2}\le 2\big(\gamma_{k}+L_{0}+2B\norm{L}\big)^{2}\norm{x^{k+1}-x^{k}}^{2}+2\norm{\zeta^{k}}^{2}.\label{eq:middle-3}
\end{equation}
Let us denote an auxiliary sequence 
$C_k=
\begin{cases}
\psi_0(x^0) & k=0\\
\psi_0(x^k) -\tsum_{s=0}^{k-1}\frac{\norm{\zeta^s}^2}{2\beta_s} & k>0
\end{cases}.$
Proposition~\ref{lem:sgd-suff-decrease} implies that  
\begin{equation}\label{eq:middle-10}
\frac{2\gamma_{k}-\beta_{k}-L_{0}}{2}\norm{x^{k+1}-x^k}^2\le C_k - C_{k+1}.
\end{equation}
Putting this relation and (\ref{eq:middle-3})
together, we have
\begin{align}
\frac{2\gamma_{k}-\beta_{k}-L_{0}}{4(\gamma_{k}+L_{0}+2B\norm{L})^{2}}\setnorm{\partial_x\Lcal(x^{k+1},\lambda^{k+1})}^{2} 
 &\le C_k-C_{k+1}+ \frac{2\gamma_{k}-\beta_{k}-L_{0}}{2(\gamma_{k}+L_{0}+2B\norm{L})^{2}} \norm{\zeta^{k}}^{2}.\label{eq:middle-07}
\end{align}
Summing up (\ref{eq:middle-07}) over $k=0, 1,\ldots,K$ weighted by $\alpha_k$ leads to
\begin{align}
& \tsum_{k=0}^{K}\frac{\alpha_k(2\gamma_{k}-\beta_{k}-L_{0})}{4(\gamma_{k}+L_{0}+2B\norm{L})^{2}}\norm{\partial_x\Lcal(x^{k+1},\lambda^{k+1})}^{2}  \nonumber\\
  &\quad \le \tsum_{k=0}^K\alpha_k\rbra{C_k-C_{k+1}}  +\tsum_{k=0}^{K}\frac{\alpha_k(2\gamma_{k}-\beta_{k}-L_{0})}{2(\gamma_{k}+L_{0}+2B\norm{L})^{2}}\norm{\zeta^{k}}^{2}.\nonumber %
\end{align}
Moreover, since $\{C_k\}$ is monotonically decreasing, we have
\begin{align}
  \tsum_{k=0}^K\alpha_k\rbra{C_k-C_{k+1}}  
 &  \le \alpha_0 C_0+\tsum_{k=1}^K\rbra{\alpha_k-\alpha_{k-1}}C_k-\alpha_K C_{K+1} \nonumber\\
 &  \le \alpha_K \rbra{C_0 - C_{K+1}} \le L_0D^2\alpha_K +\alpha_K\tsum_{k=0}^K \frac{\norm{\zeta^k}^2}{2\beta_k}.\nonumber
\end{align}
Combining the above two relations leads to our first result~\eqref{eq:lkbound-sgd-1}.

For the second part, note that (\ref{eq:bound-cs-k}) remains valid
in the stochastic setting. Putting (\ref{eq:bound-cs-k}) and (\ref{eq:middle-10}) together, we obtain
\begin{equation*}
	(2\gamma_{k}-\beta_{k}-L_{0})\langle\lambda^{k+1},\vert\psi(x^{k+1})-\eta\vert \rangle 
	\le 2 B\norm{{L}} \rbra{C_k -C_{k+1}} + B(2\gamma_k-\beta_k-L_0)\norm{\eta-\eta^k}.
\end{equation*}
Multiplying both ends by $\alpha_k$ and then summing up the resulting terms over $k=0,\ldots, K$ gives (\ref{eq:csbound-sgd-1}).
\end{proof}
We next obtain more specific convergence rate by choosing the parameters
properly.
\begin{corollary}\label{cor:rate_lcsgd}
In Algorithm {\SGD}, set $\gamma_{k}=L_{0}$, $\beta_{k}=L_{0}/2$,
$\alpha_{k}=k+1$,  $b_{k}=K+1$ and $\delta^k = \frac{(\eta-\eta^0)}{(k+1)(k+2)}$. Then $x^{\hat{k}+1}$ is a randomized $\epsilon$ type-I KKT point with
\begin{equation*}
\epsilon =\frac{4}{K+2} \max \bcbra{8(L_0+B\norm{L})^{2} \brbra{D^2 + \frac{17\sigma^2}{16L_0^2}}, 2B\norm{L}D^2 +\frac{2B\norm{L}\sigma^2}{L_0^2}+\frac{B\norm{\eta-\eta^{0}}}{2}}.
\end{equation*}
\end{corollary}
\begin{proof}
Plugging in the value of $\gamma_k$, $\alpha_k$, $\beta_k$ in the relation~(\ref{eq:lkbound-sgd-1}) and taking expectation over all the randomness, we have
\begin{align*}
&\frac{L_0}{32(L_0+B\norm{L})^{2}}  \tsum_{k=0}^{K}(k+1)\Ebb\sbra{\setnorm{\partial_x\Lcal(x^{k+1},\lambda^{k+1})}    ^{2}} \nonumber\\
  & \quad \le L_0D^2(K+1) + \tsum_{k=0}^{K}\Big(\frac{L_0(k+1)}{16(L_{0}+B\norm{L})^{2}}+\frac{k+1}{L_0}\Big)\Ebb\sbra{\norm{\delta^{k}}^{2}} \nonumber\\
&\quad  \le  L_0D^2(K+1) + \frac{17\sigma^2}{16 L_0} (K+1).
\end{align*}
Moreover, due to the random sampling of $\hat{k}$, we have
\[
	\Ebb_{\hat{k}}\big[\setnorm{\partial_{x}\Lcal(x^{\hat k+1},\lambda^{\hat{k}+1})}^2\big]  =\frac{2}{(K+1)(K+2)}\tsum_{k=0}^{K}(k+1)\setnorm{\partial_{x}\Lcal(x^{k+1},\lambda^{k+1})}^{2}.
\]
Combining the above two results, we have
\[
	\Ebb\bsbra{\setnorm{\partial_x\Lcal(x^{\hat{k}+1},\lambda^{\hat{k}+1})}^{2}}  \le \frac{32(L_0+B\norm{L})^{2}}{K+2} \Brbra{D^2 + \frac{17\sigma^2}{16L_0^2}}.
\]
Second, plugging in the values of $\gamma_k$, $\beta_k$ and $\delta^k$ in (\ref{eq:csbound-sgd-1}), we have 
\begin{equation}
 \frac{L_{0}}{2}\tsum_{k=0}^{K}(k+1)\langle\lambda^{{k+1}},\vert\psi(x^{{k}+1})-\eta\vert\rangle 
 \le  2BL_0 \norm{L}D^2(K+1) + 2B\norm{L}\frac{\sigma^2(K+1)}{L_0}+\frac{BL_0\norm{\eta-\eta^{0}}(K+1)}{2}\label{eq:sgd-mid-02}.
\end{equation}
It then follows from (\ref{eq:sgd-mid-02}) and the definition of $\hat{k}$ that
\[\Ebb\big[\langle\lambda^{\hat{k}+1},\vert\psi(x^{\hat{k}+1})-\eta\vert\rangle\big]  \le \frac{4}{K+2}\Big\{2B\norm{L}D^2 +\frac{2B\norm{L}\sigma^2}{L_0^2}+\frac{B\norm{\eta-\eta^{0}}}{2}\Big\}.  
\]
This completes our proof.
\end{proof}
\begin{remark}
In order to obtain some $\epsilon$-error in satisfying the type I KKT condition,
	  {\SGD} requires a number of $\Ocal(\vep^{-2})$  calls to the  {\SFO}, which matches the complexity bound of stochastic gradient descent for unconstrained nonconvex optimization~\cite{saeed-lan-nonconvex-2013}.
	 Moreover,  due to minibatching, {\SGD} obtains an even better $\Ocal(\vep^{-1})$ complexity in the number of evaluations of $f_i(x)$ and $\nabla f_i(x)$ ($i\in[m]$).
 \end{remark}
\subsection{Level constrained stochastic variance reduced gradient descent\label{subsec:Variance-Reduced-Gradient}}

We consider the finite sum problem:
\begin{equation}
f_{0}(x)=\frac{1}{n}\tsum_{i=1}^{n}F(x,\xi_{i}),\label{eq:finite-sum}
\end{equation}
where each $F(x,\xi_{i})$ is Lipschitz smooth with the parameter $L_{0}$, $i=1,2,\ldots, n$.
To further improve the convergence performance in this setting,
we present a new variant of the stochastic gradient method by extending the
stochastic variance reduced gradient descent to the constrained setting. 

\begin{algorithm}
	\begin{algorithmic}[1]
		\State	{\bf Input: }$x^{0}$, $x^{-1}$, $\eta_{0}<\eta$, $T$;
		\For{$k=0,1,\ldots,K$}
			\If{$k\%T==0$}
			\State $G^{k}=\nabla f_{0}(x^{k})$;
			\Else
			\State Sample a mini-batch $B_{k}$ of size $b$ uniformly at random and compute 
			\begin{equation}
				G^{k}=\frac{1}{b}\tsum_{i\in B_{k}}\big[\nabla F(x^{k},\xi_{i})-\nabla F(x^{k-1},\xi_{i})\big]+G^{k-1};\label{eq:Gk-svrg}
			\end{equation}
			\EndIf
			\State Set $\psi_{0}^{k}(x)$ by \eqref{eq:psi0-sgd} and set $\psi_i^{k}(x)$ by \eqref{eq:psik-1} for $i\in[m]$;
			\State Update $x^{k+1}$ by (\ref{subproblem}) and update $\eta^{k+1}$ by (\ref{eq:update-eta});
		\EndFor
	\end{algorithmic}
\caption{Level constrained stochastic variance-reduced gradient descent ({\SVRG})
\label{alg:lc-svrg}}
\end{algorithm}

We present the level constrained stochastic variance-reduced gradient
descent (\SVRG) in Algorithm~\ref{alg:lc-svrg}, which extends the nonconvex
variance reduced mirror descent (see \cite{lan2020first}) to handle
nonlinear constraint. Algorithm~\ref{alg:lc-svrg} can be viewed
as a double-loop algorithm in which the outer loop computes the full
gradient $\nabla f(x^{k})$ once every $T$ iterations and the nested
loop performs stochastic proximal gradient updates based on an unbiased estimator of
the true gradient. In this view, we let $k$ indicate the $t$-th iteration
at the $r$-th epoch, for some values $t$ and $r$. Then we use $k$
and $(r,t)$ interchangeably throughout the rest of this section.
We keep the notation $\zeta^{k}$ (or $\zeta^{(r,j)}$) for expressing
$G^{k}-\nabla f(x^{k})$ and note that $\zeta^{(r,0)}=0$.

Our next goal is to develop some iteration complexity results of {\SVRG}. We skip the asymptotic analysis since it is similar to that of {\SGD}. The following Lemma (see~\cite[Lemma 6.10]{lan2020first}) presents a key insight of Algorithm~\ref{alg:lc-svrg}
that the variance is controlled by the point distances. We provide
proof for completeness. 
\begin{lemma}
\label{lem:svrg-noise-bound}In Algorithm~\refeq{alg:lc-svrg}, $G^{k}$
is an unbiased estimator of $\nabla f_0(x^{k})$. Moreover, Let $(r,t)$
correspond to $k$. If $t>0$, then we have 
\[
\Ebb\big[\norm{\zeta^{(r,t)}}^{2}\big]\le\frac{L_{0}^{2}}{b}\tsum_{i=0}^{t-1}\Ebb\big[\norm{x^{(r,i+1)}-x^{(r,i)}}^{2}\big].
\]
\end{lemma}

\begin{proof}
We prove the first part by induction. When $k=0$, we have $G^{0}=\nabla f_0(x^{0})$.
Then for $k>0$, if $k\%T==0$, we have $G^{k}=\nabla f(x^{k})$ by
definition. Otherwise, we have 
\[
\Ebb_{k}\big[G^{k}\big]=\nabla f(x^{k})-\nabla f(x^{k-1})+G^{k-1}=\nabla f_0(x^{k})
\]
by induction hypothesis $\Ebb_{k-1}\big[G^{k-1}\big]=\nabla f(x^{k-1})$.

Next, we estimate the variance of the stochastic gradient. Appealing to~\eqref{eq:Gk-svrg},
we have
\begin{align*}
\Ebb_{k}\big[\norm{\zeta^{k}}^{2}\big] & =\Ebb\big[\big\Vert\frac{1}{b}\tsum_{i\in B_{k}}\big[\nabla F(x^{k},\xi_{i})-\nabla F(x^{k-1},\xi_{i})\big]+G^{k-1}-\nabla f(x^{k})\big\Vert^{2}\big]\\
 & =\Ebb\big[\big\Vert\frac{1}{b}\tsum_{i\in B_{k}}\big[\nabla F(x^{k},\xi_{i})-\nabla F(x^{k-1},\xi_{i})\big]-\big[\nabla f(x^{k})-\nabla f(x^{k-1})\big]+\zeta^{k-1}\big\Vert^{2}\big]\\
 & =\Ebb\big\Vert\frac{1}{b}\tsum_{i\in B_{k}}\big[\nabla F(x^{k},\xi_{i})-\nabla F(x^{k-1},\xi_{i})-\nabla f(x^{k})+\nabla f(x^{k-1})\big]\big\Vert^{2}+\big\Vert\zeta^{k-1}\big\Vert^{2}\\
 & \le\frac{1}{b^{2}}\tsum_{i\in B_{k}}\Ebb_{\xi}\norm{\nabla F(x^{k},\xi)-\nabla F(x^{k-1},\xi)}^{2}+\big\Vert\zeta^{k-1}\big\Vert^{2}\\
 & \le\frac{L_{0}^{2}}{b}\norm{x^{k}-x^{k-1}}^{2}+\big\Vert\zeta^{k-1}\big\Vert^{2},
\end{align*}
where the third equality uses the independence of $B_{k}$ and $\zeta^{k-1}$, the first inequality uses the bound $\Var(x)\le\Ebb\norm{x}^{2}$,
and the second inequality uses the Lipschitz smoothness of $F(\cdot,\xi)$.
Taking expectation over all the randomness generating $B_{(r,1)},B_{(r,2)},\ldots,B_{(r,t)}$,
we have
\[
\Ebb\big[\norm{\zeta^{k}}^{2}\big]\le\frac{L_0^{2}}{b}\tsum_{i=1}^{t}\Ebb\big[\norm{x^{(r,i)}-x^{(r,i-1)}}^{2}\big].
\]
\end{proof}
The next Lemma shows that the generated solutions satisfy a property of  sufficient descent on expectation.
\begin{lemma}
\label{lem:suff-decrease-svrg}Assume that $\gamma_{k}=\gamma$ and
$\beta_{k}=\beta$ and
$\tilde{L} \coloneq \frac{2\gamma- \beta- L_{0}}{2}-\frac{L_{0}^{2}(T-1)}{2\beta b}>0.$
Then we have
\begin{equation}
\tilde{L}\tsum_{j=0}^{t}\Ebb\big[\norm{x^{(r,j+1)}-x^{(r,j)}}^{2}\big]\le\Ebb\big[\psi_{0}(x^{(r,0)})\big]-\Ebb\big[\psi_{0}(x^{(r,t+1)})\big],\quad 0\le t<T \label{eq:suff-decrease-svrg}.
\end{equation}
\end{lemma}

\begin{proof}
In view of (\ref{lem:sgd-suff-decrease}), at the $j$-th iteration
of the $r$-th epoch, we have 
\begin{align*}
\psi_{0}(x^{(r,j+1)}) & \le\psi_{0}(x^{(r,j)})-\frac{2\gamma-\beta-L_{0}}{2}\norm{x^{(r,j+1)}-x^{(r,j)}}^{2}+\frac{\norm{\zeta^{(r,j)}}^{2}}{2\beta}. %
\end{align*}
Summing up the above result over $j=0,1,2,...,t$ ($t<T$) and using Lemma~\ref{lem:svrg-noise-bound},
we have 
\begin{align*}
& \frac{2\gamma-\beta-L_{0}}{2}\tsum_{j=0}^{t}\Ebb\big[\norm{x^{(r,j+1)}-x^{(r,j)}}^{2}\big]\nonumber\\
\le{} & \Ebb\big[\psi_{0}(x^{(r,0)})\big]-\Ebb\big[\psi_{0}(x^{(r,t+1)})\big]+\frac{L_{0}^{2}}{2\beta b}\tsum_{j=0}^{t}\tsum_{i=0}^{j-1}\Ebb\big[\norm{x^{(r,i+1)}-x^{(r,i)}}^{2}\big]\\
\le{} & \Ebb\big[\psi_{0}(x^{(r,0)})\big]-\Ebb\big[\psi_{0}(x^{(r,t+1)})\big]+\frac{L_{0}^{2}t}{2\beta b}\tsum_{i=0}^{t-1}\Ebb\big[\norm{x^{(r,i+1)}-x^{(r,i)}}^{2}\big] \\
\le{} & \Ebb\big[\psi_{0}(x^{(r,0)})\big]-\Ebb\big[\psi_{0}(x^{(r,t+1)})\big]+\frac{L_{0}^{2}(T-1)}{2\beta b}\tsum_{i=0}^{t-1}\Ebb\big[\norm{x^{(r,i+1)}-x^{(r,i)}}^{2}\big].
\end{align*}
Here we use $\tsum_{j=0}^{-1}\cdot=0$. 
\end{proof}
We present the main convergence property of Algorithm~\ref{alg:lc-svrg}
in the next theorem.
\begin{theorem}
\label{thm:main-svrg}Suppose that condition~\eqref{eq:lambda-bound} and assumptions of Lemma~\refeq{lem:suff-decrease-svrg} hold, $b\ge2T$ and $K=r_0T+j_0$ for some $r_0, j_0 \ge 0$. Let $\{\alpha_k\}$ be a non-decreasing sequence and $\{\alpha_{(r,j)}\}$ be its equivalent form in $(r,j)$ notations. Suppose that $\alpha_{(r,j)}=\alpha_{(r,0)}$ for $j=1,2,...,T-1$. Then we have 
\begin{align}
\tsum_{k=0}^{K}\alpha_k\Ebb\bsbra{\setnorm{\partial_x\Lcal(x^{k+1},\lambda^{k+1})}^{2}} & \le 8\til{L}^{-1}L_0(\gamma+L_{0}+B\norm{L})^{2}D^2 \alpha_{(r_0,0)},\label{eq:sum-lag-svrg}\\
\tsum_{k=0}^{K}\alpha_k\Ebb\big[\langle\lambda^{k+1},\vert\psi(x^{k+1})-\eta\vert\rangle\big] & \le B\tsum_{k=0}^K \alpha_k \gnorm{\eta - \eta_k}{}{} + B\til{L}^{-1}\gnorm{L}{}{}L_0D^2 \alpha_{(r_0,0)}.\label{eq:sum-cs-svrg}
\end{align}
Moreover, if we take $T=\lceil\sqrt{n}\rceil$,
$b=8T,\gamma=L_{0},\beta=L_{0}/2,$ and $\alpha_k= T\,\lfloor k/T\rfloor+1$, and set $\delta^{k}=\frac{\eta-\eta^{0}}{(k+1)(k+2)}$.
Then $x^{\hat{k}+1}$ is a randomized Type-I $\epsilon$-KKT point with
\begin{equation}\label{eq:svrg-epsi-kkt}
	\epsilon = \frac{K+1}{(K-T+1)^2}\max\bcbra{128(2L_{0}+B\norm{L})^{2}D^{2}, B\gnorm{\eta-\eta^0}{}{} + 16B\gnorm{L}{}{}D^2}.
\end{equation}
\end{theorem}

\begin{proof}
First, using Lemma~\ref{lem:svrg-noise-bound} and the assumption that $b\ge2T$, for any $t\le T-1$ we have
\begin{align}
\tsum_{j=0}^{t}\Ebb\bsbra{\norm{\zeta^{(r,j)}}^{2}} 
 & \le\frac{L_{0}^{2}}{b}\tsum_{j=0}^{t}\tsum_{i=0}^{j-1}\Ebb\big[\norm{x^{(r,i+1)}-x^{(r,i)}}^{2}\big]\nonumber \\
 & \le\frac{L_{0}^{2}t}{b}\tsum_{j=0}^{t-1}\Ebb\big[\norm{x^{(r,i+1)}-x^{(r,i)}}^{2}\big]\nonumber \\
 & \le \frac{L_{0}^{2}}{2}\tsum_{j=0}^{t-1}\Ebb\big[\norm{x^{(r,i+1)}-x^{(r,i)}}^{2}.\label{eq:middle-12}
\end{align}
Note that (\ref{eq:middle-3}) still holds. Therefore, 
combining (\ref{eq:middle-3})  and (\ref{eq:middle-12}) leads to
\begin{align*}
\tsum_{j=0}^{t}\Ebb\big[\setnorm{\Lcal(x^{(r,j+1)},\lambda^{(r,j+1)})}^{2}\big] 
& \le 2\big(\gamma+L_{0}+2B\norm{L}\big)^{2}\tsum_{j=0}^{t}\Ebb\big[\norm{x^{(r,j+1)}-x^{(r,j)}}^{2}\big]+2\tsum_{j=0}^{t}\Ebb\big[\norm{\zeta^{(r,j)}}^{2}\big] \\
 & \le 8(\gamma+L_{0}+B\norm{L})^{2}\tsum_{j=0}^{t}\Ebb\big[\norm{x^{(r,j+1)}-x^{(r,j)}}^{2}\big]. %
\end{align*}
It then follows from Lemma~\ref{lem:suff-decrease-svrg} that
\begin{align}\label{eq:middle-11}
	\tsum_{j=0}^{t}\Ebb\big[\setnorm{\partial_x\Lcal(x^{(r,j+1)},\lambda^{(r,j+1)})}^{2}\big] 
	& \le 8\til{L}^{-1}(\gamma+L_{0}+B\norm{L})^{2}\Ebb\big[\psi_{0}(x^{(r,0)})-\psi_{0}(x^{(r,t+1)})\big]. 
\end{align}
Let $K=r_0T+j_0$. Summing up the above inequality weighted by $\alpha_k$ and exchanging the notation                                                                                                             $\alpha_k \leftrightarrow \alpha_{(r,j)}$, then we have 
\begin{align}
&\tsum_{k=0}^{K}\alpha_k\Ebb\big[\setnorm{\partial_x\Lcal(x^{k+1},\lambda^{k+1})}^{2}\big] \nonumber\\
&\quad = \tsum_{r=0}^{r_0-1}\tsum_{j=0}^{T-1}\alpha_{(r,j)}\Ebb \bsbra{\setnorm{\partial_x\Lcal(x^{(r,j+1)},\lambda^{(r,j+1)})}^{2}}
+ \tsum_{j=0}^{j_0}\alpha_{(r_0,j)}\Ebb \bsbra{\setnorm{\partial_x\Lcal(x^{(r_0,j+1)},\lambda^{(r_0,j+1)})}^{2}} \nonumber\\
&\quad \le 8\til{L}^{-1}(\gamma+L_{0}+B\norm{L})^{2}\Bcbra{\tsum_{r=0}^{r_0-1}\alpha_{(r,0)} \Ebb\sbra{\psi_0(x^{(r,0)})-\psi_0(x^{(r+1,0)})}+\alpha_{(r_0,0)}\Ebb\sbra{\psi_0(x^{(r_0,0)}) - \psi_0(x^{(r_0,j_0+1)})}} \nonumber\\
&\quad \le 8\til{L}^{-1}(\gamma+L_{0}+B\norm{L})^{2}\alpha_{(r_0,0)}\Ebb\sbra{\psi_0(x^{(0,0)})- \psi_0(x^{(r_0,j_0+1)})}\nonumber \\
& \quad \le 8\til{L}^{-1}L_0(\gamma+L_{0}+B\norm{L})^{2}D^2 \alpha_{(r_0,0)}.
\end{align}
Above, the first inequality applies \eqref{eq:middle-11} and uses $x^{(r,T)}=x^{(r+1,0)}$
while the second inequality uses the monotonicity of $\{\psi_0(x^k)\}$ and an argument similar to \eqref{eq:weighted-square-sum}.

The second part is similar to the argument of Theorem~\ref{thm:main-sgd}.
Particularly, combining \eqref{eq:bound-cs-k} and \eqref{eq:suff-decrease-svrg} gives
\begin{equation}
	\tsum_{j=0}^t \Ebb \inner{\lambda^{(r,j+1)}}{\vert \psi_i(x^{(r,j+1)})-\eta\vert} \le B\til{L}^{-1}\gnorm{L}{}{}\Ebb\sbra{\psi_0(x^{(r,0)}) - \psi_0 (x^{(r,t+1)}) }+B\tsum_{j=0}^t\bnorm{\eta-\eta^{(r,j+1)}}
\end{equation}
Consequently, using the above relation and an argument similar to show \eqref{eq:weighted-square-sum}, we deduce
\begin{align*}
&\tsum_{k=0}^{K} \alpha_k\Ebb\,\langle\lambda^{k+1},\vert\psi(x^{k+1})-\eta\vert\rangle  \nonumber \\
& \quad = \tsum_{r=0}^{r_0-1}\tsum_{j=0}^{T-1} \alpha_{(r,j)} \Ebb\,\langle\lambda^{(r,j+1)},\vert\psi(x^{(r,j+1)})-\eta\vert\rangle 
	+  \tsum_{j=0}^{j_0}\alpha_{(r_0,j)} \Ebb\,\langle\lambda^{(r_0,j+1)},\vert\psi(x^{(r_0,j+1)})-\eta\vert\rangle 
	\nonumber \\
& \quad = B\tsum_{k=0}^K \alpha_k \gnorm{\eta - \eta_k}{}{} \nonumber\\
	&\quad\quad\quad +B\til{L}^{-1}\gnorm{L}{}{}\Ebb\Bcbra{ \tsum_{r=0}^{r_0-1} \alpha_{(r,0)} \sbra{\psi_0(x^{(r,0)})- \psi_0(x^{(r+1,0)})} + \alpha_{(r_0,0)}\sbra{\psi_0(x^{(r_0,0)}) - \psi_0(x^{(r_0,j_0+1)})} }\nonumber\\
&\quad \le B\tsum_{k=0}^K \alpha_k \gnorm{\eta - \eta_k}{}{} + B\til{L}^{-1}\gnorm{L}{}{}L_0D^2 \alpha_{(r_0,0)}.
\end{align*}
Therefore, we complete the proof of \eqref{eq:sum-lag-svrg} and \eqref{eq:sum-cs-svrg}.

Using the provided parameter setting, we have $\tilde{L}=\frac{2\gamma-L_{0}-\beta}{2}-\frac{L_{0}^{2}(T-1)}{2\beta b}\ge\frac{L_{0}}{4}-\frac{L_{0}}{8}=\frac{L_{0}}{8}$. Moreover, since $\alpha_k=T\,\lfloor k/T\rfloor  + 1$,  we have $\alpha_{k}\le T\cdot k / T+1\le k+1$. It is easy to check 
\[\tsum_{k=0}^K\alpha_k= T + \tsum_{k=T}^K \paran[\big]{\big\lfloor \frac{k}{T}\big\rfloor T + 1} \ge T+  \tsum_{k=T}^K  \bracket[\big]{\bracket[\big]{\frac{k}{T}-1} T+1}  \ge \frac{(K-T+1)^2}{2}.\]
\end{proof}

\begin{remark}
It is interesting to compare the performance of {\SVRG} with the other level constrained first-order methods in the finite sum setting~\eqref{eq:finite-sum}. 
Similar to {\GD}, {\SVRG} runs for $\Ocal(\vep^{-1})$ iterations to compute Type-I $\epsilon$-KKT point.
Moreover, {\SVRG} has an appealing feature that the number of stochastic gradient $\nabla F(x,\xi)$ computed  can be significantly reduced for a large value of $n$.
Specifically,   Algorithm~\ref{alg:lc-svrg} requires
 a full gradient $\nabla f_0(x)$ every $T$ iterations, which contributes $N_{1}=\Ocal\big(n\big\lceil\frac{K}{T}\big\rceil\big)=\Ocal(\sqrt{n}K)$
stochastic gradient computations. During the other iterations,
Algorithm~\ref{alg:lc-svrg} invokes a batch of size $b=\Ocal(T)$
each time, exhibiting a complexity of $N_{2}=\Ocal\big(bK\big)=\Ocal(\sqrt{n}K)$.
Therefore, the total number of stochastic gradient computations is
$N=N_{1}+N_{2}=\Ocal\big(\sqrt{n}K\big).$ This is better than the $O(nK)$ stochastic gradients needed by {\GD}. Moreover, it is better than the bound $O(K^2)$ of {\SGD} when $K$ is at an order larger than $\Omega(\sqrt{n})$, which corresponds to a higher accuracy regime of $\epsilon \ll \tfrac{1}{\sqrt{n}}$. 
The complexities of all the proposed algorithms for getting some $\vep$-KKT solutions are listed in Table\ \ref{tab:count-sfo}.
\end{remark}
\begin{remark}
While we mainly discuss the finite-sum objective~\eqref{eq:finite-sum}, it is possible to extend the variance reduction technique to handle the expectation-based objective~\eqref{eq:f0-expectation} and improve the $\Ocal(\vep^{-2})$ bound of \SGD{} to $\Ocal(\vep^{-3/2})$. To achieve this goal, we  impose an additional assumption  that $F(x,\xi)$ is $L_0$-Lipschitz smooth for each $\xi$ in the support set. We choose to omit a detailed discussion on this particular extension, as the technical development for this can be readily derived from the arguments in Sec. 6.5.2.~\cite{Lan19} and  our previous analysis.
\end{remark}

\section{Smooth optimization of nonsmooth constrained problems\label{sec:nonsmooth}}

In this section, we consider the constrained problem~\eqref{prob:main} with nonsmooth objective and nonsmooth constraint functions. We assume that $f_i$ ($i\in \{0,1,...,m\}$) exhibits a difference-of-convex (DC) structure $f_i(x) := g_i(x) - h_i(x)$:  1) $h_i$ is an \added{$L_{h_i}$}-Lipschitz-smooth convex function and 2) $g_i$ is a structured nonsmooth convex function:
\begin{align*}
	g_i(x) = \max_{y_i \in \Ycal_i} \inprod{A_ix}{y_i} - p_i(y_i), 
\end{align*} 
where $A_i \in \Rbb^{a_i\times n}$ is a linear mapping, $\Ycal_i \subset \Rbb^{a_i}$ is a convex compact set and $p_i:\Ycal_i \to \Rbb$ is a convex continuous function. 
In view of such a nonsmooth structure, we can not simply apply the {\GD} method, as the crucial quadratic upper-bound on $f_i(x)$ does not hold in the nonsmooth cases. 
 However, as pointed out by Nesterov~\cite{Nesterov2005smooth}, the nonsmooth convex function $g_i$ can be closely approximated by a smooth convex function. Let us denote $\wh{y}_i \coloneq \argmin_{y_i \in \Ycal_i} \gnorm{y_i}{}{}$, $D_{\Ycal_i} \coloneq \max_{y_i \in \Ycal_i} \gnorm{y_i-\wh{y}_i}{}{}$ and define the approximation function
\[g_i^{\beta_i}(x) := \max_{y_i \in \Ycal_i} \inprod{A_ix}{y_i} - p(y_i) - \tfrac{\beta_i}{2}\gnorm{y_i-\wh{y}_i}{}2, \  f^{\beta_i}_{i}(x) := g^{\beta_i}_{i}(x) - h_i(x),\ \text{where } \beta_i > 0.\]
Given some properly chosen smoothing parameter $\beta_i$, we propose to apply {\GD} to solve the following smooth approximation problem:
\begin{mini}
	{x}{\psi_0^{\beta_0}(x)=f_0^{\beta_0}(x)+\chi_0(x)}{}{}\label{prob:smooth_main}
	\addConstraint{\psi_i^{\beta_i}(x)=f_i^{\beta_i}(x)+\chi_{i}(x)}{\le \eta_i\quad i=1,\dots,m}.
\end{mini}

Prior to the analysis of our algorithm, we need to develop some properties of the smooth function $f_i^{\beta_i}$. We first present a key Lemma which builds some important connection between the quadratic approximation of smooth function and Lipschitz smoothness. The proof is left in Appendix \ref{appx:tight_Lipschitz_const}.
\begin{lemma}\label{lem:tight_lipschitz-const}
	Suppose $p(\cdot)$ is continuously differentiable function satisfying 
	\begin{equation}\label{eq:lower-upper-curvature}
		-\tfrac{\mu}{2} \gnorm{x-y}{}{2} \le p(x) -p(y) -\inprod{\grad p(y)}{x-y} \le \tfrac{L}{2}\gnorm{x-y}{}{2},
	\end{equation}
	for all $x,y$. Then, $p(\cdot)$ satisfies 
	\begin{equation}\label{eq:tight-Lipschitz-const}
		\gnorm{\grad p(x) - \grad p(y)}{}{} \le \max\{L ,  \mu\} \gnorm{x-y}{}{}.
	\end{equation}
\end{lemma}
In smooth approximation, it is shown in \cite{Nesterov2005smooth} that $g_i^{\beta_i}$ is a Lipschitz smooth function and it approximates the function value of $g_i$ with some $\Ocal(\beta_i)$-error: 
\begin{align}\label{eq:int_rel5}
	& g_i^{\beta_i}(x)  \le g_i(x) \le g_i^{\beta_i}(x) + \tfrac{\beta_i D_{\Ycal_i}^2}{2}, \qquad \forall x, \\
	& \gnorm{\nabla g_i^{\beta_i}(x) - \nabla g_i^{\beta_i}(z)}{}{}  \le L^{\beta_i}_{g_i} \gnorm{x-z}{}{},\quad \forall x, z,\quad L^{\beta_i}_{g_i} :=\tfrac{\gnorm{A_i}{}2}{\beta_i}.
\end{align} 
Similar properties of  $f_i^{\beta_i}$ are developed in the following proposition. 
\begin{proposition}\label{prop:upp-low-curv-nonsmooth}
We have the following properties about the approximation function $f_i^{\beta_i}$ $(\beta_i>0)$.
\begin{enumerate}
	\item Let $ \bar{\beta}_i \in [0, \beta_i]$, then we have 
	\begin{equation} \label{eq:apprx-bound}
	f_i^{\beta_i}(x) \le  f_i^{\bar{\beta}_i}(x) \le f_i^{\beta_i}(x) + \tfrac{(\beta_i-\bar{\beta}_i) D_{\Ycal_i}^2}{2}.
\end{equation} 
	\item $f_i^{\beta_i}(x)$ has upper curvature $L^{\beta_i}_{g_i}$ and negative lower curvature $-L_{h_i}$, namely,
	\begin{align}
		f_i^{\beta_i}(x) & \le f_i^{\beta_i}(y)+\langle\nabla 	f_i^{\beta_i}(y),x-y\rangle+\frac{L^{\beta_i}_{g_i}}{2}\|x-y\|^{2},\label{eq:upper-curve} \\
		f_i^{\beta_i}(x) & \ge f_i^{\beta_i}(y)+\langle\nabla f_i^{\beta_i}(y),x-y)-\frac{L_{h_i}}{2}\|x-y\|^{2}.\label{eq:lower-curve}
	\end{align}

	\item $f_i^{\beta_i}$ is  Lipschitz smooth with modulus $L_i^{\beta_i}  := \max\{ L_{g_i}^{\beta_i}, L_{h_i} \}$. Namely, for any $x,y$, we have 
\begin{equation}
\gnorm{ \grad f_i^\beta(x) - \grad f_i^\beta(y)}{}{} \le L_i^{\beta_i} \gnorm{x-y}{}{}.\label{eq:tight-Lip-const}
\end{equation}

\end{enumerate}
	\end{proposition}
\begin{proof}
Part 1. If $\bar{\beta}<\beta$, then by definition we have $f_i^{\bar{\beta}_i}(x) \ge  f_i^{\beta_i}(x)$. On the other hand, using the boundedness of $\Ycal_i$, we have
\begin{align*}
	f_i^{\beta_i}(x) & = \max_{y_i\in \Ycal_i}\, \inprod{A_ix}{y_i} - p(y_i)-\frac{\bar{\beta}_i}{2}\gnorm{y_i-\hat{y}_i}{}{2} -\frac{\beta_i-\bar{\beta}_i}{2}\gnorm{y_i-\hat{y}_i}{}{2} - h(x)\\
	& \ge \max_{y_i\in \Ycal_i}\, \inprod{A_ix}{y_i} - p(y_i)-\frac{\bar{\beta}_i}{2}\gnorm{y_i-\hat{y}_i}{}{2} - h(x) -  \frac{\beta_i-\bar{\beta}_i}{2} D_{\Ycal_i}^2 \\
	& = f_i^{\bar{\beta}_i}(x) - \frac{\beta_i-\bar{\beta}_i}{2} D_{\Ycal_i}^2.
\end{align*}
Combining the above two results gives the desired inequality.

Part 2. Since $g^{\beta_i}_{i}$ and $h^i$ are both convex and smooth functions, we have
	\begin{align*}
		g^{\beta_i}_i(x_1) &\le g^{\beta_i}_i(x_2) + \inprod{\nabla g^{\beta_i}_{i}(x_2)}{x_1-x_2} + \tfrac{L^{\beta_i}{g_i}}{2}\gnorm{x_1-x_2}{}2, \\
		h_i(x_1) &\le h_i(x_2) - \inprod{\nabla h_i(x_2)}{x_1-x_2}.
	\end{align*}
	Summing up the above two inequalities and noting the definition of $f^{\beta_i}_i, \nabla f^{\beta_i}_i$, we conclude that $f^{\beta_i}_{i}$ has an upper curvature of $L^{\beta_i}_{g_i}$.  
	Similarly, using convexity of $g^{\beta_i}_{i}$ and smoothness of $h_i$, we obtain that  $f^{\beta_i}_{i}$ has a negative lower curvature $-L_{h_i}$.
	
Part 3. The Lipschitz continuity \eqref{eq:tight-Lip-const} is an immediate consequence of part 2) and Lemma~\ref{lem:tight_lipschitz-const}. 
\end{proof}
\begin{remark}
	When $\bar{\beta_i}=0$, Relation \eqref{eq:apprx-bound} reads 
		$f_i^{\beta_i}(x) \le  f_i(x) \le f_i^{\beta_i}(x) + \tfrac{\beta_i D_{\Ycal_i}^2}{2}.$ Together with Assumption~\ref{ass:strict-feasible}, it can be seen that $x^0$ is also strictly feasible for problem~\eqref{prob:smooth_main}. This justifies that {\GD} is well-defined for problem~\eqref{prob:smooth_main}.
\end{remark}
\begin{remark}
	The Lipschitz constant of $\nabla f^{\beta_i}_i$ can be derived in a different way. Since $\nabla g^{\beta_i}_i$ and $\nabla h_i$ are $L_{g_i}^{\beta_i}$ and $L_{h_i}$ Lipschitz continuous, respectively, we can show by triangle inequality that $\nabla f^{\beta_i}_{i}(x)$ is $L_{g_i}^{\beta_i} + L_{h_i}$-Lipschitz continuous.
	In contrast, by exploiting the asymmetry between lower and upper curvature, Proposition~\ref{prop:upp-low-curv-nonsmooth} derived a slightly sharper bound on the gradient Lipschitz constant.
\end{remark}

Throughout this section, we choose specific $\beta_i$ to ensure $\beta_iD_{\Ycal_i}^2$ is constant for all $i \in [m]$. Hence, we can define the additive approximation factor above as 
\begin{equation}
	\label{eq:def_nu}
	\nu := \tfrac{\beta_iD_{\Ycal_i}^2}{2},\ i\in [m].
\end{equation} 
Note that \eqref{eq:int_rel5} provides an approximation error for function values, or the so-called zeroth-order oracle of function $g_i$. However, convergence results for nonconvex optimization are generally given in terms of first-order stationarity measure, implying that we need approximation for first-order oracle for the function $f_i$ and consequently function $g_i$. Below we discuss a widely used approximate subdifferential for convex functions and generalize it for nonsmooth nonconvex functions. 
\begin{definition}
	[$\nu$-subdifferential] We say that a vector $v \in \Rbb^n$ is a $\nu$-subgradient of the convex function  $p(\cdot)$ at $x$ if for any $z$, we have 
	\[p(z) \ge p(x) + \inprod{v}{z-x} -\nu.\]
	The set of all $\nu$-subgradients of $p$ at $x$ is called the $\nu$-subdifferential, denoted by $\partial^{\nu}p(x)$. Moreover, we define $\nu$ subdifferential of nonconvex function $f_i$ as $\partial^{\nu} f_i(x) := \partial^{\nu} g_i(x) + \{-\nabla h_i(x)\}$ where the addition of sets is in Minkowski sense.
\end{definition}

Finally, we define a generalization of type-I KKT convergence criterion for structured nonsmooth nonconvex function constrained optimization problem:
\begin{definition}\label{def:app_KKT_nonsmooth}
	We say that a point $x$ is an $(\epsilon,\nu)$ type-III KKT point of \eqref{prob:main} if there exists $\lambda \ge 0$ satisfying the following conditions:
	\begin{align}
		\gnorm{\partial^{\nu} \psi_{0}(x) + \tsum_{i=1}^m \lambda_i\partial^{\nu}\psi_i(x)}{-}{2} & \le\epsilon,\label{eq:stationary-03}\\
		\tsum_{i=1}^m\lambda_i \vert\psi_i(x)-\eta_i\vert & \le\epsilon,\label{eq:stationary-04}\\
		\gnorm{ [\psi(x)-\eta]_+}{1}{} &\le \epsilon.\label{eq:stationary-05}
	\end{align}
	Moreover, we say that $x$ is a randomized $(\epsilon,\nu)$ type-III KKT point of \eqref{prob:main} if \eqref{eq:stationary-03}, \eqref{eq:stationary-04} and \eqref{eq:stationary-05} are satisfied in expectation.
\end{definition}

The $\epsilon$-subdifferential and the type-III KKT point are essential for associating smooth approximation with the original nonsmooth problem.   We build some important properties in the following proposition. 
\begin{proposition}\label{prop:subdiff-kkt}
	Let $\beta_i$ and $\nu$ satisfy \eqref{eq:def_nu}. 
	\begin{enumerate}
		\item For any $x\in \Rbb^d$,  we have $\nabla f_i^{\beta_i}(x) \in  \partial^{\nu}f_i(x)$, $i=0,1,2,...,m$.%
		\item Suppose that $x$ is a (randomized) Type-I $\epsilon$-KKT point of problem~\eqref{prob:smooth_main} and $\lambda$ is the associated dual variable with bound $\norm{\lambda}\le B$, then $x$ is a (randomized) Type-III $(\bar{\epsilon}, \nu)$-KKT point of problem~\eqref{prob:main} for $\bar{\epsilon}=\max\{\epsilon+B\nu, m\nu\}$.
	\end{enumerate}
\end{proposition}

\begin{proof}
	Part 1. It suffices to show $\nabla g_i^\beta(x) \in  \partial^{\nu}g_i(x)$. Due to the convexity of $g_i^{\beta_i}$ and \eqref{eq:int_rel5}, we have
	\begin{equation*}
		g_i(z) \ge g_i^{\beta_i}(z) \ge g_i^{\beta_i}(x) + \inprod{\nabla g_i^{\beta_i}(x)}{z-x} \ge g_i(x) + \inprod{\nabla g_i^{\beta_i}(x)}{z-x} - \tfrac{\beta_i D_{\Ycal_i}^2}{2},
	\end{equation*}
	where the first inequality follows from the first relation in \eqref{eq:int_rel5}, and the third inequality follows from second relation in \eqref{eq:int_rel5}.
	Noting the definition of $\nu$-subgradient, we conclude the proof.
	
	Part 2. It suffices to show the conversion of randomized Type-I KKT points to randomized Type-III KKT points. Suppose that $x$ is a randomized Type-I $\epsilon$-KKT point and we have $\norm{\lambda}_1\le B$.
	Using Part 1 it is easy to see $\partial \Lcal (x, \lambda) \subseteq \partial^\nu \psi_0(x)+\tsum_{i=1}^m \partial^\nu \psi_i(x)$, therefore, we have
	\[
	\Ebb\gnorm{\partial^\nu \psi(x)+\tsum_{i=1}^m \partial^\nu \psi_i^\nu(x)}{-}{2}\le \epsilon.
	\]
Using Proposition~\ref{prop:upp-low-curv-nonsmooth}, we have 
	 \[
	 \lambda_i(\psi_i^{\beta_i}(x)-\eta_i) \le \tsum_{i=1}^m \lambda_i(\psi_i(x)-\eta_i)\le \lambda_i(\psi_i^{\beta_i}(x)-\eta_i) + \lambda_i{\nu} \le \lambda_i{\nu}.\] 
This implies  
	\[\vert\lambda_i(\psi_i(x)-\eta_i)\vert \le \max\{|\lambda_i(\psi_i^{\beta_i}(x)-\eta_i)|, \lambda_i{\nu}\}.\]
	Summing up this inequality over $i=1,2,3,...,m$ and taking expectation with all the randomness, we have the
	\[
	\tsum_{i=1}^m \Ebb\vert\lambda_i(\psi_i(x)-\eta_i)\vert \le \Ebb\tsum_{i=1}^m |\lambda_i(\psi_i^{\beta_i}(x)-\eta_i)| + \tsum_{i=1}^m\lambda_i{\nu} \le \epsilon + B\nu.\]
	Moreover, we have 
	\[ 
	\tsum_{i=1}^m [\psi_i(x)-\eta_i]_+= \tsum_{i=1}^m [\psi_i^{\beta_i}(x)-\eta_i+\psi_i(x)-\psi_i^{\beta_i}(x)]_+\le \tsum_{i=1}^m [\psi_i(x)-\psi_i^{\beta_i}(x)] \le m \nu. 
	\]
	
\end{proof}

Now, we are ready to discuss the convergence rate of {\GD} for nonsmooth nonconvex function constrained optimization. 
\begin{theorem}
	Assume that $\beta_i, \nu$ satisfy \eqref{eq:def_nu} and set $\delta^{k}=\tfrac{\eta-\eta^{0}}{(k+1)(k+2)}$ when running {\GD} to solve problem~\eqref{prob:smooth_main}. 
Denote $c_i = \gnorm{A_i}{}{2}D_{\Ycal_i}^2$ $(0\le i \le m)$, $c=[c_1,c_2,..., c_m]^T$ and let $\norm{\lambda^k}_1\le B$. Suppose that $\nu=o(\frac{c_i}{L_{h_i}})$ for $i=0,1,2,...,m$, then $x^{\hat{k}+1}$ is a randomized Type-III $(\bar{\epsilon},\nu)$-KKT point with 
\[
\bar{\epsilon} =\Ocal\Bcbra{\frac{2}{K+2}\bsbra{\brbra{\frac{8(c_0+B\gnorm{c}{}{})^2}{c_0\nu}+\frac{2B\gnorm{c}{}{}}{c_0}}(\Delta+\nu) +B\gnorm{\eta-\eta^0}{}{}} + B \nu + m\nu}.
\]
\end{theorem}

\begin{proof}
	Our analysis resembles the proof of Theorem~\ref{thm:lcgd:sum-bound}. Using a similar argument in~\eqref{eq:weighted-square-sum}, we have
	\begin{align}
\tsum_{k=0}^K {\alpha_k}\gnorm{x^k-x^{k+1}}{}{2} &  \le \tfrac{2\alpha_K}{L_0} \bsbra{\psi_0^{\beta_0}(x^0) - \psi_0^{\beta_0}(x^{K+1})} \nonumber\\
& \le \tfrac{2\alpha_K}{L_0^{\beta_0}} \bsbra{\psi_0(x^0) - \psi_0(x^{K+1})+\nu}. 
\end{align}
Combining this result with Lemma~\ref{lem:lk-bound} we obtain
\begin{equation}
\begin{aligned}
\tsum_{k=0}^{K}\alpha_k\gnorm{ \partial_{x}\Lcal(x^{k+1},\lambda^{k+1})}{-}{2} & \le 8\frac{(L_{0}^{\beta_0}+B\gnorm{L^{\beta}}{}{})^{2}}{L_0^{\beta_0}}\alpha_K (\Delta+\nu),\\
\tsum_{k=0}^{K} \alpha_k \langle \lambda^{k+1}, \vert\psi(x^{k+1})-\eta\vert \rangle  
& \le 2B\frac{\gnorm{L^{\beta}}{}{}}{L_0^{\beta_0}}\alpha_K (\Delta+\nu) + B \tsum_{k=0}^K\alpha_k\gnorm{\eta-\eta^k}{}{},%
\end{aligned}\label{eq:nonsmooth}
\end{equation}
where $\Delta=\psi_0(x^0)-\psi_0(x^*)$,  $\alpha_k\ge 0$ and
		$\Lcal^\beta(x, \lambda) \coloneq \psi_{0}^{\beta_0}(x) + \tsum_{i=1}^m\lambda_{i}(\psi^{\beta_i}_i(x) -\eta_i)$.
Taking $\delta^{k}=\tfrac{\eta-\eta^{0}}{(k+1)(k+2)}$ and $\alpha_k=k+1$ in \eqref{eq:nonsmooth}, we see that $x^{\hat{k}+1}$ is a Type-I $\epsilon$-KKT point for 
\[
\epsilon = \tfrac{2}{K+2} \max \big\{\tfrac{8(L_{0}^{\beta_0}+B\gnorm{L^{\beta}}{}{})^{2}}{L_0^{\beta_0}}(\Delta+\nu),   \tfrac{2B\gnorm{L^{\beta}}{}{}}{L_0^{\beta_0}} (\Delta+\nu)+B\gnorm{\eta-\eta^0}{}{}\big\}.
\]
Noting that  $L_{g_i}^{\beta_i} = \tfrac{\gnorm{A_i}{}2}{\beta_i} = \tfrac{\gnorm{A_i}{}{2}D_{\Ycal_i}^2}{2\nu}=\frac{c_i}{2\nu}$ and $\nu=o(\frac{c_i}{L_{h_i}})$, we have $\tfrac{(L_{0}^{\beta_0}+B\gnorm{L^{\beta}}{}{})^{2}}{L_0^{\beta_0}}=\Ocal\big(\tfrac{(c_0+B\gnorm{c}{}{})^2}{c_0\nu}\big)$ and $\tfrac{\gnorm{L^{\beta}}{}{}}{L_0^{\beta_0}}=\tfrac{\gnorm{c}{}{}}{c_0}$.
 Using the definition of $\hat{k}$ and  Proposition~\ref{prop:subdiff-kkt} we obtain the desired result.
\end{proof}

\section{Inexact {\GD}}\label{sec:inexact-GD}
{\GD} requires the exact optimal solution of subproblem~(\ref{subproblem}), which, however,  poses a  great challenge when the subproblem is difficult to solve. To alleviate such an issue, we consider an inexact variant of \GD~method for which the update of $x^{k+1}$ only  solves problem~\eqref{subproblem} approximately. 
This section is organized as follows. First, we present
 a general convergence property of inexact {\GD} when  the subproblem solutions satisfy certain approximation criterion. Next, we analyze the efficiency  of inexact {\GD} when the subproblems are handled by different external solvers. When the subproblem is a quadratically constrained quadratic program (QCQP), we propose an efficient interior point algorithm  by exploiting the diagonal structure.
  When the subproblem has general proximal components, we propose to solve it by  first-order methods. Particularly, we consider solving the subproblem by the constraint extrapolation ({\conex}) method  and develop the total iteration complexity of {\conex}-based {\GD}.

\subsection{Convergence analysis under an inexactness criterion}
Throughout the rest of this section, we will denote the exact primal-dual solution of \eqref{subproblem} as $(\wtil{x}^{k+1}, \wtil\lambda^{k+1})$.
We use the following criterion for measuring the accuracy of subproblem solutions. 
\begin{definition}\label{def:inexact-criterion}
We say that a point $x$ is an $\epsilon$-solution of~\eqref{subproblem} if 
\begin{equation*}
	\psi_0^{{k}}(x)-\psi_0^{k}{(\wtil{x}^{k+1})} \le \epsilon,\quad
	\gnorm{[\psi^{k}(x)]_+}{}{} \le \epsilon,\quad
	\Lcal_{k}(x, \wtil\lambda^{k+1}) \le \Lcal_{k}(\wtil{x}^{k+1}, \wtil\lambda^{k+1} ) + \epsilon.
\end{equation*}
\end{definition}

The following theorem shows asymptotic convergence to stationarity  for inexact \GD~method under mild assumptions. Since the proof is similar to the previous argument, we present the details in Appendix \ref{appx:LCGD_inexact-nonconvex} for the sake of completeness. Note that the theorem applies to a general nonconvex problem and hence applies to convex problems as well. 
\begin{theorem}\label{thm:lcgd-inexact-asymp}
Suppose that Assumption~\refeq{assu:y-bounded} holds and let $x^{k+1}$ be an $\epsilon_k$-solution of \eqref{subproblem} satisfying $\epsilon_k < \min_{i \in [m]} \delta_{i}^k$. %
Then all the conclusions of Theorem~\refeq{prop:lcgd:bound-dual} still hold. Then the dual sequence $\{\tilde{\lambda}^{k}\}$ is uniformly bounded by a constant $B>0$. Moreover, every limit point of inexact {\GD} is a KKT point.
\end{theorem}

Under the inexactness condition in Definition~\ref{def:inexact-criterion}, we establish the  complexity of inexact {\GD} in the following theorem. 
\begin{theorem}\label{thm:inexact-rate}
\added{Under the assumptions of Theorem~\refeq{thm:lcgd-inexact-asymp}, we have
\begin{align}
	\tsum_{k=0}^K \alpha_k\gnorm{\partial_{x}\Lcal(\tilde{x}^{k+1},\tilde{\lambda}^{k+1})}{-}{2} &\le \tfrac{8(L_0+B\gnorm{L}{}{})^2}{L_0} \tilde{\Delta}, \label{eq:inexact-mid-5} \\
	\tsum_{k=0}^K \alpha_k\langle\tilde{\lambda}^{k+1}, \vert\psi(\tilde{x}^{k+1})-\eta\vert\rangle & \le B\tsum_{k=0}^K \alpha_k\gnorm{\eta-\eta^k}{}{} + \tfrac{2B\gnorm{L}{}{}}{L_0}\tilde{\Delta},\label{eq:inexact-mid-6}\\
	\tsum_{k=0}^K\alpha_{k}\gnorm{x^k-\wtil{x}^{k+1}}{}{2} &\le \tfrac{2}{L_0}\tilde{\Delta}, \label{eq:inexact-mid-7}
\end{align}
where $\tilde{\Delta}=\tsum_{k=0}^K\alpha_{k}\sbra{\psi_0(x^k)-\psi_0(x^{k+1})+{\vep}_k}$. Moreover, if we choose the index $\hat{k} \in \{0,1, \dots, K\}$ with probability $\prob(\hat{k} = k) = \alpha_k/(\tsum_{i=0}^K\alpha_{i}) $, then $x^{\hat{k}}$ is a randomized $(\epsilon,\delta)$ type-II KKT point with 
\begin{equation}\label{eq:KKT-inexact-case}
	\begin{split}
	\epsilon &= 1\big/(\tsum_{i=0}^K\alpha_{i})\max\braces[\big]{\tfrac{8(L_0+B\gnorm{L}{}{})^2}{L_0} \tilde{\Delta}, B\tsum_{k=0}^K\alpha_k \gnorm{\eta-\eta^k}{}{} + \tfrac{2B\gnorm{L}{}{}}{L_0}\tilde{\Delta}
	},\\
	\delta &= 2\tilde{\Delta}\big/{(L_0\tsum_{i=0}^K\alpha_{i})}.
	\end{split}
\end{equation}
In particular, using $\alpha_k = k+1$, $\epsilon_k = \min_{i \in [m]} \tfrac{\delta^k_i}{2}$ and $\delta_{i}^k = \tfrac{\eta_i - \eta^k_i}{(k+1)(k+2)}$, we have $x^{\hat{k}}$ is $(\epsilon, \delta)$ type-II KKT point of \eqref{prob:main} where 
\begin{equation}\label{eq:conv-KKT-inexact}
	\begin{split}
		\epsilon &= \tfrac{2}{K+2}\max\bcbra{ 4(L_0+B\gnorm{L}{}{})^2\bracket{2D^2 + \tfrac{\gnorm{\eta-\eta^0}{}{}}{L_0} }, B\gnorm{\eta-\eta^0}{}{} + B\gnorm{L}{}{}
			\bracket{2D^2 + \tfrac{\gnorm{\eta-\eta^0}{}{}}{L_0} }
		},\\
		\delta &= \tfrac{2}{K+2}\brbra{ 2D^2  + \tfrac{\gnorm{\eta-\eta^0}{}{}}{L_0}}.
	\end{split}
\end{equation} 
}
\end{theorem}
\begin{proof}
Using \eqref{eq:int_rel20} with $x^{k+1}$ replaced by $\wtil{x}^{k+1}$ (the optimal solution of problem~\eqref{subproblem}) and adding $f(x^k) + \tfrac{L_0}{2}\gnorm{x^k-\wtil{x}^{k+1}}{}{2}$ on both sides, we have
\begin{align}
		\tfrac{L_0}{2}\gnorm{x^k-\wtil{x}^{k+1}}{}{2} &\le \psi_{0}^k(x^k)-\psi_{0}^k(\wtil{x}^{k+1}) \nonumber\\
		&\le \psi_{0}(x_k) - \psi_{0}^k(x^{k+1})  + \epsilon_k \nonumber\\
		&\le \psi_{0}(x_k) - \psi_{0}(x^{k+1})  + \epsilon_k,\label{eq:almost_suff_descent}
\end{align}
where the second inequality follows from $\psi_{0}^k(x^k) = \psi_{0}(x^k)$ as well as $x^{k+1}$ being an $\epsilon_k$-solution (see Definition~\ref{def:inexact-criterion}) of subproblem~\eqref{subproblem}, and the third inequality follows from the fact that $\psi_{0}^k(x) \ge \psi_{0}(x)$ for all $x \in \dom{\chi_{0}}$.
	
Using Lemma~\ref{lem:lk-bound} (again $x^{k+1}$ is replaced by $\wtil{x}^{k+1}$), noting that $\epsilon_{k}$ satisfies the requirements of Theorem~\ref{thm:lcgd-inexact-asymp} implying that $\gnorm{\wtil{\lambda}^k}{}{} \le B$ and using \eqref{eq:almost_suff_descent}, we have 
	\begin{align}
		\gnorm{\partial_{x}\Lcal(\tilde{x}^{k+1},\tilde{\lambda}^{k+1})}{-}{2} &\le4(L_0+B\gnorm{L}{}{})^2 \gnorm{\tilde{x}^{k+1}-x^k}{}2 \nonumber \\
		& \le \tfrac{8(L_0+B\gnorm{L}{}{})^2}{L_0}\big[\psi_0(x^k)-\psi_0(x^{k+1})+\vep_k\big].\label{eq:inexact-mid-4}
	\end{align}
Similar to the argument of (\ref{eq:bound-cs-k}), we have
	\begin{align}
		\tsum_{i=1}^m\tilde{\lambda}_{i}^{k+1}\vert\psi_{i}(\tilde{x}^{k+1})-\eta_{i}\vert & \le  B \gnorm{\eta-\eta^k}{}{} + B\gnorm{L}{}{}\gnorm{\tilde{x}^{k+1}-x^k}{}2 \nonumber \\
		& \le  B \gnorm{\eta-\eta^k}{}{} +  \tfrac{2B\gnorm{L}{}{}}{L_0}\big[\psi_0(x^k)-\psi_0(x^{k+1})+\vep_k\big]. \label{eq:inexact-mid-3}
	\end{align}
Multiplying \eqref{eq:almost_suff_descent},  \eqref{eq:inexact-mid-4} and \eqref{eq:inexact-mid-3} by $\alpha_{k}$ and summing over $k=0, 1, \ldots, K$ give \eqref{eq:inexact-mid-5}, \eqref{eq:inexact-mid-6} and \eqref{eq:inexact-mid-7}. 
	
	We derive a convergence rate  based on the specified parameters. First, from relation \eqref{eq:almost_suff_descent}, we note that $\psi_{0}(x^{k+1}) \le \psi_{0}(x^k) + \epsilon_k$. Hence, we have by induction that 
	\begin{equation}\label{eq:int_rel21}
		\psi_{0}(x^{k+1}) \le \psi_{0}(x^0) + \tsum_{i=0}^k\epsilon_i.
	\end{equation}
	By setting $\alpha_k = (k+1)$ and $\epsilon_{k} = \min_{i \in [m]} \tfrac{\delta_{i}^k}{2} = \min_{i \in [m]} \tfrac{\eta_i-\eta^k_i}{2(k+1)(k+2)}$ for all $k \ge 0$ (note that $\epsilon_k$ satisfies the requirement of Theorem \ref{thm:lcgd-inexact-asymp}), we have 
	\begin{align}
		\wtil{\Delta} &= \tsum_{k =0}^K\alpha_k[\psi_{0}(x^k) - \psi_{0}(x^{k+1})] + \tsum_{k =0}^K\alpha_{k}\epsilon_k \nonumber\\
		&= \alpha_0\psi_{0}(x^0) + \tsum_{k =0}^{K-1}(\alpha_{k+1}-\alpha_{k})\psi_{0}(x^{k+1}) -\alpha_{K}\psi_{0}(x^{K+1}) + \tsum_{k =0}^K \alpha_k\epsilon_k \nonumber\\	
		&\myrel{(i)}{\le} \alpha_0\psi_{0}(x^0) + \tsum_{k =0}^{K-1}(\alpha_{k+1}-\alpha_{k})[\psi_{0}(x^0) + \tsum_{i=0}^k\epsilon_i]- \alpha_K\psi_{0}(x^{K+1}) + \tsum_{k =0}^K\alpha_{k}\epsilon_k \nonumber\\
		&\myrel{(ii)}{=} \alpha_K[\psi_{0}(x^0) - \psi_{0}(x^{K+1})] + \tsum_{k =0}^{K-1}\tsum_{i=0}^k\epsilon_i + \tsum_{k =0}^K \alpha_k\epsilon_k \nonumber\\
		&=\alpha_K[\psi_{0}(x^0) - \psi_{0}(x^{K+1})] + \tsum_{i=0}^{K-1}\tsum_{k =i}^{K-1}\epsilon_i + \tsum_{k =0}^K \alpha_k\epsilon_k \nonumber\\
		&\myrel{(iii)}{=}\alpha_K[\psi_{0}(x^0) - \psi_{0}(x^{K+1})] + \tsum_{i=0}^{K}(K-i)\epsilon_i + \tsum_{k =0}^K \alpha_k\epsilon_k \nonumber\\
		&\myrel{(iv)}{\le} \alpha_K[\psi_{0}(x^0) - \psi_{0}(x^{K+1})] + \tfrac{\gnorm{\eta-\eta^0}{}{}}{2}\tsum_{k=0}^{K}\tfrac{K-k}{(k+1)(k+2)} + \tfrac{1}{k+2} \nonumber\\
		&=\alpha_K[\psi_{0}(x^0) - \psi_{0}(x^{K+1})] + \tfrac{\gnorm{\eta-\eta^0}{}{}}{2}\tsum_{k=0}^{K}\tfrac{K+1}{(k+1)(k+2)} \nonumber\\
		&\le \alpha_K\bracket[\big]{\psi_{0}(x^0) - \psi_{0}(x^{K+1}) + \tfrac{\gnorm{\eta-\eta^0}{}{}}{2}}. \label{eq:int_rel22}
	\end{align} 
	Here, {\sffamily (i), (ii)} follows from \eqref{eq:int_rel21} and $\alpha_{k+1}-\alpha_k = 1 (> 0)$, {\sffamily (iii)} follows $(K-i)$ is $0$ at $i = K$, {\sffamily (iv)} follows by observing $\epsilon_k \le \tfrac{\gnorm{\eta-\eta^0}{}{}}{2(k+1)(k+2)}$ and last inequality follows since $\tsum_{k =0}^\infty \tfrac{1}{(k+1)(k+2)} = 1$ and $\alpha_K = K+1$. 
	
	Applying same arguments as those in Corollary \ref{cor:rate-lcgd}, we have $\tsum_{k =0}^K\alpha_k\gnorm{\eta-\eta^k}{}{} = \alpha_K\gnorm{\eta-\eta^0}{}{}$. Using this relation along with $\tsum_{k =0}^K\alpha_k = \tfrac{(K+1)(K+2)}{2}$ and \eqref{eq:int_rel22} inside \eqref{eq:KKT-inexact-case}, we have \eqref{eq:conv-KKT-inexact}. Hence, we conclude the proof.	
\end{proof}

\begin{remark} 
Compared to the convergence result \eqref{eq:sum-slack-1} for  exact {\GD}, we have to control the accumulated error  in  $\tilde{\Delta}$  for the inexact case \eqref{eq:KKT-inexact-case}.
However, we need an even more stringent condition on the error to ensure asymptotic convergence. 
Specifically, we assume ${\vep}_k$ to be smaller than the level increments $\delta_{i}^k$ to ensure that each subsequent subproblem is strictly feasible.
As long as the subproblems are solved deterministically with sufficient accuracy, we can ensure such feasibility as well as the boundedness of the dual.
\end{remark}

\begin{remark}
	{Note that the convergence analysis of the inexact method for the stochastic case will go through in a similar fashion. In particular, the subproblems of \SGD{} are still deterministic in nature. Hence, a deterministic error can be easily incorporated into the analysis of the stochastic outer loop. In particular, Proposition \ref{lem:sgd-suff-decrease} will have an additional $\epsilon_k$ in the RHS. We can use $\epsilon_k = \min_{i \in [m]} \tfrac{\delta^k_i}{2}$ to ensure the strict feasibility. Following the analysis in Theorem \ref{thm:main-sgd}, we will get the additional term $\tsum_{k=0}^K \alpha_k\epsilon_k $. Note that we have identical policies for $\alpha_k$ in the above analysis and Corollary \ref{cor:rate_lcsgd}. Furthermore, since $\delta_k$ used above and in Corollary \ref{cor:rate_lcsgd} are the same, we have identical values of $\epsilon_k$ as well. Following the above development, we can easily bound the additional $\tsum_{k=0}^K \alpha_k\epsilon_k $ term. 
 }
\end{remark}

\subsection{Solving the subproblem with the interior point method}\label{sec:solving-subproblem}

Our goal  is to develop an efficient interior point algorithm to solve
 problem~\eqref{subproblem} when $\chi_i(x)=0$, $i\in[m]$. Without loss of generality, we express the subproblem as the following QCQP:
\begin{equation}\label{prob:diag-qp}
\begin{aligned}
	\min_{x\in\Rbb^d} & \quad g_0(x)\coloneq \tfrac{L_0}{2}\|x-a_0\|^2 \\
	\st & \quad g_i(x)\coloneq  \tfrac{L_0}{2}\|x-a_i\|^2 - b_i \le 0, \quad  i\in[m].
\end{aligned}
\end{equation}
We assume that the  initial solution $\hat{x}$ of such problem is strictly feasible, namely, there exists $\delta>0 $ such that
\begin{equation}
	g_i(\hat{x})\le -\delta,\quad i=1,2,\ldots m.
\end{equation}
Let $e_1=[1,0,\ldots,0]^\trans\in\Rbb^{d+1}$. With a slight abuse of notation, we can formulate~\eqref{prob:diag-qp} as the following problem
\begin{equation}\label{prob:lin-form}
	\begin{aligned}
		\min & \qquad \quad e_1^\trans u \\
		 \st & \quad \til{g}_0(u) = g_0(x)-\eta \le 0, \\
		 	& \quad \til{g}_i(u) = g_i(x)\le 0,\ i\in[m], \\
		 	& \quad \til{g}_{m+1}(u)=\tfrac{L_{m+1}}{2}\norm{u-(0, a_{m+1})^\trans}^2-b_{m+1} \le 0, \\
		 	& \quad u = (\eta, x) \in \Rbb\times \Rbb^d.
	\end{aligned}
\end{equation}
Here we set artificial variables $L_{m+1}=1, a_{m+1}=0$ and $b_{m+1}=\frac{1}{2}R^2$ for some sufficiently large $R$.   We explicitly add such a ball constraint to ensure bound on $(\eta, x)$. Note that the bound $R$ always exists since our domain is compact and the objective is Lipschitz continuous. 
Our goal is to apply the path-following method to solve (\ref{prob:lin-form}).  We denote 
\begin{equation}
	\phi(u)= -\tsum_{i=0}^{m+1}\log -\til{g_i}(u)
\end{equation}
Since each $\til{g}_i(u)$ is convex quadratic in $u$, $\phi(u)$ is a self-concordant barrier with $\upsilon=m+2$.
The key idea of the path-following algorithm is to approximately solve a sequence of penalized problems
\begin{mini}
{u}{\phi_\tau(u) := \tau \eta +\phi(u)}{}{}\label{prob:barrier}
\end{mini}
with increased values of $\tau$,
and generate a sequence of strictly feasible solution $u_\tau$ close to the central path--a trajectory composed of the  minimizers $u_\tau^*=\argmin_u \phi_\tau(u)$.

We apply a standard path-following algorithm (See~\cite[Chapter~4]{nemirovski2004interior}) for solving (\ref{prob:lin-form}), and outline the overall procedure in Algorithm~\ref{alg:barrier}. This algorithm consists of two main steps: 
\begin{enumerate}
	\item Initialization: We seek a solution $u^0$ near the analytic center (i.e. minimizer of $\phi(u)$). To this end, we  solve a sequence of auxiliary problems $\hat{\phi}_\tau(u)=\tau w^\trans u + \phi(u)$ where $w=-\nabla \phi(\hat{u})$. It can be readily seen that $\hat{u}$ is in the central path of this auxiliary problem with $\tau=1$.  Performing a reverse path-following scheme ( decreasing rather than increasing $\tau$), we gradually converge to the analytic center.
	\item Path-following: We solve a sequence of penalized problems with  an increasing value of $\tau$ by a damped version of Newton's method, which ensures the solutions in the proximity of the central path.
\end{enumerate}
\begin{algorithm}[ht]
	\begin{algorithmic}[1]
		\LeftComment{Newton decrement $n(f, u) \coloneq \sqrt{\nabla f(u)^\trans [\nabla^2 f(u)]^{-1} \nabla f(u)}$}
		\State {\bf Input: }$\hat{u}, \kappa\in(0,1)$, $\gamma>0$, $\vep$; Set $\tau_0=1$, $u^0=\hat{u}$;
		\LeftComment{Phase Zero: Approximate analytic center}
		\For{$i = 0,1, \dots$}
			\State $\tau_{i+1}=\big(1+\tfrac{\gamma}{\sqrt{\upsilon}}\big)^{-1}\tau_i$;
			\State Obtain $u^{i+1}$ from calling \texttt{Newton}$(\phi_{\tau_{i+1}}, u^i,\tau_{i+1}, \kappa/2)$;
			\If{$n(\phi, u^{i+1})\le\tfrac{3}{4}\kappa$}
				\State Set $u^*=u^{i+1}$ and \bf{Break};
			\EndIf
		\EndFor
		\LeftComment{Phase One:  Path-following scheme}
		\State Set $u^0=u^*$, $\tau_0=\max\{\tau:\, n(\phi_\tau, u^0)\le\kappa\}$, $s= \big\lceil \tfrac{\sqrt{\upsilon}}{\gamma}\ln\tfrac{2\upsilon}{\tau_0\vep}\big\rceil-1$;
		\For{$i =0,1, \dots, s$}
			\State $\tau_{i+1}=\big(1+\tfrac{\gamma}{\sqrt{\upsilon}}\big)\tau_i$;
			\State Obtain $u^{i+1}$ from calling \texttt{Newton}$(\phi_{\tau_{i+1}}, u^i, \tau_{i+1}, \kappa)$;
		\EndFor
		\State {\bf Output: } $u^{s+1}$.
		\LeftComment{Damped Newton method for solving the subproblem.}
		\Function{\texttt{Newton}}{$f$, $v^0$, $\tau$, $\epsilon$}
			\For{$s=0, 1,2,\dots$}
				\If{$n(f, v^s)\le \epsilon $} {\bf Break};
				\EndIf
				\State $v^{s+1}=v^s - \tfrac{1}{1+n(f, v^s)}[\nabla^2f(v^s)]^{-1} \nabla f(v^s)$;
			\EndFor
		\EndFunction
	\end{algorithmic}
	\caption{Path-following Interior Point Method (\cite{nemirovski2004interior})}\label{alg:barrier}
\end{algorithm}

\subsubsection{Solving the Newton Equation\label{newtonsys}}
First, we calculate the gradient and Hessian map of $\phi_t(\cdot)$:
\begin{align*}
\nabla\phi_\tau(u)& = \tau e_1 +\tsum_{i=0}^{m+1} \theta_i \nabla\til{g}_i(u), \\
	\nabla^2 \phi_\tau(u) &= \tsum_{i=0}^{m+1} \theta_i^2 \nabla\til{g}_i(u)\til{g}_i(u)^\trans+\tsum_{i=0}^{m+1}\theta_i \nabla^2\til{g}_i(u) = NN^\trans+\Gamma,
\end{align*}
where $\theta_i = -\til{g}_i(u)^{-1}$, and 
\begin{equation*}
\begin{aligned}
N & = \begin{bmatrix}
	\theta_0  \nabla\til{g}_0(u),\ldots, \theta_{m+1}  \nabla\til{g}_{m+1}(u)
\end{bmatrix} \in\Rbb^{(d+1)\times (m+2)}, \\
\Gamma & = \begin{bmatrix}
	\theta_{m+1}L_{m+1} & 0 \\
	0 & \tsum_{i=0}^{m+1}\theta_iL_i I_d
\end{bmatrix}	\in\Rbb^{(d+1)\times(d+1)}.
\end{aligned}
\end{equation*}
Note that computing the gradient $\nabla \phi_t(u)$ takes $\Ocal(dm)$, hence the computation burden is from forming and solving the Newton systems.  This is divided into two cases.
\begin{enumerate}
\item  $m < d$. Then the Hessian is the sum of a low rank matrix and a diagonal matrix. Based on the Sherman-Morrison-Woodbury formula, we have 
\begin{equation}\label{eq:phi-Hes-inv}
[\nabla^2\phi_\tau(u)]^{-1} = \Gamma^{-1}-\Gamma^{-1}N\big(I+N^\trans \Gamma^{-1} N\big)^{-1}N^\trans \Gamma^{-1}.	
\end{equation}
Computing the product $N^\trans \Gamma^{-1} N$ takes $\Ocal(m^2d)$ while performing  Cholesky factorization takes $\Ocal({m^3})$.  Therefore, the overall complexity of each Newton step is $\Ocal(m^3+m^2d)=\Ocal(m^2d)$.
\item $m \ge d$. In such case, we can directly compute  $NN^\trans$ in $\Ocal(md^{2})$ and then perform
Cholesky factorization  $\nabla^2\phi_\tau(x)=LL^\trans $
in $\Ocal(d^{3})$ , followed by  two triangle systems. Hence the overall complexity of a Newton step is $\Ocal(d^{3}+md^2)=\Ocal(md^2)$.
\end{enumerate}
Due to the above discussion, the cost of computing each Newton system is 
\begin{equation}\label{eq:cost-newton}
	\Ocal\big(\min\{d, m\}\cdot md\big).
\end{equation}

\subsubsection{Complexity}
Before deriving the complexity of solving the subproblems, we require some additional assumptions.
We assume that $M=\max_x\big\{ \| \nabla g_i(x)\|\big\}$  and $\max g_0({x})-\min_x g_0(x)\le V$.  Note that these assumptions are easily satisfied  if we assume functions in the original problem have bounded level sets.

According to~\cite[Theorem 4.5.1]{nemirovski2004interior}, the complexity of interior point methods depends not only on the time to follow the central path, but also on the time to arrive near the analytic center from an arbitrary initial point. Let us put it in the context of Algorithm~\ref{alg:lcgd}. Despite the strict feasibility guarantee, we do not know whether  $x^k$ is near the analytic center of each subproblem. It remains to show how to control the complexity of approximating the analytic center.

To measure the strict feasibility of the initial point,  we use the Minkowsky function of the domain, which is  defined by $\pi_x(y)=\inf\{t>0: \, x+t^{-1}(y-x)\in D\}$ for any given $x$ in the interior of the domain.
With the help of the Minkowsky function, we bound the distance between the initial point and the boundary in the following proposition.
\begin{proposition}\label{prop:Minko}
Let  $\hat{u}=(g_0(\hat{\eta}), \hat{x})$ where $\hat{\eta}=g_0(\hat{x})+\delta$.  If $\|u-\hat{u}\| \le \tfrac{\delta}{M+1}$, then $u$ is feasible for problem~\eqref{prob:lin-form}. 
Moreover, we have 
\begin{equation}\label{eq:bound-pistar}
\pi_{u^*}(\hat{u})\le \tfrac{(M+1)R}{(M+1)R+\delta},
\end{equation}
where $u^*$ is defined in phase zero of Algorithm~\refeq{alg:barrier}.
\end{proposition}
\begin{proof}
We have 
\[
	\vert  \til{g}_0(u)-\til{g}_0(\hat{u})\vert  \le \vert g_0(x)-g_0(\hat{x}) \vert +\vert\eta-\hat\eta \vert  
	\le  \tfrac{M}{M+1}\delta + \tfrac{\delta}{M+1}  = \delta,
\]
implying that  \[ \til{g}_0(u)\le  \til{g}_0(\hat{u})+\delta = g_0(\hat{x})-\hat{\eta}+\delta=0.\]
Analogously, 	for $i=0,1,\ldots,m$, we have 
\[
 | g_i(x)-g_i(\hat{x})| \le M \| x-\hat{x}\|\le \tfrac{M}{M+1}\delta \le  \delta.	
\]
Using triangle inequality, we have $\til{g}_i(u)=g_i(x)\le g_i(\hat{x})+\delta $=0. The last constraint in~(\ref{prob:lin-form}) is trivially satisfied for sufficiently large $R$. Therefore,  $u$ is a feasible point of (\ref{prob:lin-form}).

Let  $t^+=\tfrac{(M+1)\|\hat{u}-u^* \|}{(M+1)\|\hat{u}-u^*\|+\delta}$, then from the above analysis, we know that the point
\[
u^+ =u^*+\tfrac{1}{t^+} (\hat{u}-u^*)  = \hat{u}+\tfrac{\delta(\hat{u}-u^*)}{(M+1)\|\hat{u}-u^*\|}	
\]
must be  a feasible solution. Using the last constraint $\|u\|\le R$, we immediately obtain the bound~(\ref{eq:bound-pistar}).
\end{proof}

Using~\cite[Theorem~4.5.1]{nemirovski2004interior} and Proposition~\ref{prop:Minko}, we can derive the total complexity of solving the diagonal QCQP.
\begin{theorem}\label{thm:barrier-bound}
Under the assumptions of Proposition~\refeq{prop:Minko}, the total number of Newton steps to get an $\vep$ solution is 
\[
N_{\vep} = \Ocal(1)\sqrt{m+2}\ln\left( \tfrac{(m+2) V((M+1)R+\delta)}{\delta\vep} +1\right).
\]
\end{theorem}

\begin{corollary}
In the inexact {\GD} method, assume that the subproblems are solved by Algorithm~\refeq{alg:barrier} and the returned solution satisfies the inexactness requirement in Theorem~\refeq{thm:inexact-rate}. Then, to get an $\Ocal(\epsilon, \epsilon)$ Type-II KKT point, the overall arithmetic cost of Algorithm~\refeq{alg:barrier} is 
\begin{equation*}
\Tcal = \Ocal\big(\min\{m, d\}\cdot m^{1.5}d \cdot \frac{1}{\vep} \ln\brbra{\frac{1}{\vep}}\big).\label{eq:total-flops}	
\end{equation*}
\end{corollary}
\begin{proof}
According to Theorem~\ref{thm:inexact-rate}, the total number of {\GD} is $K=\Ocal(1/\vep)$. In the  $k$-th iteration of {\GD}, we set the error criteria $\nu=\Ocal(\tfrac{1}{k^2})$ and $\vep=\Ocal(\tfrac{1}{k^2})$. Theorem~\ref{thm:barrier-bound} implies that the number of Newton steps is $N_k=\Ocal(\sqrt{m}\ln (k))$. Therefore, the total number of Newton steps in {\GD} is
$T_K = \tsum_{k=0}^K N_k = \Ocal\brbra{\sqrt{m}\frac{1}{\vep}\ln\brbra{\frac{1}{\vep}}}. $
Combining this result with \eqref{eq:cost-newton} gives us the desired bound.
\end{proof}
\begin{remark}
First,  at the $k$-th step of {\GD}, we need $\log(k)$ iterations of interior point methods, of which the complexity order  is equally contributed by the two phases of IPM. Specifically, we first require  $\Ocal(\ln(k))$ Newton steps to pull the iterates from near the boundary to  the proximity of the central path, and then require $\Ocal(\ln(k))$ to obtain an $\Ocal(1/k^2)$-accurate solution. 
Second, it is interesting to consider the case when the constraint is far less than the feature dimensionality, namely, $m \ll d $. We observe that the total computation 
\[ \Ocal\big(d m^{2.5} K\ln K\big) \]
is linear in dimensionality.
Third, despite the simplicity,  the basic barrier method offers a relatively stronger approximate solution than what is needed in Theorem~\ref{thm:inexact-rate}, the feasibility of the solution path allows us to weaken the  assumption  to $\hat{\vep}_k=0$. 
Nevertheless, besides  our approach, it is possible to employ  long-step and infeasible primal-dual interior point methods which may give a better empirical performance. 
\end{remark}

\subsection{Solving subproblems with the first-order method}\label{sec:ConEx+LCGD} 
In this section, we use a previously proposed ConEx method~\cite{boob2019proximal} to solve the subproblem \eqref{subproblem} when general proximal functions $\chi_{i}$ are present. Then, we  analyze the overall complexity of \GD~method with {\conex} method as a subproblem solver.  First, we formally state the extended version of problem \eqref{prob:diag-qp} as follows:
\begin{equation} \label{prob:subprob-2}
	\begin{split}
		\min_{x \in X} \quad &\phi_0(x)\coloneq g_0(x) + \chi_{0}(x)\\
		\text{s.t.} \quad &\phi_i(x) \coloneq g_i(x) + \chi_{i}(x) \le 0, \quad i = 1, \dots, m.
	\end{split}
\end{equation} 
For the application of {\conex}  for the subproblem, we need access to a  convex compact set $X$ such that $\cap_i\dom{\chi_{i}} \subseteq X$. Moreover, $X$ is a ``simple'' set in the sense that it allows easy computation of  the proximal operator of $\chi_{0}(x) + \tsum_{i=1}w_i\chi_{i}(x)$ for any given weights $w_i, i =1, \dots, m$. Such assumptions are not very restrictive as many machine learning and engineering problems explicitly seek the optimal solution from a bounded set. Under these assumptions, we apply {\conex} to solve the subproblem \eqref{subproblem} of {\GD}. We now reproduce a simplified optimality guarantee of the {\conex} method below without necessarily going into the details of the algorithm.
\begin{theorem}\cite{boob2019proximal}\label{thm:conv_conex}
	Let $x$ be the output of {\conex} after $T$ iterations for  problem~\eqref{prob:subprob-2}. Assume that $\phi_0$ is a strongly convex function and $(\wtil{x}, \wtil{\lambda})$ is the optimal primal-dual solution. Moreover, Let $B$ be a parameter of the {\conex} method which satisfies $B> \gnorm{\wtil{\lambda}}{}{}$. Then, the solution $x$ satisfies
	\begin{align*}
		\phi_0(x) - \phi_0(\wtil{x}) &\le O\paran[\big]{\tfrac{1}{T^2}\paran{ B^2 + \gnorm{\wtil{\lambda}}{}{2} }},\\
		\gnorm{[\phi(x)]_+ }{}{} &\le O\paran[\big]{\tfrac{1}{T^2}\paran{ B^2 + \gnorm{\wtil{\lambda}}{}{2} }}.
	\end{align*}
\end{theorem}

Even though {\conex}  can be applied to a wider variety of convex function constrained problems, it has two vital and intricate issues that need to be addressed in our context:
\begin{enumerate}
	\item The solution path of {\conex}  can be arbitrarily infeasible in the early iterations, while the successive iterations make the solutions infeasibility smaller. Note that the approximation criterion in Definition~\ref{def:inexact-criterion} requires guarantees on the amount of infeasibility. %
	This implies {\conex}  has to run a significant number of iterations before getting sufficiently close to the feasible set.
	\item Since {\conex} is a primal-dual method, its convergence guarantees depend on the optimal dual solution $\lambda^*$. Moreover, a bound on the dual, $B (> \gnorm{\lambda^*}{}{})$, is required to implement the algorithm to achieve an accelerated convergence rate of $O(1/T^2)$ for strongly convex problems.
\end{enumerate} 
 From Theorem \ref{thm:conv_conex}, it is clear that {\conex}  requires a bound $B$. This requirement naturally leads to two cases: (1) bound $B$ can be estimated apriori, e.g., see Lemma \ref{lem:strong_feas_prior_bound}; and (2) bound $B$ is known to exist but cannot be estimated, e.g., see Theorem \ref{prop:lcgd:bound-dual}. Both cases have different convergence rates for the subproblem which leads to different overall computational complexity.

\paragraph{Case 1: $B$ can be estimated apriori.}
In this case, we do not need to estimate $B^k$ as in \eqref{eq:B^k_bound}. Using the bound $B$, we can get accelerated convergence of \conex~in accordance with Theorem \ref{thm:conv_conex} which leads to better performance of the \GD~method. The corollary below formally states the total computational complexity of \GD~method for this case.
\begin{corollary}
    If an explicit value of $B$ is known, the \GD~method with \conex~as subproblem solver obtains $O(\tfrac{1}{K},\tfrac{1}{K})$ type-II KKT point in $O(K^2)$ computations.
\end{corollary}
\begin{proof}
    According to Theorem \ref{thm:conv_conex}, the required \conex~iterations for each subproblem can be bounded by 
    \[T^k = O(\tfrac{B}{\sqrt{\epsilon_k}}).\]
    Since $B$ is a constant, we have $T^k = O(\epsilon_k^{-1/2}) = O(k)$. Finally, we have total computations $\tsum_{k=1}^K T^k = O(K^2)$. Hence, we conclude the proof.
\end{proof}

 \paragraph{Case 2: $B$ is known to exist but cannot be estimated.}
 For the subproblem~\eqref{subproblem}, we can easily find $B^k > \gnorm{\wtil{\lambda}^{k+1}}{}{}$ by using the difference in levels of successive iterations. This bound is weak, especially in the limiting case as it does not take into account
\begin{proposition}\label{prop:bound_B^k}
	For subproblem \eqref{subproblem}, we have 
	\begin{equation}\label{eq:B^k_bound}
		\gnorm{\wtil{\lambda}^{k+1}}{}{} \le \tfrac{\psi_{0}(x^k) - \psi_{0}^*}{\min_{i \in [m]}\delta^k_i}.
	\end{equation}
\end{proposition}
\begin{proof}
	By Slater's condition, we know that $\wtil{\lambda}^{k+1}$ exists. Then, due to saddle point property of $(\wtil{x}^{k+1}, \wtil{\lambda}^{k+1})$, we have for all $x \in X$
	\begin{align*}
		\psi^k_{0}(x)  + \wtil{\lambda}^{k+1}[\psi^k(x) -\eta^{k+1}] &\ge \psi^k_{0}(\wtil{x}^{k+1})  + \wtil{\lambda}^{k+1}[\psi^k(\wtil{x}^{k+1})  -\eta^{k+1}] \\
		&=\psi^k_{0}(\wtil{x}^{k+1}) ,
	\end{align*}
	where the equality follows by complementary slackness.
	Using $x = x^k$ in the above relation and noting that $x^k$ satisfies $\psi(x^k) \le \psi^{k-1}(x^k) \le \eta^k$, we have $\eta^{k+1} - \psi^{k}(x^k) = \eta^{k+1} - \psi(x^k) \ge \eta^{k+1} - \eta^k = \delta^k$ implying that
	\begin{align*}
    \psi_{0}^k(x^k) - \psi_{0}^k(\wtil{x}^{k+1}) &\ge \inprod{\wtil{\lambda}^{k+1}}{\delta^k} \ge \gnorm{\lambda^{k+1}}{1}{} \cdot \min_{i \in [m]}\delta_{i}^k \ge \gnorm{\lambda^{k+1}}{}{} \cdot \min_{i \in [m]}\delta_{i}^k,
    \end{align*}
	where second inequality follows from $\wtil{\lambda}^{k+1} \ge 0$ and $\delta^k > 0$ and last inequality follows due to the fact that $\gnorm{\wtil{\lambda}^{k+1}}{1}{} \ge \gnorm{\wtil{\lambda}^{k+1}}{}{}$. We can further upper bound the LHS of the above relation as follows
    \begin{equation*}
	   \psi^k_{0}(x^k) - \psi_{0}^k(\wtil{x}^{k+1})= \psi_{0}(x^k) - \psi_{0}^k(\wtil{x}^{k+1}) \le \psi_{0}(x^k) - \psi_{0}(\wtil{x}^{k+1}) \le  \psi_{0}(x^k) - \psi_{0}^*,
    \end{equation*}
    where the last inequality follows since $\wtil{x}^{k+1}$ is feasible for the original problem \eqref{prob:main}.
    Combining the above two relations, we obtain \eqref{eq:B^k_bound}. Hence, we conclude the proof.
\end{proof}
We now state the final computation complexity of \GD~with\conex~which uses the bound in \eqref{eq:B^k_bound}.
\begin{corollary}
    If an explicit value of $B$ is not known, the \GD~method with \conex~as subproblem solver obtains $O(\tfrac{1}{K},\tfrac{1}{K})$ type-II KKT point in $O(K^4)$ computations.
\end{corollary}
\begin{proof}
    Using Proposition~\ref{prop:bound_B^k}, we can set $B^k := \tfrac{\psi_{0}(x^k)-\psi_{0}^*}{\delta_{i_*}^k}$ where $i_* := \argmin_{i \in m} \eta_i - \eta^0_i$. Then, required {\conex} iterations $T^k$ can be bounded by 
    \[ T^k = O\big(\tfrac{B^k}{\sqrt{\epsilon_{k}}}\big).\]
    Finally, in view of \eqref{eq:int_rel21} and the fact that $\tsum_{i=0}^\infty \epsilon_i \le \gnorm{\eta - \eta^0}{}{}$ implies that $B^k \le \tfrac{1}{\delta_{i_*}^k}[\psi_{0}(x^0) - \psi_{0}^*+ \gnorm{\eta -\eta^0}{}{}]$ for all $k$. Moreover, for all $k \le K$, we have $\epsilon_k = \frac{\delta^k_{i_*}}{2}$. Hence, we get $T^k = O\rbra{{\epsilon_k^{-3/2}}} = O(k^3)$. Finally, we have $\sum_{k= 1}^{K} T^k = O(K^4)$ which is the overall computational complexity of \GD~method with {\conex} as subproblem solver to obtain $(O(\tfrac{1}{K}),  O(\frac1K))$ type-II KKT point.  
\end{proof}
\begin{remark}\label{rem:grad_compx_vs_comp_compx}[Gradient complexity vs. computational complexity]
		Note that evaluating the gradient of $\psi_{i}^k(x)$ is relatively simple  since it does not involve any new computation of $\grad f_i(x)$. In that sense, the entire inner loop requires only one $\grad f_i$ computation;  hence the total gradient complexity of $\grad f_i$ equals the total outer loops of inexact \GD. On the other hand, inner loop computation does contribute to the problem's computational complexity. However, such iterations are expected to be very cheap given the ease of obtaining gradients for the QP subproblem~\eqref{subproblem} with identity hessian matrices.
\end{remark}

\section{{\GD} for convex optimization}\label{sec:convex-programming}
	In this section, we establish the complexity of  {\GD} (i.e., Algorithm~\ref{alg:lcgd})  when the objective $f_0$ and constraint $f_i,  i \in [m]$ are convex. In particular, we consider two  convex  problems, depending on whether $f_0$ is convex or strongly convex. To provide a combined analysis of the two cases, we assume the following:
	\begin{assumption}
		$f_0(x)$ is $\mu_0$-convex function for some $\mu_0\ge 0$. Namely,
		\[f_0(x)\ge f_0(y)+\inner{\nabla f_0(y)}{x-y}+\tfrac{\mu_0}{2}\gnorm{x-y}{}{2}, \hspace{1em}\text{for any }x,y\in \Rbb^d.\] 
	\end{assumption}
Note that if $\mu_0= 0$ then $f_0$ is simply a convex function.  Now we provide the convergence rate of {\GD}  to optimality.

For more generality, we consider an inexact variant of {\GD} for which an  approximate solution in terms of Definition~\ref{def:inexact-criterion} is returned in each iteration. Let $(\wtil{x}^{k+1}, \wtil\lambda^{k+1})$ be the saddle-point solution $\Lcal_k(x, \wtil\lambda)$, i.e., $\wtil{x}^{k+1}$ is an exact solution of the subproblem \eqref{subproblem}. 
First, we extend the three-point inequality in Lemma~\ref{lem:3pt} for an inexact solution.
\begin{lemma}\label{lem:3-point-lem-appx}
	Let $z^+$ be an $\epsilon$-approximate solution of problem $\min_{x \in \Rbb^d}\{g(x) + \tfrac{\gamma}{2}\gnorm{x-z}{}{2}\}$ where $g(x)$ is a proper, lsc. and convex function. Then, 
	\begin{equation}\label{eq:3pt-mid-5}
		g(z^+) -g(x) \le \tfrac{\gamma}{2}\bsbra{ \gnorm{z-x}{}{2} - \gnorm{z^+-x}{}{2} - \gnorm{z^+-z}{}{2}} + \epsilon + \mathfontsize{11}{\sqrt{2\gamma\epsilon}}\, \gnorm{z^+-x}{}{}.
	\end{equation}
\end{lemma}
\begin{proof}First, let $x^+$ be the optimal solution of $\min_{x \in \Rbb^d}\{g(x) + \tfrac{\gamma}{2}\gnorm{x-z}{}{2}\}$. In view of Lemma~\ref{lem:3pt}, for any $x$, we have
\begin{equation}\label{eq:3pt-mid-1}
g(x^+)+\frac{\gamma}{2}\norm{x^+-z}^2 + \frac{\gamma}{2}\norm{x-x^+}^2
\le g(x) + \frac{\gamma}{2}\norm{x-z}^2.
\end{equation}
Placing $x=z^+$ above, we have
\begin{equation}\label{eq:3pt-mid-2}
g(x^+)+\frac{\gamma}{2}\norm{x^+-z}^2 + \frac{\gamma}{2}\norm{z^+-x^+}^2
\le g(z^+) + \frac{\gamma}{2}\norm{z^+-z}^2.
\end{equation}
On the other hand, by the definition of $\epsilon$-solution, we have
\begin{equation}\label{eq:3pt-mid-3}
		g(z^+)+\frac{\gamma}{2}\norm{z^+-z}^2 \le g(x^+)+\frac{\gamma}{2}\norm{x^+-z}^2 + \epsilon.
\end{equation}
Combining the above two inequalities gives
\begin{equation}\label{eq:3pt-mid-4}
	\frac{\gamma}{2}\norm{z^+-x^+}^2 \le \epsilon.
\end{equation}
Summing up \eqref{eq:3pt-mid-1} and \eqref{eq:3pt-mid-3} again and then rearranging the terms, we get
	\begin{align*}
	g(z^+) - g(x) 
	& \le \frac{\gamma}{2}\norm{x-z}^2 - \frac{\gamma}{2}\norm{z^+-z}^2 - \frac{\gamma}{2}\norm{x-x^+}^2 +\epsilon \nonumber \\
	& \le \frac{\gamma}{2}\norm{x-z}^2 - \frac{\gamma}{2}\norm{z^+-z}^2 - \frac{\gamma}{2}\norm{x-z^+}^2 +\gamma\norm{x-z^+}\norm{z^+-x^+} +\epsilon,
	\end{align*}
	where the last inequality uses the fact that $-\frac{1}{2}\norm{a+b}^2\le -\frac{1}{2}\norm{a}^2-\inprod{a}{b}\le -\frac{1}{2}\norm{a}^2+\norm{a}\norm{b}$ with $a=x-z^+$ and $b=z^+-x^+$.
	Finally, combining the above two results gives the desired inequality~\eqref{eq:3pt-mid-5}.
\end{proof}
Using the above lemma, we provide the main convergence property of {\GD} for convex optimization.
\begin{lemma} \label{lem:suff-descent-cvx-inexact}
	Let $x$ be {feasible solution}. Then, we have
	\begin{align}
		\psi_{0}(x^{k+1}) - \psi_{0}(x) &\le \inprod{\wtil\lambda^{k+1}}{\psi(x) - \eta^k} + \tfrac{L_0+ \inprod{\wtil\lambda^{k+1}}{L} -\mu_0}{2} \gnorm{x^k-x}{}{2} - \tfrac{L_0 + \inprod{\wtil\lambda^{k+1}}{L}}{2} \gnorm{x^{k+1}-x}{}{2} \nonumber\\
		&\quad+ {2}\epsilon_k + 
		\mathfontsize{9}{\sqrt{2(L_0+ \inprod{\wtil\lambda^{k+1}}{L}) \epsilon_k}} \gnorm{x^{k+1}-x}{}{}.\label{eq:appx-3pt}
	\end{align}
\end{lemma}
\begin{proof}Note that
	\begin{align}
		{\psi_{0}(%
			{x}^{k+1})} & \ {\myrel{(i)}{\le}}\ \psi_{0}^k(%
			{x}^{k+1}) \myrel{(ii)}{\le} \psi_{0}^k(\wtil{x}^{k+1}) + \epsilon_k \nonumber\\
		&\myrel{(iii)}{=} \psi_{0}^k(\wtil{x}^{k+1}) + \inprod{\wtil\lambda^{k+1}}{\psi^k(\wtil{x}^{k+1}) - \eta^k} +\ \epsilon_k \nonumber \\
		&{\le} \psi_{0}^k(x^{k+1}) + \inprod{\wtil\lambda^{k+1}}{\psi^k(x^{k+1}) - \eta^k} {+\ \epsilon_k}\nonumber \\
		&\myrel{(iv)}{\le} \psi_{0}^k(x) + \inprod{\wtil\lambda^{k+1}}{\psi^k(x) - \eta^k} - \tfrac{L_0 + \inprod{\wtil\lambda^{k+1}}{L}}{2} \gnorm{x^{k+1}-x}{}{2} \nonumber \\
		& \quad + 2\epsilon_k +\mathfontsize{8}{\sqrt{2(L_0+ \inprod{\wtil\lambda^{k+1}}{L}) \epsilon_k}} 
		\gnorm{x^{k+1}-x}{}{},\label{eq:int_rel16}
	\end{align}
	where {\sffamily(i)} follows from the definition of $\psi_{0}^k$, {\sffamily (ii)} follows since $x^{k+1}$ is an $\epsilon_k$ solution of \eqref{subproblem}, {\sffamily(iii)} follows by complementary slackness for the optimal primal-dual solution for \eqref{subproblem} and {\sffamily(iv)} follows from Lemma~\ref{lem:3-point-lem-appx}. In particular, we use $g(x) + \tfrac{\gamma}{2}\gnorm{x-z}{}{2} = \psi_0^k(x) + \inprod{\wtil\lambda^{k+1}}{\psi^k(x) - \eta^k}$ with $z = x^k$, $z^+ = x^{k+1}$, $\epsilon = \epsilon_k$ and $\gamma = L_0 + \inprod{\wtil\lambda^{k+1}}{L}$. Note that $x^{k+1}$ is an $\epsilon_k$-approximate solution for $\min_{x \in \Rbb^d} \psi_{0}^k(x) + \inprod{\wtil\lambda^{k+1}}{\psi^k(x) -\eta^k}$ due to Definition \ref{def:inexact-criterion}.

Finally, note that 
	\begin{align*}
		\psi_{0}^k(x) + \tfrac{\mu_0}{2} \gnorm{x-x^k}{}{2} &\le \psi_{0}(x)  + \tfrac{L_0}{2}\gnorm{x-x^k}{}{2},\\
		\psi_{i}^k(x) &\le \psi_i(x) + \tfrac{L_i}{2} \gnorm{x-x^k}{}{2} . 
	\end{align*}
	Using the above two relations in \eqref{eq:int_rel16}, we obtain \eqref{eq:appx-3pt}. Hence, we conclude the proof.
\end{proof}

Let	$x^*$ be an optimal solution of \eqref{prob:main}  and $\wtil{D} := \max\{\gnorm{x-y}{}{}:x, y \in \dom{\chi_{0}},  \psi_i(x) \le \eta_i, \psi_{i}(y) \le \eta_{i}, \text{ for all } i\in [m] \}$.
Now, we show convergence rate guarantees.
\begin{theorem}
Consider general convex optimization problems with $\mu_0 = 0$. Suppose Assumption~\refeq{assu:y-bounded} is satisfied and set $\delta^k = \tfrac{(\eta-\eta^0)}{(k+1)(k+2)}$. 
	Then we have
	\begin{align}
		\psi_{0}(\bar{x}_K) - \psi_{0}(x^*) \le \tfrac{L_0+ B\gnorm{L}{}{}}{(K+1)}\bracket[\big]{\wtil{D}^2 + \tfrac{(4B+2)\gnorm{\eta-\eta^0}{}{}}{L_0} + \wtil{D}\sqrt{\tfrac{\gnorm{\eta-\eta^0}{}{}}{L_0}} +\tfrac{\gnorm{\eta-\eta^0}{}{}}{L_0}\tfrac{\log{K}}{K} }\label{eq:cvx-convegence rate}
	\end{align}
\end{theorem}
\begin{proof}
	From Lemma \ref{lem:suff-descent-cvx-inexact} with $\mu_0 = 0$ for convex part and $\psi(x^*) \le \eta$, we have 
	\begin{align*}
		\psi_{0}(x^{k+1}) - \psi_{0}(x^*) 
		&= \inprod{\wtil\lambda^{k+1}}{\eta-\eta^k} + \tfrac{L_0 + \inprod{\wtil\lambda^{k+1}}{L}}{2}\gnorm{x^k-x^*}{}{2} - 
		\tfrac{L_0 + \inprod{\wtil\lambda^{k+1}}{L}}{2}\gnorm{x^{k+1}-x^*}{}{2} \\
		&\quad+ 2\epsilon_k 
		+ \mathfontsize{9}{\sqrt{2(L_0+ \inprod{\wtil\lambda^{k+1}}{L}) \epsilon_k}} \gnorm{x^{k+1}-x}{}{}.
	\end{align*}
	Dividing both sides by $\tfrac{L_0 + \inprod{\wtil\lambda^{k+1}}{L}}{2}$, we have
	\[\tfrac{2[\psi_{0}(x^{k+1}) - \psi_{0}(x^*) ]}{L_0 + \inprod{\wtil\lambda^{k+1}}{L}} \le \gnorm{x^k-x^*}{}{2} - \gnorm{x^{k+1}-x^*}{}{2} +  \tfrac{2\inprod{\wtil\lambda^{k+1}}{\eta - \eta^k}}{L_0 + \inprod{\wtil\lambda^{k+1}}{L}} + \tfrac{4\epsilon_k}{L_0 + \inprod{\wtil\lambda^{k+1}}{L}} + \sqrt{\tfrac{8\epsilon_k}{L_0 + \inprod{\wtil\lambda^{k+1}}{L}} }\gnorm{x^{k+1}-x^*}{}{}
	\]
	Note that the sequence $\{\wtil\lambda^{k+1}\}$ is uniformly bounded above such that $\gnorm{\wtil\lambda^{k+1}}{}{} \le B$ for all $k \ge 0$. Using this fact and the above relation, we have
	\begin{align}
		\tfrac{2[\psi_{0}(x^{k+1}) - \psi_{0}(x^*) ]}{L_0 + B\gnorm{L}{}{}} \le \gnorm{x^k-x^*}{}{2} - \gnorm{x^{k+1}-x^*}{}{2} + \tfrac{2B\gnorm{\eta-\eta^k}{}{}}{L_0} + \tfrac{4\epsilon_k}{L_0} + \sqrt{\tfrac{8\epsilon_k}{L_0}}\gnorm{x^{k+1}-x^*}{}{}.\label{eq:int_rel18}
	\end{align}
	Using $\delta_k = \tfrac{\eta-\eta^0}{(k+1)(k+2)}$ and $\epsilon_k = \tfrac{\gnorm{\eta-\eta^0}{}{}}{2(k+1)(k+2)}$, we have $x^{k}$ is strictly feasible solutions for \eqref{subproblem} for all~$k$. Hence, under Assumption \ref{assu:y-bounded}, we can follow the steps of Theorem \ref{prop:lcgd:bound-dual} to show uniform bound $B$ on sequence~$\{\gnorm{\wtil\lambda^k}{}{}\}$.
	Using these values in \eqref{eq:int_rel18}, we have
	\begin{equation}\label{eq:convergence_bd_inexact}
		\tfrac{2[\psi_{0}(x^{k+1}) - \psi_{0}(x^*) ]}{L_0 + B\gnorm{L}{}{}} \le \gnorm{x^k-x^*}{}{2} - \gnorm{x^{k+1}-x^*}{}{2} + \tfrac{2B\gnorm{\eta-\eta^0}{}{}}{L_0(k+1)} + \tfrac{2\gnorm{\eta-\eta^0}{}{}}{L_0(k+1)(k+2)} + \sqrt{\tfrac{\gnorm{\eta-\eta^0}{}{}}{L_0}}\tfrac{1}{k+1}\gnorm{x^{k+1}-x^*}{}{}.
	\end{equation}
	Due to the optimality of the exact solution $\wtil{x}^{k+1}$, we have $\psi^k_{0}(\wtil{x}^{k+1}) \le \psi^k_{0}(x^k)  = \psi_{0}(x^k)$. We also have
		$\psi_0(x^{k+1}) \le \psi_0^k(x^{k+1}) \le \psi_{0}^k(\wtil{x}^{k+1}) + \epsilon_k$. Combining these two relations, we get: 
		\[\psi_{0}(x^{k+1})  \le \psi_{0}(x^k) + \epsilon_k.\] 
		Effectively, inexact \GD~method is almost (up to an additive error of $\epsilon_k$) a descent method. Using this relation recursively, we have 
		\begin{align*}
			\psi_{0}(x^K) &\le \psi_{0}(x^{k+1}) + \tsum_{i=k+1}^{K-1}\epsilon_i.\\
			&\le \psi_0(x^{k+1}) + \tfrac{\gnorm{\eta-\eta^0}{}{}}{2}\tsum_{i=k+1}^{K-1}\tfrac{1}{(i+1)(i+2)}\\
			&= \psi_0(x^{k+1}) + \tfrac{\gnorm{\eta-\eta^0}{}{}}{2}\tfrac{(K-k-1)}{(k+2)(K+1)}.
		\end{align*}
		Using the above relation in \eqref{eq:convergence_bd_inexact}, we have
		\begin{align*}
			\tfrac{2[\psi_{0}(x^K) - \psi_{0}(x^*)]}{L_0 + B\gnorm{L}{}{}}  &\le \gnorm{x^k-x^*}{}{2} - \gnorm{x^{k+1}-x^*}{}{2} + \tfrac{2B\gnorm{\eta-\eta^0}{}{}}{L_0(k+1)} + \tfrac{\added{2}\gnorm{\eta-\eta^0}{}{}}{L_0(k+1)(k+2)} + \sqrt{\tfrac{\gnorm{\eta-\eta^0}{}{}}{L_0}}\tfrac{1}{k+1}\gnorm{x^{k+1}-x^*}{}{}\\
			&\quad+\tfrac{\gnorm{\eta-\eta^0}{}{}}{L_0+B\gnorm{L}{}{}}\tfrac{K-k-1}{(k+2)(K+1)}.
		\end{align*}
	Multiplying the above relation by $k+1$ and summing from $k = 0$ to $K-1$, we have
	\begin{align*}
		\tfrac{K(K+1)[\psi_{0}(\added{x_K}) -\psi_{0}(x^*)]}{L_0 + B\gnorm{L}{}{}}&\le \tsum_{k =0}^{K-1}\gnorm{x^k-x^*}{}{2} + \tfrac{2B\gnorm{\eta-\eta^0}{}{}K}{L_0} + \tfrac{\gnorm{\eta-\eta^0}{}{}\log(K+2)}{L_0} \\
		&\qquad + \sqrt{\tfrac{\gnorm{\eta-\eta^0}{}{}}{L_0}}\tsum_{k =0}^{K-1}\gnorm{x^{k+1}-x^*}{}{} \added{+ \tfrac{\gnorm{\eta-\eta^0}{}{}K}{L_0 + B\gnorm{L}{}{}}}\\
		&\le K\wtil{D}^2 + \tfrac{2(B\added{+1})\gnorm{\eta-\eta^0}{}{}K}{L_0} %
		+\sqrt{\tfrac{\gnorm{\eta-\eta^0}{}{}}{L_0}}K\wtil{D}. %
	\end{align*}
	After rearranging, this relation implies \eqref{eq:cvx-convegence rate}. Hence, we conclude the proof.
\end{proof}

\begin{theorem}\label{thm:inexact-strongly-cvx}
Consider strongly convex problems ($\mu_0>0$) and suppose that Assumption~\refeq{assu:y-bounded} is satisfied. Set $\delta^k = \rho^k(1-\rho)(\eta-\eta^0)$ where \added{$\rho =\tfrac{L_0 - \mu_0}{2(L_0-a\mu_0)}$}, $2\epsilon_k \le a (1-\rho)\rho^k\gnorm{\eta-\eta^0}{}{}$ and $a\in(0,1)$.  Then we have
	\begin{align}\label{eq:inex-str-cvx}
		\psi_{0}(x^K) - \psi_{0}(x^*) & \le \exp\brbra{-\tfrac{(1-a)\mu_0 K}{L_0+B\gnorm{L}{}{}-a\mu_0}} \rbra{L_0+B\gnorm{L}{}{}-\mu_0}\nonumber \\
		&	\quad \big\{\bsbra{\tfrac{(4B+1)}{2(L_0-\mu_0)}+\tfrac{L_0+B\norm{L}+2a\mu_0}{\mu_0(L_0 - \mu_0)}\rbra{1-\rho}} \gnorm{\eta-\eta^0}{}{} + \tfrac{1}{2}\norm{x^0-x^*}^2\big\}. 
	\end{align}
Moreover, if $\epsilon_k=0$, we have
	\begin{equation}\label{eq:ex-str-cvx}
		\psi_{0}(x^K) - \psi_{0}(x^*) \le \exp\brbra{-\tfrac{\mu_0 K}{L_0+B\gnorm{L}{}{}}} \rbra{L_0+B\gnorm{L}{}{}-\mu_0} \Bcbra{\bsbra{\tfrac{(4B+1)}{2(L_0-\mu_0)}} \gnorm{\eta-\eta^0}{}{} + \tfrac{1}{2}\norm{x^0-x^*}^2}. 
	\end{equation}
\end{theorem}
\begin{proof}
	Proceeding similar to the convex case, using Lemma~\ref{lem:suff-descent-cvx-inexact}, we obtain
		\begin{align*}
		\psi_{0}({x^{k+1}}) - \psi_{0}(x^*) 
		&= \inprod{\wtil\lambda^{k+1}}{\eta-\eta^k} + \tfrac{L_0 + \inprod{\wtil\lambda^{k+1}}{L} - \mu_0}{2}\gnorm{x^k-x^*}{}{2} - 
		\tfrac{L_0 + \inprod{\wtil\lambda^{k+1}}{L}}{2}\gnorm{x^{k+1}-x^*}{}{2} \\
		&\quad+ 2\epsilon_k 
		+ \mathfontsize{9}{\sqrt{2(L_0+ \inprod{\wtil\lambda^{k+1}}{L}) \epsilon_k}} \gnorm{x^{k+1}-x}{}{}. 
	\end{align*}
	For $0< a < 1$, we have 
	\begin{align*}
	{\sqrt{2(L_0+ \inprod{\wtil\lambda^{k+1}}{L}) \epsilon_k}} \gnorm{x^{k+1}-x}{}{} \le \tfrac{L_0+\inprod{\wtil\lambda^{k+1}}{L}}{a\mu_0}\epsilon_k + \tfrac{a\mu_0}{2}	\norm{x^{k+1}-x^*}^2.
	\end{align*}
	Combining the above two results, we have
	\begin{align*}
	\psi_{0}({x^{k+1}}) - \psi_{0}(x^*) 
	& \le \inprod{\wtil\lambda^{k+1}}{\eta-\eta^k} + \tfrac{L_0 + \inprod{\wtil\lambda^{k+1}}{L} - \mu_0}{2}\gnorm{x^k-x^*}{}{2}
		- \tfrac{L_0 + \inprod{\wtil\lambda^{k+1}}{L}-a\mu_0}{2}\gnorm{x^{k+1}-x^*}{}{2} \\
		& \quad + \brbra{2+\tfrac{L_0+\inprod{\wtil\lambda^{k+1}}{L}}{a\mu_0}} \epsilon_k.
	\end{align*} 
Let us denote 
	\[\Gamma_k=
	\begin{cases}
		1 & \text{ if } k = 0;\\
		\tfrac{L_0+\inner{\wtil\lambda^k}{L}-a\mu_0}{L_0+\inner{\wtil\lambda^k}{L}-\mu_0} \Gamma_{k-1} &\text{ if } k \ge 1.
	\end{cases}\]
	Multiplying both sides of the above inequality by $\tfrac{\Gamma_{k}}{L_0+\inner{\wtil\lambda^{k+1}}{L}-\mu_0}$ and noting that $\eta-\eta^k = \rho^k (\eta-\eta^0)$ (follows by the choice of $\delta^k$), we obtain
	\begin{align}
		& \tfrac{\Gamma_{k}}{L_0+\inner{\wtil\lambda^{k+1}}{L}-\mu_0} \big[\psi_0(x^{k+1})- \psi_0(x^*)\big] \nonumber \\
		& \le \tfrac{\Gamma_{k}}{L_0+\inner{\wtil\lambda^{k+1}}{L}-\mu_0} \rho^k\inner{\wtil\lambda^{k+1}}{\eta-\eta^0}
		 +\tfrac{\Gamma_k}{2} \gnorm{x^*-x^{k}}{}{2}  - \tfrac{\Gamma_{k+1}}{2} \gnorm{x^*-x^{k+1}}{}{2} 
		 + \tfrac{ L_0+\inprod{\wtil\lambda^{k+1}}{L}+2a\mu_0}{a\mu_0(L_0 + \inprod{\wtil\lambda^{k+1}}{L} - \mu_0)} \Gamma_{k}\epsilon_k. \label{eq:int_rel1}
	\end{align}
	Since $\gnorm{\wtil\lambda^k}{}{} \le B$, we have $\paran{\tfrac{L_0 + B\gnorm{L}{}{}-a\mu_0}{L_0 + B\gnorm{L}{}{} - \mu_0}}^k \le \Gamma_k \le \paran{\tfrac{L_0-a\mu_0}{L_0 - \mu_0}}^k$. Moreover,  we have $2\epsilon_k \le a (1-\rho)\rho^k\gnorm{\eta-\eta^0}{}{}$ and $\rho =\tfrac{L_0 - \mu_0}{2(L_0-a\mu_0)}$. Using these relations in \eqref{eq:int_rel1}, we have
	\begin{align}
	&	\tfrac{\Gamma_{k}}{L_0+B\gnorm{L}{}{}-\mu_0} \big[\psi_0(x^{k+1})- \psi_0(x^*)\big]  \nonumber \\
	& \quad  \le \tfrac{B\gnorm{\eta-\eta^0}{}{}}{L_0-\mu_0}\Gamma_{k}\rho^k + \tfrac{L_0+B\norm{L}+2a\mu_0}{a\mu_0(L_0 - \mu_0)}\Gamma_k\epsilon_k +\tfrac{\Gamma_k}{2} \gnorm{x^k-x^*}{}{2}  - \tfrac{\Gamma_{k+1}}{2} \gnorm{x^{k+1}-x^*}{}{2}  \nonumber\\
	& \quad  \le \tfrac{B\gnorm{\eta-\eta^0}{}{}}{L_0-\mu_0}\tfrac{1}{2^k} + \tfrac{L_0+B\norm{L}+2a\mu_0}{\mu_0(L_0 - \mu_0)} \tfrac{1-\rho}{2^{k+1}}\norm{\eta-\eta^0} +\tfrac{\Gamma_k}{2} \gnorm{x^k-x^*}{}{2}  - \tfrac{\Gamma_{k+1}}{2} \gnorm{x^{k+1}-x^*}{}{2}.   \label{eq:int_rel2}
	\end{align}
	Similar to the convex part, we also have 
	\begin{align*}
		\psi_{0}(x^K) &\le \psi_{0}(x^{k+1}) + \tsum_{i=k+1}^{K-1}\epsilon_i 
		\le \psi_0(x^{k+1}) + \tfrac{\gnorm{\eta-\eta^0}{}{}(1-\rho)}{2}\tsum_{i=k+1}^{K-1}\rho^i 
		\le \psi_0(x^{k+1}) + \tfrac{\gnorm{\eta-\eta^0}{}{}\rho^{k+1}}{2}.
	\end{align*}
	Using the above relation into \eqref{eq:int_rel2}, we have
	\begin{align*}
		\tfrac{\Gamma_{k}}{L_0+B\gnorm{L}{}{}-\mu_0} \big[\psi_0(x^K)- \psi_0(x^*)\big] 
		&\le \tfrac{(4B+1)\gnorm{\eta-\eta^0}{}{}}{L_0-\mu_0}\tfrac{1}{2^{k+2}} + \tfrac{L_0+B\norm{L}+2a\mu_0}{\mu_0(L_0 - \mu_0)} \tfrac{1-\rho}{2^{k+1}}\norm{\eta-\eta^0} \nonumber\\
		&\quad  +\tfrac{\Gamma_k}{2} \gnorm{x^k-x^*}{}{2}  - \tfrac{\Gamma_{k+1}}{2} \gnorm{x^{k+1}-x^*}{}{2}.
	\end{align*}
	Summing the above relation from $k = 0$ to $K-1$,  we have
	\begin{align*}
	\tfrac{\Gamma_{K-1}}{L_0+B\gnorm{L}{}{}-\mu_0} \big[\psi_0(x^K)- \psi_0(x^*)\big] & \le \tsum_{k=0}^{K-1}\tfrac{\Gamma_{k}}{L_0+B\gnorm{L}{}{}-\mu_0} \big[\psi_0(x^K)- \psi_0(x^*)\big] \\
	 & \le \bsbra{\tfrac{(4B+1)}{2(L_0-\mu_0)}+\tfrac{L_0+B\norm{L}+2a\mu_0}{\mu_0(L_0 - \mu_0)}\rbra{1-\rho}} \gnorm{\eta-\eta^0}{}{} + \tfrac{1}{2}\norm{x^0-x^*}^2.
	\end{align*}
	Note that
	\[
	\Gamma_{K-1}^{-1}\le \brbra{1-\tfrac{(1-a)\mu_0}{L_0+B\gnorm{L}{}{}-a\mu_0}}^{K}\le \exp\brbra{-\tfrac{(1-a)\mu_0 K}{L_0+B\gnorm{L}{}{}-a\mu_0}}.
	\]
Combining the above two relations we obtain the desired result \eqref{eq:inex-str-cvx}.
\end{proof}

\global\long\def\DCCP{\texttt{DCCP}}%
\global\long\def\CVXPY{\texttt{CVXPY}}

\section{Numerical study}\label{sec:experiments}
	In this section, we conduct some preliminary studies to examine our theoretical results and the performance of  the \GD{} method. The experiments are run on CentOS with Intel Xeon (2.60 GHz)  and 128 GB memory.

\subsection{A simulated study on the QCQP}
In the first experiment, we compare \GD{} with some established open-source solvers such as \CVXPY{}~\cite{diamond2016cvxpy} and \DCCP{}~\cite{shen2016disciplined}.
We consider the penalized Quadratically Constrained Quadratic Program (QCQP) described as follows,
	\begin{equation}\label{eq:ncvx-qcqp}
	\begin{aligned}
	\min_{x\in \Rbb^n} &\quad  \frac{1}{2}x^TQ_0x + b_0^T x +\alpha \norm{x}_1 \\
	\st & \quad  \frac{1}{2}x^TQ_i x + b_i^T x +c_i\le 0, \quad i=1,2,\ldots, m-1 \\
	\st & \quad \|x\|\le r
	\end{aligned}
	\end{equation} 
where each $Q_i$ ($0\le i \le m$) is an $n\times n$ matrix, $b_0,b_1,\ldots, b_m$ are $n$-dimensional real vectors, $\alpha$ is a positive weight on the $\ell_1$ norm penalty, which helps to promote sparse solution.  In the first setting, we consider a convex constrained problem where each $Q_i$ is a positive semidefinite matrix.
We set $Q_i=VDV^\trans$ where $V$ is an $n\times n$ random sparse matrix with density $0.01$, and its nonzero entries are uniformly distributed in $[0,1]$.  $D$ is a diagonal matrix whose diagonal elements are uniformly distributed in between $[0, 100]$.  We set $b_i=10e+v$, where $e$ is a vector of ones and $v\in\Ncal(0, I_{n\times n})$ is sampled from standard Gaussian distribution.  We set $c_i=-10$ to make $x=0$ a strictly feasible initial solution.
Furthermore, we add a ball constraint to ensure that the domain is a compact set. We set $r=\sqrt{20}$ and $\alpha=1$. We fix $m=10$ and explore different dimensions $n$ from the set $\{500, 1000, 2000, 3000, 4000\}$. 

We solve Problem~\eqref{eq:ncvx-qcqp} by both  \texttt{CVXPY} and \GD{}. Both use the initial solution $x=0$.  For \CVXPY{}, we use MOSEK as the internal solver due to its superior performance in quadratic optimization.  In \GD{}, for simplicity, we also solve the diagonal quadratic subproblem by \texttt{MOSEK} through \CVXPY{}. Note that calling the external API repetitively for each \GD{} subproblem only causes more overheads to run \GD{}. Nonetheless, as we shall see,  the standard IPM solvers can still fully leverage the diagonal structure and exhibit fast convergence. 

In Table~\ref{tab:cvxl1qcqp}, we present the experiment results of the compared algorithms. The final objective, the norm of the dual solution (DNorm), and for \GD{},  the maximum dual norm in the solution path (Max DNorm) are reported. All values represent the average of 5 independent runs.  From the results, we observe that while \GD{} does not outperform \CVXPY{} for the small-size problem ($n=500$), \GD{} becomes increasingly favorable  as the problem dimension increases. This justifies the empirical advantage of our proposed approach as we do not need to construct a full Hessian matrix. Moreover, interestingly,  we observe that the dual solution norm $\{\|\lambda^k\|\}$ is increasing, reaching the maximum at the last iteration. This accounts for the equal values of DNorm and Max DNorm. Meanwhile, in all the cases,  the dual remains bounded and the reported dual norm closely aligns with the solution returned by \CVXPY{}. This result confirms our intuition that the dual bound is  intricately tied to the nature of problems. 

\begin{table}[ht]
    \centering
\begin{tabular}{c|c|c|c|c|c|c|c}
\hline
 \multirow{2}{*}{$n$} & \multicolumn{4}{c|}{LCPG} & \multicolumn{3}{c}{CVXPY} \\ 
\cline{2-8}
& Objective & Time & DNorm  & Max DNorm & Objective & Time & DNorm   \\ 
\hline
500 & -1.785e+02  & 9.46e+00  & 1.773e-01   & 1.773e-01 &-1.785e+02  & 3.29e+00  & 1.774e-01  \\ 
\hline
1000 & -9.698e+01  & 1.97e+01  & 1.478e-01   & 1.478e-01 &-9.698e+01  & 1.61e+01  & 1.479e-01  \\ 
\hline
2000 & -5.418e+01  & 6.30e+01  & 1.117e-01   & 1.117e-01 &-5.419e+01  & 7.82e+01  & 1.120e-01  \\ 
\hline
3000 & -4.010e+01  & 1.44e+02  & 9.525e-02   & 9.525e-02 &-4.011e+01  & 2.36e+02  & 9.602e-02  \\ 
\hline
4000 & -3.293e+01 & 2.45e+02  & 8.255e-02   & 8.255e-02 &-3.294e+01 & 4.96e+02  & 8.384e-02   \\ 
\hline
\end{tabular}
    \caption{Comparison of algorithms on convex quadratic problems. Running time is measured in seconds.}
    \label{tab:cvxl1qcqp}
\end{table}

In the second setting of this experiment, we examine the performance of \GD{} on nonconvex constrained optimization. 
Specifically,  we express $Q_i$ as the difference of two matrices: ${Q}_i=P_i-S_i$, where $P_i$ is generated in the  same manner as $Q_i$  in the first setting,  and $S_i = 10 I_{n\times n}$.
  Given the construction of the quadratic components, it is natural to view the function $ \frac{1}{2}x^TQ_i x + b_i^T x +c_i$ as a difference of two convex quadratic functions: $\frac{1}{2}x^TP_i x + b_i^T x +c_i-\frac{1}{2}x^T S_i x$. Leveraging this decomposition, we apply the DC programming, and more specifically, the \texttt{DCCP} framework to solve \eqref{eq:ncvx-qcqp}. Each convex subproblem of \DCCP{} is solved by \texttt{MOSEK} through the CVXPY interface. 
In Table~\ref{tab:ncvx-qcqp}, we describe the performance of \GD{} and the \texttt{DCCP} algorithm.  It can be observed that \GD{} compares favorably against the \texttt{DCCP} solver. Furthermore, the boundedness of the dual for both algorithms is also observed, which is consistent with our intuition.

{
\footnotesize
\begin{table}[ht]
\begin{tabular}{c|c|c|c|c|c|c|c|c}
\hline
 \multirow{3}{*}{$n$} & \multicolumn{4}{c|}{LCPG} & \multicolumn{4}{c}{DCCP} \\ 
\cline{2-9}
& \multirow{2}{*}{Objective} & \multirow{2}{*}{Time} & \multirow{2}{*}{DNorm}  & \multirow{2}{1.2cm}{Max\\ DNorm} & \multirow{2}{*}{Objective} & \multirow{2}{*}{Time} & \multirow{2}{*}{DNorm}  & \multirow{2}{1.2cm}{Max\\ DNorm} \\ 
&&&&&&&&\\
\hline
500 & -3.056e+01 & 1.525e+01  & 1.194e-01 &  1.194e-01&-3.056e+01 & 1.400e+01  & 1.199e-01 &  1.199e-01\\ 
\hline
1000 & -2.112e+01 & 3.350e+01  & 7.602e-02 &  7.608e-02&-2.112e+01 & 5.665e+01  & 7.672e-02 &  7.672e-02\\ 
\hline
2000 & -1.609e+01 & 1.052e+02  & 4.502e-02 &  4.588e-02&-1.609e+01 & 2.678e+02  & 4.571e-02 &  4.571e-02\\ 
\hline
3000 & -1.425e+01 & 1.918e+02  & 3.216e-02 &  3.523e-02&-1.426e+01 & 6.505e+02  & 3.301e-02 &  3.301e-02\\ 
\hline
4000 & -1.329e+01 & 3.837e+02  & 2.485e-02 &  2.738e-02&-1.330e+01 & 1.201e+03  & 2.574e-02 &  2.574e-02\\ 
\hline
\end{tabular}
\caption{Comparison of algorithms on nonconvex quadratic problems.}\label{tab:ncvx-qcqp}
\end{table}
}

\subsection{Study of gradient complexities}
\begin{figure}[h]
    \centering
    \includegraphics[width=0.45\textwidth]{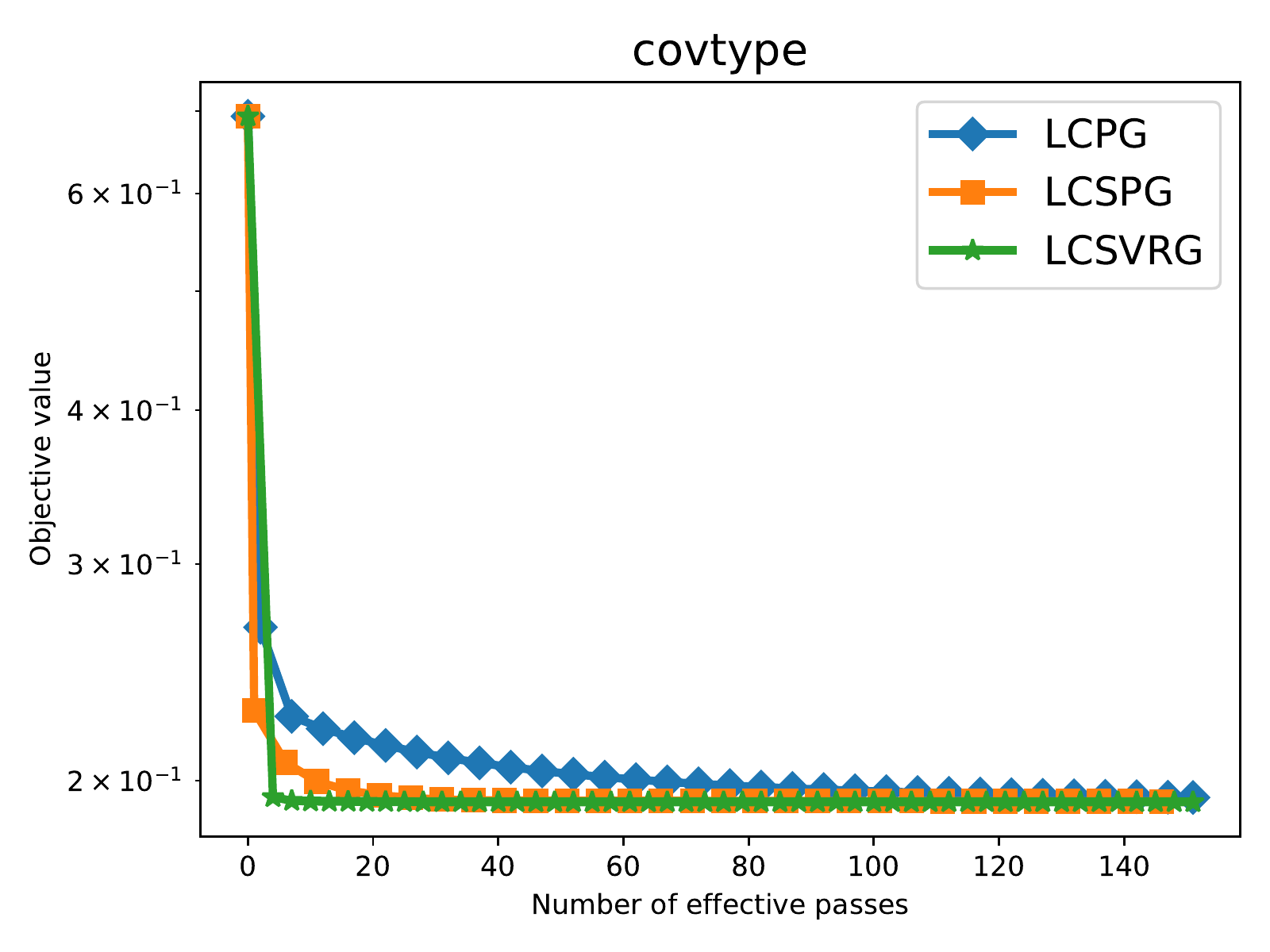}
    \includegraphics[width=0.45\textwidth]{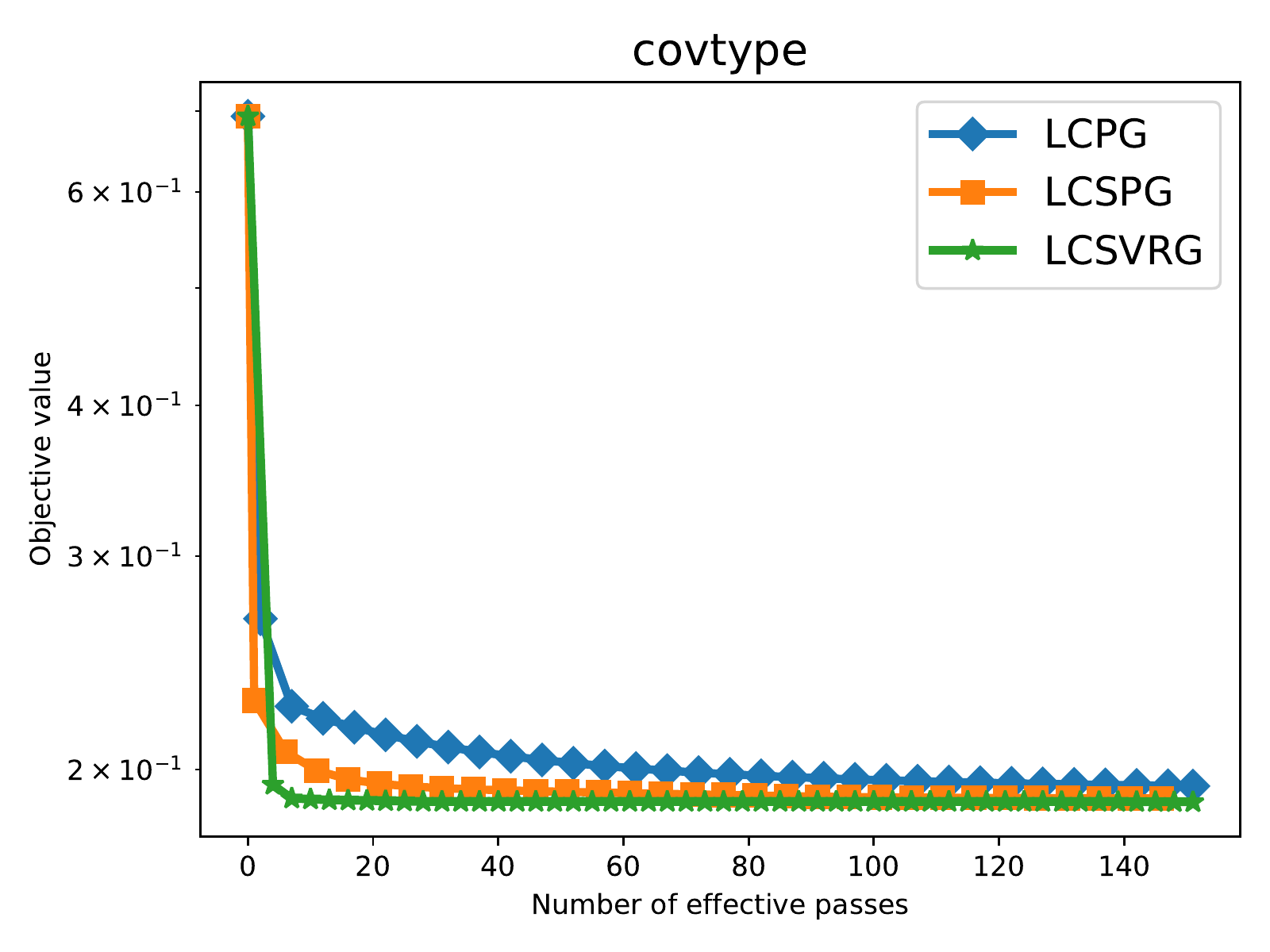}
    \includegraphics[width=0.45\textwidth]{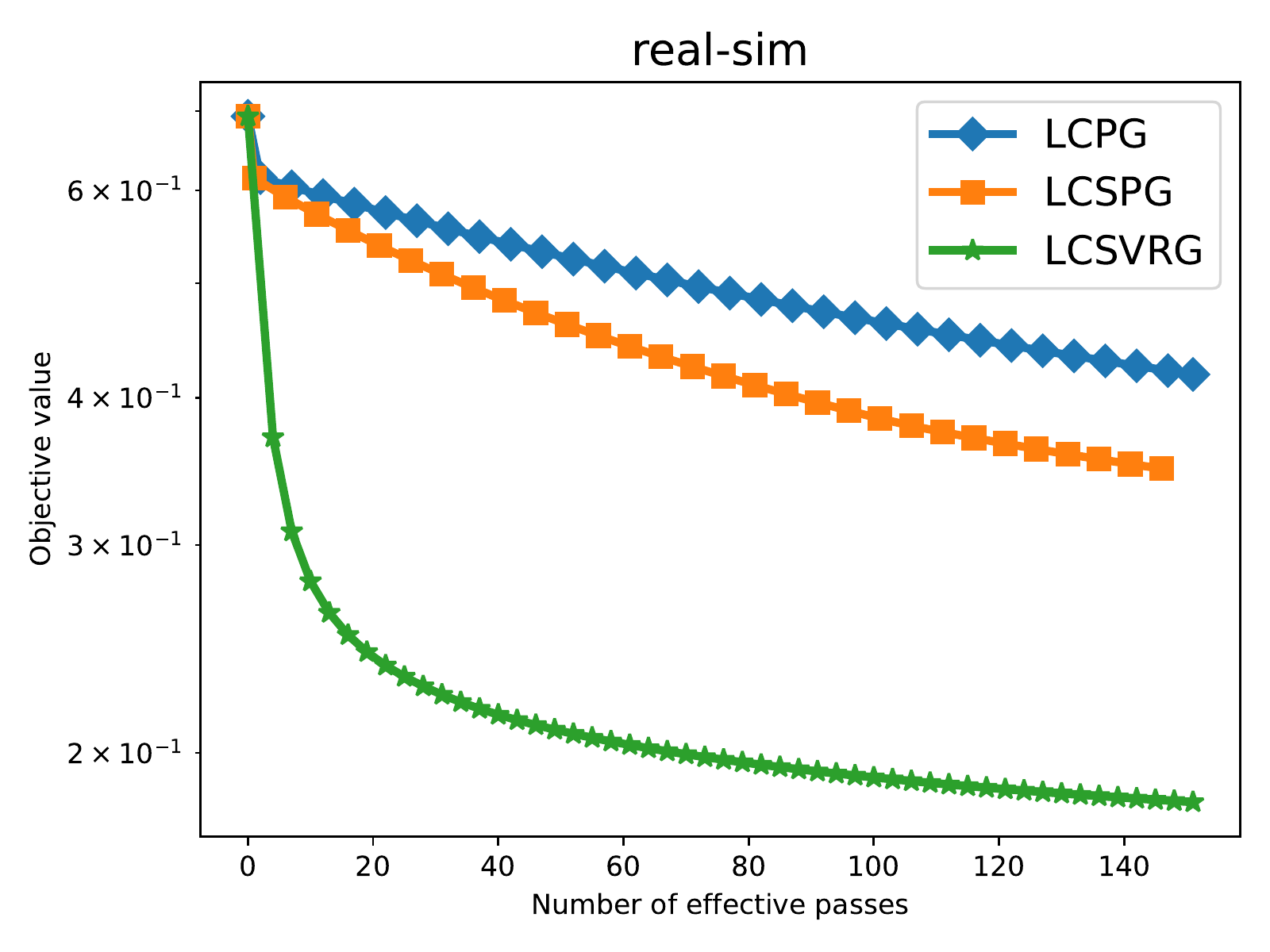}
    \includegraphics[width=0.45\textwidth]{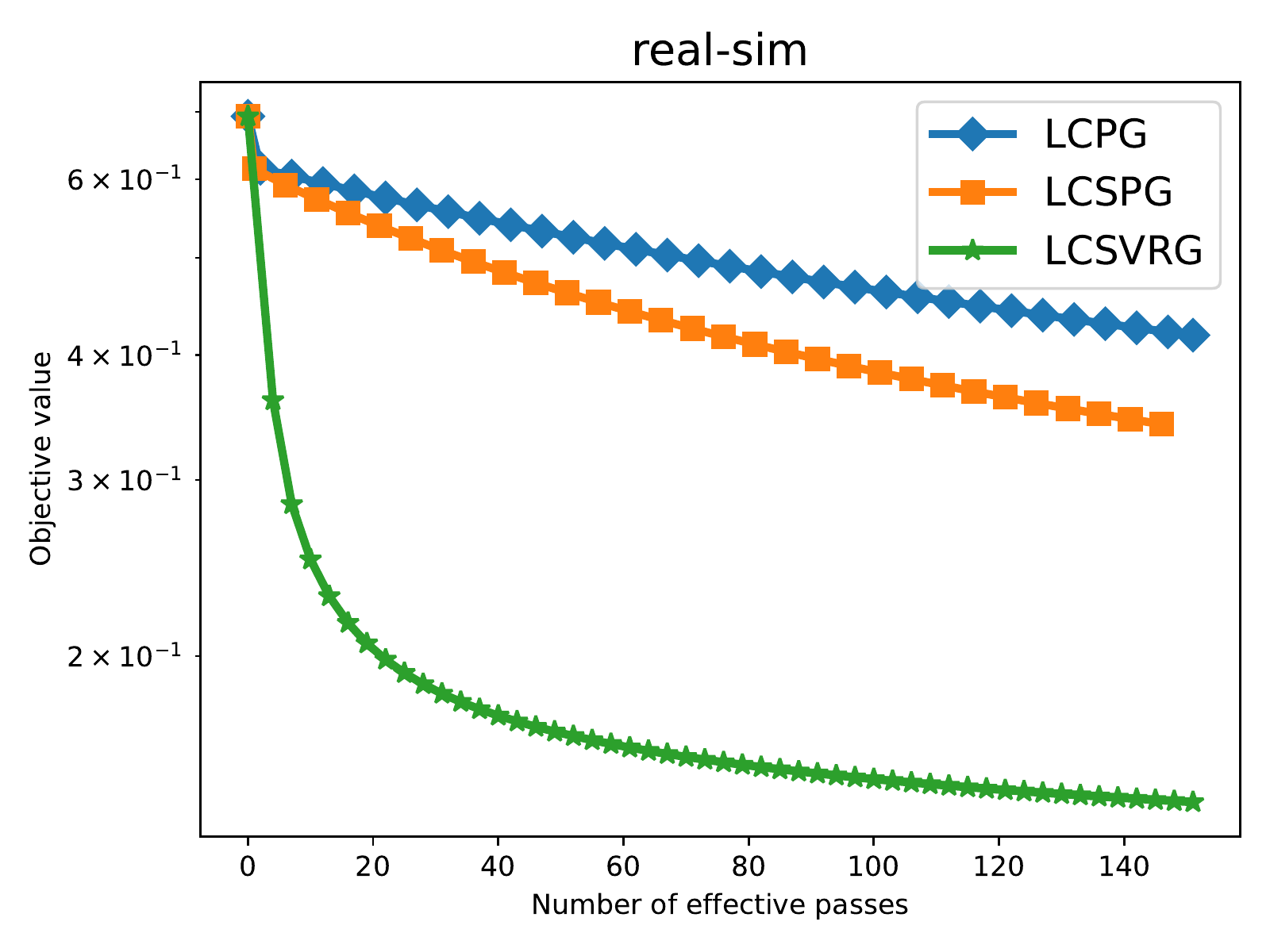}
    \caption{Comparison of \GD{}, \SGD{} and \SVRG{}. The first row reports the results on covtype~(left: $\sigma=0.4$; right: $\sigma=0.6$). 
    The second row reports the results on real-sim~(left: $\sigma=0.1$; right: $\sigma=0.2$).}
    \label{fig:stochastic}
\end{figure}

In the next experiment,  our primary goal is to examine the main theoretical results, namely, the gradient complexities  of \GD{}, its stochastic variants \SGD{} and \SVRG{}.  We apply all these algorithms to a sparsity-induced finite-sum problem, wherein a nonconvex constraint is incorporated into the supervised learning framework to actively enforce a sparse solution. The optimization problem is as follows
\begin{equation} \label{prob:spars}
	\begin{aligned}
		\min_{x\in\Rbb^d} &\quad  \psi_0(x)\coloneq \frac{1}{n}\sum_{i=1}^n f_i(x) \\
		\st & \quad  \psi_1(x)= \beta\norm{x}_1 - g(x) \le \eta_1,
	\end{aligned}
\end{equation}
where  $f_i(x)$ is a smooth loss function associated with the $i$-th sample,  $\psi_1(x)$ is the difference between $\ell_1$-penalty and  a convex smooth function $g(x)$. Employing a difference-of-convex constraint is seen as a tighter relaxation of the cardinality constraint $\norm{x}_0\le \kappa$ than the $\ell_1$ relaxation. The appealing properties of  difference-of-convex penalties have been demonstrated in various studies~\cite{fan2001variable,zhang2010nearly,gotoh2018dc,gong2013general,Boob2020feasible,zhang2012general}.

In view of the concave structure of $-g(x)$, there is a strong   asymmetry between the lower and upper curvature of $-g(x)$, namely, the following  
\begin{equation}\label{eq:g-concave}
-\frac{L_g}{2}\norm{y-x}^2 -\nabla g(y)^T(x-y)-g(y)\le  -g(x)\le -g(y) -\nabla g(y)^T(x-y)
\end{equation}
holds for certain $L_g>0$. Note that this is much stronger than the $L_g$ smoothness condition which adds an extra $\frac{L_g}{2}\norm{y-x}^2$ on the right-hand side of \eqref{eq:g-concave}. 
Due to this feature, one can impose a tighter piece-wise linear surrogate function constraint 
\[\beta\norm{x}_1 - g(x^k)-\nabla g(x^k)(x-x^k) \le \eta^k_1\] in the \GD{}~subproblem.  It should be noted that  our analysis is readily adaptable to accommodate this scenario since it is the smoothness,  as opposed to concavity/convexity, that plays a central role in our convergence analysis and that remains valid. An empirical advantage of this approach is that we now have a tractable subproblem solvable in nearly linear time. See more discussion in \cite{Boob2020feasible}. 

Our experiment considers the task of binary classification with logistic loss, denoted by $f_i(x)=\log(1+\exp(-b_i(a_i^Tx))$, where $a_i\in\Rbb^d,b_i\in\{1,-1\}$, $1\le i \le n$. We use the SCAD penalty $g(x)=\sum_{j=1}^d h_{\beta,\theta}(x_j) $ where $h_{\beta,\theta}(\cdot)$ is defined in \eqref{eq:scad}.
We use the real-sim dataset from the LibSVM repository~\cite{Chang2011libsvm} and the covtype data from the UCI repository~\cite{Kelly2019}. For the latter, we formulate a binary classification task by distinguishing  class ``3'' from the other classes. We set $\beta=2$, $\theta=5$, and $\eta_1=\sigma d$, with $\sigma\in\{0.4, 0.6\}$ for covtype and $\sigma \in\{0.1, 0.2\}$ for real-sim dataset. For each algorithm, we use its theoretically suggested batch size and stepsize. for a fair comparison, we count $n$ evaluations of the stochastic gradient as an effective pass over the dataset and plot the objective value over the number of effective passes.  Figure~\ref{fig:stochastic} plots the convergence result of the compared algorithms. It can be readily seen that \SGD{} performs better than \GD{} in terms of the number of gradient samples, and \SVRG{} achieves the best performance among the three algorithms. The empirical findings further confirm our theoretical complexity analysis.

\section{Conclusion}
In this work, we presented a new {\GD} method for nonconvex function constrained optimization which can achieve gradient complexity of the same order as that of unconstrained nonconvex problems. The key ingredient in our algorithm design is the use of constraint levels to ensure the subproblem feasibility,  which allows us to overcome a well-known difficulty in bounding the Lagrange multipliers in the presence of nonsmooth constraints. Moreover, a merit of our convergence analysis is its striking similarity with that of gradient descent methods for unconstrained problems. Therefore, we can easily  extend our method to minimizing  stochastic, finite-sum and structured nonsmooth functions with nonconvex function constraints; many of the complexity results were not known before. 
Another important feature of our work is that the method can deal with complex scenarios where the subproblems are not exactly solvable. To the best of our knowledge, existing work on sequential convex optimization (\SQP, \MBA) only assume the subproblems to be exactly solved.
 We provided a detailed complexity analysis of {\GD} when the subproblems are inexactly solved by customized interior point method and first order method. 
Finally, we clearly distinguished the notion of gradient complexity from that of computational complexity. In terms of gradient complexity, all of our proposed methods are state of the art and  easy to implement. Whether the computational complexity can be further improved for composite case remains an open problem which we leave as a future direction.

\bibliographystyle{abbrv}
{\footnotesize\bibliography{main}}

\begin{thebibliography}{10}

\bibitem{Auslender2010moving}
A.~Auslender, R.~Shefi, and M.~Teboulle.
\newblock A moving balls approximation method for a class of smooth constrained
  minimization problems.
\newblock {\em SIAM Journal on Optimization}, 20(6):3232--3259, 2010.

\bibitem{bertsekasnonlinear}
D.~P. Bertsekas.
\newblock {\em Nonlinear programming}.
\newblock Athena Scientific, 1999.

\bibitem{Bolte2016majorization}
J.~Bolte and E.~Pauwels.
\newblock Majorization-minimization procedures and convergence of sqp methods
  for semi-algebraic and tame programs.
\newblock {\em Mathematics of Operations Research}, 41(2):442--465, 2016.

\bibitem{boob2019proximal}
D.~Boob, Q.~Deng, and G.~Lan.
\newblock Stochastic first-order methods for convex and nonconvex functional
  constrained optimization.
\newblock {\em Mathematical Programming}, pages 1--65, 2022.

\bibitem{Boob2020feasible}
D.~Boob, Q.~Deng, G.~Lan, and Y.~Wang.
\newblock A feasible level proximal point method for nonconvex sparse
  constrained optimization.
\newblock In {\em Advances in Neural Information Processing Systems},
  volume~33, pages 16773--16784. Curran Associates, Inc., 2020.

\bibitem{burke1989sequential}
J.~Burke.
\newblock A sequential quadratic programming method for potentially infeasible
  mathematical programs.
\newblock {\em Journal of Mathematical Analysis and Applications},
  139(2):319--351, 1989.

\bibitem{burke1989robust}
J.~V. Burke and S.-P. Han.
\newblock A robust sequential quadratic programming method.
\newblock {\em Mathematical Programming}, 43(1):277--303, 1989.

\bibitem{CartisGouldToint11-1}
C.~Cartis, N.~I. Gould, and P.~L. Toint.
\newblock On the evaluation complexity of composite function minimization with
  applications to nonconvex nonlinear programming.
\newblock {\em {SIAM} Journal on Optimization}, 4:1721--1739, 2011.

\bibitem{cartis2014on}
C.~{Cartis}, N.~I. {Gould}, and P.~L. {Toint}.
\newblock On the complexity of finding first-order critical points in
  constrained nonlinear optimization.
\newblock {\em Mathematical Programming}, 144(1):93--106, 2014.

\bibitem{Chang2011libsvm}
C.-C. Chang and C.-J. Lin.
\newblock Libsvm: A library for support vector machines.
\newblock {\em ACM Transactions on Intelligent Systems and Technology},
  2:27:1--27:27, 2011.

\bibitem{diamond2016cvxpy}
S.~Diamond and S.~Boyd.
\newblock Cvxpy: A python-embedded modeling language for convex optimization.
\newblock {\em The Journal of Machine Learning Research}, 17(1):2909--2913,
  2016.

\bibitem{Facchinei19}
F.~Facchinei, V.~Kungurtsev, L.~Lampariello, and G.~Scutari.
\newblock Ghost penalties in nonconvex constrained optimization: Diminishing
  stepsizes and iteration complexity.
\newblock {\em Mathematics of Operations Research}, 46(2):595--627, 2021.

\bibitem{facchinei2017feasible}
F.~Facchinei, L.~Lampariello, and G.~Scutari.
\newblock Feasible methods for nonconvex nonsmooth problems with applications
  in green communications.
\newblock {\em Mathematical Programming}, 164(1):55--90, 2017.

\bibitem{fan2001variable}
J.~{Fan} and R.~{Li}.
\newblock Variable selection via nonconcave penalized likelihood and its oracle
  properties.
\newblock {\em Journal of the American Statistical Association},
  96(456):1348--1360, 2001.

\bibitem{saeed-lan-nonconvex-2013}
S.~Ghadimi and G.~Lan.
\newblock Stochastic first-and zeroth-order methods for nonconvex stochastic
  programming.
\newblock {\em SIAM Journal on Optimization}, 23(4):2341--2368, 2013.

\bibitem{gong2013general}
P.~Gong, C.~Zhang, Z.~Lu, J.~Huang, and J.~Ye.
\newblock A general iterative shrinkage and thresholding algorithm for
  non-convex regularized optimization problems.
\newblock In {\em international conference on machine learning}, pages 37--45.
  PMLR, 2013.

\bibitem{gotoh2018dc}
J.-y. Gotoh, A.~Takeda, and K.~Tono.
\newblock Dc formulations and algorithms for sparse optimization problems.
\newblock {\em Mathematical Programming}, 169:141--176, 2018.

\bibitem{Reddi2016proximal}
S.~J.~Reddi, S.~Sra, B.~Poczos, and A.~J. Smola.
\newblock Proximal stochastic methods for nonsmooth nonconvex finite-sum
  optimization.
\newblock In {\em Advances in Neural Information Processing Systems},
  volume~29, 2016.

\bibitem{Kelly2019}
M.~Kelly, R.~Longjohn, and K.~Nottingham.
\newblock The {UCI} machine learning repository, 2019.

\bibitem{lan2020first}
G.~Lan.
\newblock {\em First-order and Stochastic Optimization Methods for Machine
  Learning}.
\newblock Springer, 2020.

\bibitem{Lan19}
G.~Lan.
\newblock {\em First-order and Stochastic Optimization Methods for Machine
  Learning}.
\newblock Springer-Nature, 2020.

\bibitem{LanMon13-1}
G.~Lan and R.~D.~C. Monteiro.
\newblock Iteration-complexity of first-order penalty methods for convex
  programming.
\newblock {\em Mathematical Programming}, 138:115--139, 2013.

\bibitem{LanMon16-1}
G.~Lan and R.~D.~C. Monteiro.
\newblock Iteration-complexity of first-order augmented lagrangian methods for
  convex programming.
\newblock {\em Mathematical Programming}, 155(1-2):511--547, 2016.

\bibitem{li2021augmented}
Z.~Li and Y.~Xu.
\newblock Augmented lagrangian--based first-order methods for
  convex-constrained programs with weakly convex objective.
\newblock {\em INFORMS Journal on Optimization}, 3(4):373--397, 2021.

\bibitem{lin2019inexact}
Q.~Lin, R.~Ma, and Y.~Xu.
\newblock Inexact proximal-point penalty methods for constrained non-convex
  optimization.
\newblock {\em arXiv preprint arXiv:1908.11518}, 2019.

\bibitem{ma2020quadratically}
R.~Ma, Q.~Lin, and T.~Yang.
\newblock Quadratically regularized subgradient methods for weakly convex
  optimization with weakly convex constraints.
\newblock In {\em International Conference on Machine Learning}, pages
  6554--6564. PMLR, 2020.

\bibitem{manga67}
O.~Mangasarian and S.~Fromovitz.
\newblock The fritz john necessary optimality conditions in the presence of
  equality and inequality constraints.
\newblock {\em Journal of Mathematical Analysis and Applications}, 17:37--47,
  1967.

\bibitem{nemirovski2004interior}
A.~Nemirovski.
\newblock Interior point polynomial time methods in convex programming.
\newblock {\em Lecture notes}, 2004.

\bibitem{Nesterov2005smooth}
Y.~Nesterov.
\newblock Smooth minimization of non-smooth functions.
\newblock {\em Mathematical Programming}, 103(1):127--152, 2005.

\bibitem{robbins1985a}
H.~{Robbins} and D.~{Siegmund}.
\newblock A convergence theorem for non negative almost supermartingales and
  some applications.
\newblock {\em Optimizing Methods in Statistics}, pages 111--135, 1971.

\bibitem{Schmidt2011convergence}
M.~Schmidt, N.~L. Roux, and F.~Bach.
\newblock Convergence rates of inexact proximal-gradient methods for convex
  optimization.
\newblock {\em arXiv preprint arXiv:1109.2415}, 2011.

\bibitem{shen2016disciplined}
X.~Shen, S.~Diamond, Y.~Gu, and S.~Boyd.
\newblock Disciplined convex-concave programming.
\newblock In {\em 2016 IEEE 55th Conference on Decision and Control (CDC)},
  pages 1009--1014. IEEE, 2016.

\bibitem{WangMaYuan17-1}
X.~Wang, S.~Ma, and Y.~Yuan.
\newblock Penalty methods with stochastic approximation for stochastic
  nonlinear programming.
\newblock {\em Mathematics of Computation}, 86 (306):1793--1820, 2017.

\bibitem{Xu2019-1}
Y.~Xu.
\newblock Iteration complexity of inexact augmented lagrangian methods for
  constrained convex programming.
\newblock {\em Mathematical Programming}, 2019.

\bibitem{Zhaosong2021}
P.~Yu, T.~K. Pong, and Z.~Lu.
\newblock Convergence rate analysis of a sequential convex programming method
  with line search for a class of constrained difference-of-convex optimization
  problems.
\newblock {\em SIAM Journal on Optimization}, 31(3):2024--2054, 2021.

\bibitem{zhang2010nearly}
C.-H. ZHANG.
\newblock Nearly unbiased variable selection under minimax concave penalty.
\newblock {\em Annals of statistics}, 38(2):894--942, 2010.

\bibitem{zhang2012general}
C.-H. Zhang and T.~Zhang.
\newblock A general theory of concave regularization for high-dimensional
  sparse estimation problems.
\newblock {\em Statistical science}, 27(4):576--593, 2012.

\end{thebibliography}

\appendix

\section{Proof of Lemma \ref{lem:tight_lipschitz-const}}\label{appx:tight_Lipschitz_const}

	Let $\{\delta_n\}_{n \ge 1}$ be a function of $\Ccal^{\infty}$-smooth, real-valued mollifier functions over $\Rbb^d$ where, for every $n \ge 1$, we have: (i) $\delta_n \ge 0$, (ii) $\int \delta_n(\tau)d\tau = 1$, and (iii) and $\delta_n(\tau) = 0$ for $\tau$ satisfying $\gnorm{\tau}{}{} \ge \tfrac{1}{n}$. Moreover, we define $p_n = \delta_n * p$. It now follows that
	\begin{align*}
		p_n(x) - p_n(y) - \inprod{\grad p_n(y) }{x-y} &= \delta_n *[p(x) - p(y)] - \inprod{\delta_n * \grad p(y)}{x-y} \\
		&= \textstyle{\int}_{\tau} \delta_n(\tau) [p(x-\tau) - p(y-\tau)]d\tau - \inprod{\int_{\tau} \delta_n(\tau) \grad p(y-\tau)d\tau}{x-y}\\
		&= \textstyle{\int}_{\tau} \delta_n(\tau) [p(x-\tau) - p(y-\tau) - \inprod{\grad p(y-\tau)}{x-y}] d\tau
	\end{align*}
	Using  \eqref{eq:lower-upper-curvature} along with the above relation, noting that $\delta_n \ge 0$ and the fact $\int \delta_n(\tau)d\tau = 1$, we have
	\begin{equation}\label{eq:upp-low-curv-p_n}
		-\tfrac{\mu}{2} \gnorm{x-y}{}{2} \le p_n(x) - p_n(y) - \inprod{\grad p_n(y) }{x-y} \le \tfrac{L}{2}\gnorm{x-y}{}{2},
	\end{equation}
	for all $x,y$. Note that $p_n$ is $\Ccal^{\infty}$-smooth. Hence, using Taylor's theorem, there exist $\xi \in [x,y]$ such that 
	$p_n(x) - p_n(y) - \inprod{\grad p_n(y) }{x-y} = \tfrac{1}{2} \inprod{x-y}{\grad^2 p_n(\xi)(x-y)} + o(\gnorm{x-y}{}{2})$. Using this relation along with~\eqref{eq:upp-low-curv-p_n} and denoting $v := y-x$, we have
	\begin{equation}
		-\tfrac{\mu}{2} \gnorm{v}{}{2} \le \tfrac{1}{2} \inprod{v}{\grad^2 p_n(\xi)v} + o(\gnorm{v}{}{2}) \le \tfrac{L}{2}\gnorm{v}{}{2}
	\end{equation} 
	Now, diving both sides of the above relation by $\gnorm{v}{}{2}$, taking $y \to x$ which implies $\xi \to x$ and $\grad^2 p_n(\xi) \to \grad^2 p_n(x)$, we have
	\[-\mu \le \inprod{\tfrac{v}{\gnorm{v}{}{} }}{\grad^2 p_n(x) \tfrac{v}{\gnorm{v}{}{}} } \le L,\]
	for all $v$ and $x$. The above relation is equivalent to the fact that $\gnorm{\grad^2 p_n(x)}{}{} \le \max\{L, \mu\}$ for all $x$. From here, for any $x, y$, we have
	\begin{align*}
		\gnorm{\grad p_n(x) - \grad  p_n(y)}{}{} &= \gnorm{\textstyle\int_{t = 0}^1 \grad^2 p(y + t(x-y))(x-y)dt }{}{}\\
		&\le \textstyle\int_{t = 0}^1 \gnorm{\grad^2 p(y + t(x-y))(x-y)}{}{}dt\\
		&\le \textstyle\int_{t = 0}^1 \gnorm{\grad^2 p(y + t(x-y))}{}{}\gnorm{x-y}{}{}dt\\
		&\le \max\{L, \mu\} \gnorm{x-y}{}{}.
	\end{align*}
	Now, taking $n \to \infty$ and noting that $\grad p_n(x) \to \grad p(x)$ for all $x$, we have \eqref{eq:tight-Lipschitz-const}. Hence, we conclude the proof.

\section{Proof of Theorem~\ref{thm:lcgd-inexact-asymp}}\label{appx:LCGD_inexact-nonconvex}
First of all, using the definition of $\epsilon_k$ and the fact that $\psi_i^k(x^{k+1}) \le \eta^k _i + \epsilon_k$ for all $i \in [m]$, we have %
\[
\psi_i^{k}(x^{k}) = \psi_{i}(x^{k}) \le \psi_{i}^{k-1}(x^{k}) \le \eta_i^{k-1} + \epsilon_{k-1}< \eta_i^{k-1}+\delta_i^{k-1} = \eta_i^{k},\quad i=1,2,\ldots, m. 
\]
Hence $x^{k}$ is strictly feasible solution of \eqref{subproblem}. Due to Slater condition, there exists a pair of optimal primal and dual solutions, which we denote by
 $\tilde{x}^{k+1}$ and ${\lambda}^{k+1}$. We first prove the following lemma.
 Our argument will be based on the following key result.
 \begin{lemma}\label{lem:inexact-seq-conv}
 	 \begin{equation} 
 	\lim_{k\raw\infty}\gnorm{\tilde{x}^{k+1}-x^k}{}{}=0,\quad 
 	\lim_{k\raw\infty}\gnorm{{x}^{k+1}-x^k}{}{}=0 \label{eq:asymp-xk}
 \end{equation}
 \end{lemma}
\begin{proof}
	Using optimality of $\wtil{x}^{k+1}$, strong convexity of $\psi_{0}^k$, feasibility of $x_k$ for the subproblem \eqref{subproblem} and $\psi_{0}^k(x^{k+1}) \le \psi_{0}^k(\wtil{x}^{k+1}) + \epsilon_k$, we have, %
	\begin{align*}
		\psi_0(x^{k+1}) &\le \psi_0^k(x^{k+1}) \\
		&\le \psi_{0}^k(\wtil{x}^{k+1}) +\epsilon_k \\
		&\le \psi_{0}^k(x^k) -   \tfrac{L_0}{2}\gnorm{x^k - \wtil{x}^{k+1}}{}{2} + \epsilon_k \\
		&= \psi_{0}(x^k) -   \tfrac{L_0}{2}\gnorm{x^k - \wtil{x}^{k+1}}{}{2}  + \epsilon_k.
	\end{align*}
	Since $\epsilon_k$ is summable, we have that $\gnorm{x^k-\wtil{x}^{k+1}}{}{2}$ is summable. This implies $\gnorm{x^k-\wtil{x}^{k+1}}{}{} \to 0$. Since $\epsilon_k \to 0$ and $\wtil{x}^{k+1}$ is a unique optimal solution of \eqref{subproblem}, we have $\gnorm{x^{k+1} - \wtil{x}^{k+1}}{}{} \to 0$. Then, using Cauchy-Schwarz inequality,  we have $\gnorm{x^{k+1} -x^k}{}{} \le \gnorm{x^{k+1} - \wtil{x}^{k+1}}{}{} + \gnorm{\wtil{x}^{k+1}-x^k}{}{}$ and hence, $\gnorm{x^{k+1} -x^k}{}{} \to 0$. Hence, we conclude the proof.
	\end{proof}
	 Now, we show boundedness of $\tilde{\lambda}^{k+1}$. Assume, for the sake of contradiction, that $\tilde{\lambda}^{k+1}$ is unbounded. Let $\bar{x}$ be a limit point of $\{x^k\}$. Passing to a subsequence if necessary, we have $x^{k} \to \bar{x}$. Using Lemma \ref{lem:inexact-seq-conv}, we have $\wtil{x}^{k+1} \to \bar{x}$ and $x^{k+1} \to \bar{x}$. %
	 Then, we have
	 \begin{align*}
	 	\psi_{0}^{k}(x^{k+1}) + \inprod{\tilde{\lambda}^{k+1}}{\psi^{k}(x^{k+1})} &\le \psi_{0}^{k}(\wtil{x}^{k+1}) + \inprod{\tilde{\lambda}^{k+1}}{\psi^{k}(\wtil{x}^{k+1})} + \epsilon_{k}\\ 
	 	&\le \psi_{0}^{k}(x) + \inprod{\tilde{\lambda}^{k+1}}{\psi^{k}(x)} + \epsilon_{k} %
	 \end{align*}
 	Note that the above relation is comparable to \eqref{eq:lag-opt-x_k} up to an error term of $\epsilon_k$. Following the arguments in the proof of Theorem \ref{prop:lcgd:bound-dual} (Part 1, \eqref{eq:lag-opt-x_k} onwards) and noting that $\epsilon_k$ is summable, we conclude that $\{\tilde{\lambda}^{k+1}\}$ is bounded.

 	Now, we prove limit point of $\{x^k\}$ is a KKT point. Since $\Lcal_{k}(x^{k+1}, \tilde{\lambda}^{k+1}) \le \Lcal_{k}(\wtil{x}^{k+1}, \tilde{\lambda}^{k+1}) + \epsilon_k$, we rewrite \eqref{eq:middle-06} as 
 	\begin{align}
 		& {}\big\langle\nabla f_{0}(x^{k})+\tsum_{i=1}^{m}\tilde{\lambda}_i^{k+1}\nabla f_{i}(x^{k}),x^{k+1}-x\big\rangle+\chi_{0}(x^{k+1}) - \chi_{0}(x) +\langle\tilde{\lambda}^{k+1},\chi(x^{k+1}) - \chi(x)\big\rangle \nonumber \\
 		\le & {}\big\langle\nabla f_{0}(x^{k})+\tsum_{i=1}^{m}\tilde{\lambda}_i^{k+1}\nabla f_{i}(x^{k}),\wtil{x}^{k+1}-x\big\rangle+\chi_{0}(\wtil{x}^{k+1})-\chi_{0}(x)+\langle\tilde{\lambda}^{k+1},\chi(\wtil{x}^{k+1})-\chi(x)\big\rangle + \epsilon_k\nonumber \\
 		\le & {}\tfrac{L_{0}+\langle\tilde{\lambda}^{k+1},L\rangle}{2}\big[\gnorm{x-x^{k}}{}{2}-\gnorm{\wtil{x}^{k+1}-x^{k}}{}{2}-\gnorm{x-\wtil{x}^{k+1}}{}{2}\big] + \epsilon_k.\label{eq:middle-08}
 	\end{align}
 	Let $\bar{x}$ be a limit point of sequence $\{x^k\}$. Since $\tilde{\lambda}^{k+1}$ is bounded, we assume limit point $\bar{\lambda}$. Without loss of generality, we have $x^k \to \bar{x}$ and $\tilde{\lambda}^{k+1} \to \bar{\lambda}$.  Then, in view of Lemma~\ref{lem:inexact-seq-conv}, we have $\lim_{k\raw\infty}x^{k+1} \to \bar{x}$ and $\lim_{k\raw\infty} \wtil{x}^{k+1} \to \bar{x}$. Taking limit $k \to \infty$ in \eqref{eq:middle-08}, we have
 	\begin{equation*}
 		{}\big\langle\nabla f_{0}(\bar{x})+\tsum_{i=1}^{m}\bar{\lambda}_i\nabla f_{i}(\bar{x}),\bar{x}-x\big\rangle+\chi_{0}(\bar{x}) - \chi_{0}(x) +\langle\bar{\lambda},\chi(\bar{x}) - \chi(x)\big\rangle  \le 0.
 	\end{equation*}
 	Note that the above equation matches with \eqref{eq:opt_interemediate_KKT} exactly. 
 	From here, we follow the proof of Theorem~\ref{prop:lcgd:bound-dual} (Part 2, \eqref{eq:opt_interemediate_KKT} onwards) to conclude first-order stationarity of $(\bar{x}, \bar{\lambda})$. A similar argument can show complimentary slackness. Hence, we have that  $(\bar{x}, \bar{\lambda})$ is KKT-solution and we conclude the proof.

\end{document}